\documentclass[leqno,11pt]{amsart}
\usepackage{amssymb, amsmath}
\usepackage{amsthm, amsfonts,mathrsfs}
\usepackage{color}
\def\inte#1{
\displaystyle\mathop{#1\kern0pt}^\circ }

\let\pa=\partial

\let\d=\delta

\let\r=\rho

\let\f=\frac

\let\p=\psi
\let\om=\omega

\let\Om=\Omega

\def\pa{\partial}
\def\grad{\nabla}

\def\virgp{\raise 2pt\hbox{,}}
\def\cdotpv{\raise 2pt\hbox{;}}

\def\eqdefa{\buildrel\hbox{\footnotesize def}\over =}

\def\C{\mathop{\mathbb C\kern 0pt}\nolimits}
\def\DD{\mathop{\mathbb D\kern 0pt}\nolimits}
\def\EE{\mathop{{\mathbb E \kern 0pt}}\nolimits}
\def\K{\mathop{\mathbb K\kern 0pt}\nolimits}
\def\N{\mathop{\mathbb N\kern 0pt}\nolimits}
\def\Q{\mathop{\mathbb Q\kern 0pt}\nolimits}
\def\R{\mathop{\mathbb R\kern 0pt}\nolimits}
\def\SS{\mathop{\mathbb S\kern 0pt}\nolimits}
\def\ZZ{\mathop{\mathbb Z\kern 0pt}\nolimits}
\def\TT{\mathop{\mathbb T\kern 0pt}\nolimits}
\def\P{\mathop{\mathbb P\kern 0pt}\nolimits}

\def\dv{\mbox{div}}

\def\curl{\mathop{\rm curl}\nolimits}

\def\na{\nabla}
\def\p{\partial}

\newcommand{\beq}{\begin{equation}}
\newcommand{\eeq}{\end{equation}}
\newcommand{\ben}{\begin{eqnarray}}
\newcommand{\een}{\end{eqnarray}}
\newcommand{\beno}{\begin{eqnarray*}}
\newcommand{\eeno}{\end{eqnarray*}}
\newcommand{\andf}{\quad\hbox{and}\quad}
\newcommand{\with}{\quad\hbox{with}\quad}

\newtheorem{thm}{Theorem}[section]
\newtheorem{lem}{Lemma}[section]
\newtheorem{rmk}{Remark}[section]

\newtheorem{prop}{Proposition}[section]
\renewcommand{\theequation}{\thesection.\arabic{equation}}
\begin{document}

\title[GWP of $2-$D inhomogeneous NS]
{Global well-posedness of 2-D incompressible anisitropic Navier-Stokes equations with variable density}

\author[H.  Abidi]{Hammadi Abidi}
\address[H.  Abidi]{D\'epartement de Math\'ematiques
Facult\'e des Sciences de Tunis
Universit\'e de Tunis EI Manar
2092
Tunis
Tunisia}\email{hammadi.abidi@fst.utm.tn}
\author[G. Gui]{Guilong Gui}
\address[G. Gui]{School of Mathematics and Computational Science, Xiangtan University,  Xiangtan 411105,  China}\email{glgui@amss.ac.cn}
\author[P. Zhang]{Ping Zhang}
\address[P. Zhang]{State Key Laboratory of Mathematical Sciences, Academy of Mathematics $\&$ Systems Science, The Chinese Academy of
	Sciences, Beijing 100190, China, and School of Mathematical Sciences,
	University of Chinese Academy of Sciences, Beijing 100049, China.} \email{zp@amss.ac.cn}

\maketitle
\begin{abstract}

We establish the global well-posedness for two-dimensional inhomogeneous, incompressible, anisotropic Navier-Stokes systems. Two specific models are analyzed: one with partial dissipation (referred to as (AINS)) and one with only horizontal dissipation (referred to as (HINS)), under the assumption that the initial density is bounded away from zero and infinity. For the (AINS) system posed in the whole plane $\mathbb{R}^2$, we prove the existence and uniqueness of global solutions for finite-energy initial data, employing time-weighted energy estimates and a duality argument. For the (HINS) system on the domain $\mathbb{T} \times \mathbb{R}$, global well-posedness is established for sufficiently small initial velocity and sufficiently small density variation. By exploiting the anisotropic dissipation structure, employing Poincar\'{e}-type inequalities to obtain exponential decay for the oscillatory part of the velocity field, and carefully balancing the growth of the density gradient, we overcome the principal analytical challenges.
\date{}

\end{abstract}


\vskip 0.2cm


\renewcommand{\theequation}{\thesection.\arabic{equation}}
\setcounter{equation}{0}
\section{Introduction}

Describing the motion of a mixture of immiscible, incompressible viscous fluids with distinct densities requires modeling the coupled evolution of the density and velocity fields. This dynamics is governed by the following $d$-dimensional inhomogeneous incompressible Navier-Stokes equations \cite{LP1996} with $d = 2$ or $3$:
\begin{equation}\label{eqns-lions-ins-1}
\mbox{(INS)}\quad
\begin{cases}
\partial_t \rho + \nabla \cdot (\rho u) = 0, & (t,x) \in \mathbb{R}^+ \times \mathbb{R}^d, \\[4pt]
\rho\bigl(\partial_t u + u \cdot \nabla u\bigr) - \nabla \cdot \bigl(2\mu(\rho)\,\mathbb{D}(u)\bigr) + \nabla \Pi = 0, \\[4pt]
\nabla \cdot u = 0, \\[4pt]
(\rho, u)|_{t=0} = (\rho_0, u_0),
\end{cases}
\end{equation}
where $\rho$ and $u$ denote the density and velocity of the fluid, respectively; $\Pi$ stands for the scalar pressure; $\mathbb{D}(u) = (d_{ij})_{d \times d}$, with $d_{ij} \stackrel{\text{def}}{=} \frac{1}{2}(\partial_i u^j + \partial_j u^i)$, designates the deformation tensor; and $\mu(\rho) > 0$ is the kinematic viscosity of the fluid.

In particular, when the viscosity $\mu(\rho)$ is a positive constant $\mu$, the viscous term in \eqref{eqns-lions-ins-1} reduces to the Laplacian $\mu\Delta u$. Alternatively, if one considers an anisotropic viscous term of the form $\nabla\cdot (\mathbb{C} :\mathbb{D}(u))$ with a fourth-rank viscosity tensor $\mathbb{C}$, a different class of anisotropic systems arises.

The purpose of this paper is to investigate the global well-posedness of the following two-dimensional inhomogeneous incompressible anisotropic Navier-Stokes system:

\begin{itemize}
\item The 2D inhomogeneous incompressible Navier–Stokes equations with partial dissipation
\begin{equation}\label{eqns-ains-1}
{\mbox{(AINS)}}\quad \begin{cases}
&\partial_t\rho+u\cdot\nabla\rho=0, \qquad\qquad\qquad\qquad\quad\ \  (t, x) \in \mathbb{R}^+\times  \mathbb{R}^2,\\
&\rho(\partial_t u+ u\cdot\nabla u)-\left(\begin{array}{cc}\partial_2^2u^1 \\ \partial^2_1u^2\end{array}\right)+\nabla\Pi =0,\\
&{\mathop{\rm div}}\,u=0, \\
&(\rho,u)|_{t=0}=(\rho_0,u_0).
	\end{cases}
\end{equation}

\item
The 2D inhomogeneous incompressible Navier--Stokes equations with only horizontal dissipation (HINS):
\begin{equation}\label{eqns-HINS-1}
{\mbox{(HINS)}}\quad\begin{cases}
\partial_t\rho + u\cdot\nabla \rho = 0, & (t, x_1, x_2) \in \mathbb{R}^+ \times \mathbb{T} \times \mathbb{R},\\[4pt]
\rho\left(\partial_t u + u\cdot\nabla u\right) - \partial_1^2 u + \nabla\Pi = 0,\\[4pt]
\nabla \cdot u = 0,\\[4pt]
(\rho, u)|_{t=0} = (\rho_0, u_0),
\end{cases}
\end{equation}
where $\mathbb{T} $ denotes the one-dimensional torus $\R/\mathbb{Z}.$

\end{itemize}

 Systems of the above  types appear in geophysical fluid dynamics, where
meteorologists often model turbulent diffusion by separating
the horizontal  viscosity from the vertical one. One may check the references \cite{CDGG, Pedlovsky}  for more background of these systems and \cite{ZZ1} and the references therein  for the recent progress on 3-D homogeneous incompressible anisotropic Navier-Stokes system.

We observe that analogous to (INS), the systems \eqref{eqns-ains-1} and \eqref{eqns-HINS-1} possess the following scaling-invariant property: if $(\rho, u, \Pi)$ solves \eqref{eqns-ains-1} (or \eqref{eqns-HINS-1}) with initial data $(\rho_0, u_0)$, then for any $\ell > 0$,
\begin{equation}\label{S1eq1}
(\rho, u, \Pi)_{\ell}(t, x) \stackrel{\text{def}}{=} \bigl(\rho(\ell^2 t, \ell x), \ell u(\ell^2 t, \ell x), \ell^2 \Pi(\ell^2 t, \ell x)\bigr)
\end{equation}
is also a solution of \eqref{eqns-ains-1} (or \eqref{eqns-HINS-1}) with initial data $(\rho_0(\ell x), \ell u_0(\ell x))$. A functional space is called critical if its norm is invariant under the scaling transformation \eqref{S1eq1}.

For the case of variable viscosity, Lions \cite{LP1996} established the global existence of finite-energy weak solutions. However, as noted in the same work, the uniqueness and regularity of such weak solutions remain major open problems, even in two space dimensions. In the two-dimensional setting, under the additional assumptions that $\|\mu(\rho_0) - 1\|_{L^\infty(\mathbb{T}^2)} \leq \varepsilon$ and $u_0 \in H^1(\mathbb{T}^2)$ for sufficiently small $\varepsilon > 0$, Desjardins \cite{desjardins} proved that the global weak solution $(\rho, u, \nabla\Pi)$ constructed in \cite{LP1996} satisfies $u \in L^{\infty}([0, T]; H^1(\mathbb{T}^2))$ for any $T > 0$. The first and third authors of this paper \cite{A-Z} improved the regularity of the solutions in \cite{desjardins} and established uniqueness under an additional regularity assumption on the initial density.

For the system (INS) with constant viscosity, Lady\v{z}enskaja and Solonnikov \cite{LS} first studied  the well-posedness of the equations in a bounded domain $\Omega$ under homogeneous Dirichlet boundary conditions for the velocity $u$. Using energy methods, Kazhikov \cite{KA1974} later established the global existence of weak solutions in the energy space provided that the initial density is bounded from above and away from vacuum. The uniqueness issue in the two-dimensional settings considered therein was addressed in \cite{PZZ13, HSWZ2024} for the models under investigation. Simon \cite{Simon} subsequently removed the strict positivity requirement on $\rho_0$.

Concerning the critical well-posedness in two dimensions, there exists extensive literature \cite{danchin04, A-Z, BHas, AG2021, DW2023, HSWZ2024} contributing to this field. In particular, the first two authors of this paper \cite{AG2021} improved upon previous results from \cite{danchin04, BHas} by establishing global well-posedness under the conditions $u_0 \in \dot{B}^{0}_{2,1}$ and $\rho_0^{-1} - 1 \in \dot{B}^{\varepsilon}_{\frac{2}{\varepsilon}, 1}$ for some $0 < \varepsilon < 1$, with $\rho_0$ bounded between two positive constants $M_1 \leq \rho_0 \leq M_2$. Inspired by a three-dimensional result in \cite{Zhang2020}, Danchin and Wang \cite{DW2023} proved the existence and uniqueness of solutions when the initial density $\rho_0$ is close to a positive constant in $L^\infty$, and the initial velocity satisfies $u_0 \in L^2 \cap \dot{B}^{-1 + \frac{2}{p}}_{p,1}$ for $1 < p < 2$. Most recently, Hao et al. \cite{HSWZ2024} established the uniqueness of global weak solutions for the two-dimensional (INS) system with constant viscosity, whose existence had been obtained earlier by Lions \cite{LP1996}.

On the other hand,
when $\mu(\rho)\equiv 0$ and $d=2$ in (INS), one obtains the following  2-D inhomogeneous incompressible Euler equations (IE):
\begin{equation}\label{eqns-lions-IE-1}\mbox{(IE)}\quad
\begin{cases}
\pa_t \rho + u \cdot \nabla \rho=0,& (t,x)\in\R^+\times\R^2, \\
\rho(\pa_t u +u\cdot\nabla u )+\grad\Pi=0, \\
\dv\, u = 0, \\
(\rho,u)|_{t=0}=(\rho_0,u_0),
\end{cases}
\end{equation}
Danchina and Fanelli \cite{danchin-Fanelli-2011} established the local well-posedness in endpoint Besov spaces and a continuation Beale-Kato-Majda
criterion, and gave lower bounds for the lifespan of a solution to \eqref{eqns-lions-IE-1}. In contrast to the 2-D homogeneous Euler equations, the global well-posedness of the inhomogeneous Euler equations (IE) remains open and highly challenging except the special case in \cite{CWZZ2025}.

The inhomogeneous incompressible anisotropic Navier-Stokes systems \eqref{eqns-ains-1} and \eqref{eqns-HINS-1} can be viewed as a mathematical hybrid, bridging the 2-D inhomogeneous Euler equations (IE) and the 2-D inhomogeneous Navier-Stokes equations (INS) with full (constant) viscosity.

When $\rho$ is a positive constant, say $1$, the system \eqref{eqns-ains-1} (respectively \eqref{eqns-HINS-1}) reduces to the following 2-D homogeneous incompressible Euler-type equations:
\begin{equation}\label{eqns-ans-1}
\mbox{(ANS)}\quad
\begin{cases}
\partial_t u + u \cdot \nabla u - \begin{pmatrix} \partial_2^2 u^1 \\ \partial_1^2 u^2 \end{pmatrix} + \nabla \Pi = 0, &\  (t, x) \in \mathbb{R}^+ \times \mathbb{R}^2,\\[6pt]
\nabla \cdot u = 0, \\[4pt]
u|_{t=0} = u_0,
\end{cases}
\end{equation}
or
\begin{equation}\label{eqns-HNS-1}
\mbox{(HNS)}\quad
\begin{cases}
\partial_t u + u \cdot \nabla u - \partial_1^2 u + \nabla \Pi = 0, &\, (t, x_1, x_2) \in \mathbb{R}^+ \times \mathbb{T} \times \mathbb{R},\\[6pt]
\nabla \cdot u = 0, \\[4pt]
u|_{t=0} = u_0,
\end{cases}
\end{equation}
respectively.

Roughly speaking, based on the Yudovich approach to the 2D Euler equations \cite{Yudovich1963}, the unique global existence of solutions to \eqref{eqns-ans-1} (or \eqref{eqns-HNS-1}) can be established for initial velocity $u_0 \in H^{s}(\Omega)$ with $s > 2$. This result relies crucially on the uniform-in-time boundedness of the vorticity $\omega \eqdefa  \partial_1 u^2 - \partial_2 u^1$. Although the solution remains in $H^{s}(\Omega)$ for all time, its $H^{s}$-norm may grow relatively rapidly over time \cite{DongB2021}.

When considering the global well-posedness of \eqref{eqns-ans-1} (or \eqref{eqns-HNS-1}) in the energy space (or $H^2$), the answer is not straightforward. Based on the observation \eqref{momentum-ains-2} below, we find that the system \eqref{eqns-ans-1} is not genuinely anisotropic and can be solved using the same approach as for the classical 2-D Navier--Stokes equations. For the system \eqref{eqns-HNS-1}, Dong et al. \cite{DongB2021} investigated this problem when the domain $\Omega$ is $\mathbb{T} \times \mathbb{R}$, and established a uniform upper bound for $u$ in $H^2(\Omega)$ by decomposing the velocity $u$ into its horizontal average $\bar{u}$ and the corresponding oscillation $u_{\neq},$ thereby obtaining the global well-posedness of the system. Moreover, they showed that if $\|u_0\|_{H^2}$ is sufficiently small, then $u_{\neq}$ decays exponentially in time in the $H^1$ norm. It should be noted that, in the system \eqref{eqns-HNS-1}, even under the smallness assumption of $\|u_0\|_{H^2}$, the quantity $\int_0^t\|\nabla{u}(\tau)\|_{L^{\infty}}\,d\tau$ may still exhibit algebraic growth in time $t>0$.
When considering the inhomogeneous fluid system \eqref{eqns-HINS-1} (or \eqref{eqns-ains-1}), the propagation of regularity and the control of density variation present substantial mathematical challenges.

The main purpose of this paper is to establish the existence and uniqueness of global solutions to the inhomogeneous anisotropic Navier-Stokes systems \eqref{eqns-ains-1} and \eqref{eqns-HINS-1}.
Throughout the paper, we assume that the initial density satisfies
\begin{equation}\label{t.1}
    0 < M_1 \leq \rho_0(x) \leq M_2 \quad \text{for all } x \in \Omega
\end{equation}
with $\Omega = \mathbb{R}^2$ or $\Omega = \mathbb{T} \times \mathbb{R}$ for some positive constants $M_1$, $M_2$.

\medskip
\noindent\textbf{Main results.}

Our first main result establishes the existence of a unique global solution to \eqref{eqns-ains-1} with initial velocity in the energy space $L^2(\mathbb{R}^2)$.

\begin{thm}\label{thm-GUS-NS}
{\sl Let $\rho_0$ satisfy \eqref{t.1} and let $u_0 \in L^2(\mathbb{R}^2)$ be a solenoidal vector field. Then the system \eqref{eqns-ains-1} admits a unique global solution $(\rho, u, \nabla \Pi)$ such that
\[
\rho \in C_{\mathrm{w}}([0, \infty[; L^\infty), \quad
u \in C([0, \infty[; L^2) \cap L^2(\mathbb{R}^+; \dot{H}^1),
\]
and
\begin{equation}\label{bdd-density-visc-1}
0 < M_1 \leq \rho(t, x) \leq M_2, \quad
\forall\, (t, x) \in \mathbb{R}^+ \times \mathbb{R}^2.
\end{equation}
Moreover, the following estimates hold:
\begin{equation}\label{est-variable-2}
\begin{split}
&\|u\|_{L^\infty(\mathbb{R}^+; L^2)}^2 + \|\nabla u\|_{L^2(\mathbb{R}^+; L^2)}^2 + \bigl\|t^{\frac{1}{2}} \nabla u\bigr\|_{L^\infty(\mathbb{R}^+; L^2)}^2 + \bigl\|t^{\frac{1}{2}} (\nabla^2 u, \nabla \Pi, D_t u)\bigr\|_{L^2(\mathbb{R}^+; L^2)}^2 \\
&\quad + \bigl\|t (\nabla^2 u, \nabla \Pi, D_t u)\bigr\|_{L^\infty(\mathbb{R}^+; L^2)}^2 + \bigl\|t \nabla D_t u\bigr\|_{L^2(\mathbb{R}^+; L^2)}^2
+ \bigl\|t^{\frac{1}{2}} \nabla u\bigr\|_{L^2(\mathbb{R}^+; L^\infty)}^2\\
&\lesssim \|u_0\|_{L^2}^2.
\end{split}
\end{equation}
Here and below, we always denote $D_t\eqdefa \p_t+u\cdot\na$ to be the material derivative.}
\end{thm}

The second main result concerns the global well-posedness of \eqref{eqns-HINS-1}. Toward this, for a function $f = f(x_1, x_2)$ integrable in $x_1$ on $\mathbb{T}$,  we define the horizontal average by
\begin{equation}\label{average-u-1-a}
\bar{f}(x_2) \eqdefa \int_{\mathbb{T}} f(x_1, x_2) \,dx_1,
\end{equation}
and the corresponding oscillation part by
\begin{equation}\label{average-u-1-b}
f_{\neq} = \mathcal{P}_{\neq} f \eqdefa f - \bar{f}.
\end{equation}

\begin{thm}\label{thm-GWP-HINS}
{\sl Let $\Omega = \mathbb{T} \times \mathbb{R}$, $\rho_0$ satisfy \eqref{t.1} with $\nabla\rho_0 \in L^{\infty} \cap \dot{H}^1(\Omega)$, and let $u_0 \in H^2(\Omega)$ be a solenoidal vector field. Then there exists a sufficiently small positive constant $\mathfrak{c}$ such that if
\begin{equation}\label{assumption-initial-103}
\|\rho_{0,\neq}\|_{L^{\infty}} + \|u_0\|_{H^1}\bigl(1 + \|u_0\|_{H^2}^3 + \|\nabla\rho_0\|_{L^{\infty}\cap\dot{H}^1}^3(1 + \|(\partial_1 {u}_{0,\neq},\, \partial_1 {\omega}_{0,\neq})\|_{L^2}^6)\bigr) \leq \mathfrak{c},
\end{equation}
the system \eqref{eqns-HINS-1} admits a unique global solution $(\rho, u, \nabla \Pi)$ satisfying
\[
\rho \in C([0,\infty[; L^\infty \cap \dot{W}^{1, \infty} \cap \dot{H}^2), \quad
u \in C([0,\infty[; H^2) \cap L^2(\mathbb{R}^+; \dot{H}^1),
\]
and \eqref{bdd-density-visc-1} holds for $(t,x)\in\R^+\times\Om.$
Moreover,  there exist positive constants $c_1$ and $c_2$ so that there hold
\begin{equation}\label{expo-Hins-u-neq-1}
\begin{split}
\sup_{t \in \mathbb{R}^+} \bigl(e^{7c_1 t} \|\partial_1 u_{\neq}(t)\|_{L^2}^2\bigr) + \int_0^{+\infty} e^{6c_1 t} \|(\partial_1 u_{\neq}, \, \partial_1^2 u_{\neq},\, \partial_t u_{\neq},\, &\nabla\Pi_{\neq})(t)\|_{L^2}^2 \, dt\\
&\quad \lesssim \|\partial_1 u_{0\neq}\|_{L^2}^2,
\end{split}
\end{equation}
\begin{equation}\label{expo-Hins-u-neq-2}
\begin{split}
&\sup_{t \in \mathbb{R}^+} \bigl(e^{3c_{2} t} \|\partial_1 \omega_{\neq}(t)\|_{L^2}^2\bigr) + \int_0^{+\infty} e^{2c_{2} t} \|(\partial_1 \omega_{\neq},\,\partial_t \omega_{\neq},\,\partial_1^2 \omega_{\neq})(t)\|_{L^2}^2 \, dt \\
& \lesssim \|(\partial_1 u_{0\neq},\,\partial_1 \omega_{0\neq})\|_{L^2}^2 + \|u_{0}\|_{H^1}\Bigl(1 + \|u_0\|_{H^2}^2 + \|\nabla\rho_0\|_{L^{\infty}\cap\dot{H}^1}^2 \\
&\qquad \qquad \qquad \qquad \qquad \qquad \qquad + \|\nabla\rho_0\|_{L^{\infty}\cap\dot{H}^1}^2 \|(\partial_1 {u}_{0\neq},\, \partial_1 {\omega}_{0\neq})\|_{L^2}^2\Bigr),
\end{split}
\end{equation}
\begin{equation}\label{grow-Hins-u-neq-3}
\begin{split}
\|\nabla\rho\|_{L^\infty_t(L^{\infty} \cap \dot{H}^1)} \leq 2\|\nabla\rho_0\|_{L^{\infty} \cap \dot{H}^1} \, e^{\frac{c_1}{2}  t} \quad \forall\, t > 0,
\end{split}
\end{equation}
and
\begin{equation}\label{bdd-Hins-u-neq-4}
\begin{split}
& \|(u,\,\omega,\,\partial_1 u,\,\nabla\omega) \|_{L^\infty(\R^+; L^2)}^2 + \|(\partial_1 u,\,\partial_1 \omega,\,\partial_1^2 u,\,\partial_t u,\,\partial_t \omega,\,\partial_1 \nabla\omega)\|_{L^2(\mathbb{R}^+;L^2)}^2 \\
&\lesssim \|u_0\|_{H^2}^2 + \|u_{0}\|_{H^1}^2 \|\nabla\rho_0\|_{L^{\infty}\cap\dot{H}^1}^6 \\
&\qquad + \|\nabla\rho_0\|_{L^{\infty}\cap\dot{H}^1}^2 \bigl(\|(\partial_1 {u}_{0\neq},\, \partial_1 {\omega}_{0\neq})\|_{L^2}^2 + \|(\partial_1 {u}_{0\neq},\, \partial_1 {\omega}_{0\neq})\|_{L^2}^4\bigr).
\end{split}
\end{equation}}
\end{thm}

\begin{rmk}\label{rmk-hins-thm-1}
By virtue of \eqref{bdd-density-visc-1}, we readily obtain
\begin{equation*}\label{bddunif-rho-1}
\begin{split}
\|\bar{\rho}\|_{L^{\infty}(\mathbb{R}^+; L^{\infty})} \leq \|\rho\|_{L^{\infty}(\mathbb{R}^+; L^{\infty})}, \quad
\|\rho_{\neq}\|_{L^{\infty}(\mathbb{R}^+; L^{\infty})} \leq 2\|\rho\|_{L^{\infty}(\mathbb{R}^+; L^{\infty})}
\end{split}
\end{equation*}
and
\begin{equation*}\label{bddunif-rho-2}
0 < M_1 \leq \bar{\rho}(t, x_2) \leq M_2, \quad
\forall\, (t, x_2) \in \mathbb{R}^+ \times \mathbb{R}.
\end{equation*}
\end{rmk}

\medskip
\noindent\textbf{Challenges and methodology.}
\smallskip

The primary obstacle to establishing global well-posedness for \eqref{eqns-ains-1} is controlling the time-weighted gradient norm. This requires a uniform-in-time bound for the quantity $\int_0^t \|\tau^{\frac{1}{2}} \nabla u\|_{L^{\infty}}^2 \, d\tau$. This difficulty stems from the lack of full viscosity in the momentum equation. The interplay among the density $\rho$, the acceleration $\partial_t u$, and the pressure $\Pi$, along with the nonorthogonal relation between $\rho \partial_t u$ and $\nabla \Pi$, precludes the use of straightforward energy methods that are applicable to the classical Navier--Stokes equations. To overcome this challenge, we adapt the strategy of Yudovich \cite{Yudovich1963}, who solved the 2D classical incompressible Euler equations via the vorticity equation. This approach is motivated by our observation that the vorticity equation derived from \eqref{momentum-ains-2} constitutes a fully dissipative system, yet retains a key structural similarity with its counterpart in Euler's equations. The proof is completed by synthesizing two key ingredients. First, via time-weighted energy estimates---which rely on the boundedness of the density and necessary \textit{a priori} bounds---we derive the crucial uniform estimate for $\int_0^t \|\tau^{\frac{1}{2}} \nabla u\|_{L^{\infty}}^2 \, d\tau$. Second, we employ a duality argument from \cite{HSWZ2024} to establish uniqueness. Together, these results yield the global well-posedness of system \eqref{eqns-ains-1}.

\vspace{0.2cm}

A central challenge in analyzing system \eqref{eqns-HINS-1} is obtaining an estimate for $\int_0^t \|\nabla u\|_{L^{\infty}} \, d\tau$ without the benefit of full viscous dissipation. This control plays a dual role, enabling the propagation of regularity while simultaneously ensuring that the density variation and velocity remain suitably small. As previously noted, even in the homogeneous case, the integral $\int_0^t \|\nabla u(\tau)\|_{L^{\infty}} \, d\tau$ may exhibit algebraic growth in time $t > 0$, regardless of the smallness of the velocity. This algebraic growth, in turn, leads to exponential growth of $\|\nabla \rho\|_{L^{\infty}}$ in system \eqref{eqns-HINS-1}. To address this, we aim to control $\|\nabla u\|_{L^{\infty}}$ using an anisotropic interpolation inequality (see \eqref{est-Lip-u-1}). However, this strategy faces a fundamental obstacle: the exponential growth of $\|\nabla \rho\|_{L^{\infty}}$ itself. To counteract this growth, it becomes necessary to establish a compensating exponential decay for the velocity field. The key to achieving this balance lies in the application of Poincaré's inequality. A crucial observation is that while the horizontal mean of $u$ is not conserved, its oscillatory component $u_{\neq}$ satisfies Poincaré's inequality $\|u_{\neq}\|_{L^2(\Omega)} \lesssim \|\partial_1 u_{\neq}\|_{L^2(\Omega)}$. This inequality, employed within an energy framework, ultimately yields the desired exponential decay of $\|u_{\neq}\|_{L^2(\Omega)}$.

\vspace{0.2cm}

We conclude this section by introducing the notations used throughout the paper.

\medskip
\noindent\textbf{Notations.} Let $A$ and $B$ be two operators. We denote by $[A, B] \stackrel{\text{def}}{=} AB - BA$ the commutator of $A$ and $B$. The notation $a \lesssim b$ means that there exists a uniform constant $C > 0$ (which may vary from line to line) such that $a \leq C b$. We denote by $(a \,|\, b)$  the $L^2(\Om)$ inner product of $a$ and $b$. We designate $(-\partial_2, \partial_1)^{T}$ as $\nabla^{\perp}$.

Let $X$ be a Banach space and $I$ an interval in $\mathbb{R}$. $C(I; X)$ stands for the space of continuous functions from $I$ into $X$. For $q \in [1, +\infty]$, the space $L^q(I; X)$ consists of measurable functions $f : I \to X$ such that $t \mapsto \|f(t)\|_X$ belongs to $L^q(I)$. In particular, if $I=[0,t],$ we may simplify $L^q(I)$ by $L^q_t.$
 If $v = (v^1, v^2)^T \in X$ is a vector-valued function, we mean that each component $v^i$ ($i = 1, 2$) belongs to $X$.
 
\medskip

\renewcommand{\theequation}{\thesection.\arabic{equation}}
\setcounter{equation}{0}
\section{The global well-posedness of the system (AINS) }\label{sect3}

In this section, we first derive the  time-weighted energy estimates for sufficiently smooth solutions of \eqref{eqns-ains-1} with bounded density. We then present the proof of Theorem \ref{thm-GUS-NS}.

\subsection{Time-weighted energy estimates with bounded density}

\begin{prop}[$L^2$ estimate of $u$]\label{prop-L-2-u-1}
{\sl Let $\rho_0$ satisfy \eqref{t.1} and let $u_0 \in L^2(\mathbb{R}^2)$ be a solenoidal vector field. Let $(\rho, u)$ be a sufficiently smooth solution of \eqref{eqns-ains-1} on $[0, T^{\ast}[$. Then for any $T \in [0, T^{\ast}[$, there hold \eqref{bdd-density-visc-1} and
\begin{equation}\label{visco-resis-1}
\|\mu(\rho)-1\|_{L^{\infty}_T(L^\infty)}=\|\mu(\rho_0)-1\|_{L^\infty}.
\end{equation}
Moreover, there exists a positive constant $C$ such that
\begin{equation}\label{L-2-u-0}
\|u\|_{L^\infty(\R^+; L^2)}^2 + \|\nabla u\|_{L^2(\R^+;L^2)}^2 \leq C\|u_0\|_{L^2}^2.
\end{equation}}
\end{prop}

\begin{proof}
First, from the continuity equation in \eqref{eqns-ains-1}, we deduce that for any $T \in [0, T^{\ast}[$, both \eqref{bdd-density-visc-1} and \eqref{visco-resis-1} hold.

Next, taking the $L^2$ inner product of the velocity equation in \eqref{eqns-ains-1} with $u$ and integrating by parts yields
\begin{equation}\label{L2-ains-u-1}
\frac{1}{2}\frac{d}{dt}\|\sqrt{\rho}u(t)\|_{L^2}^2 + \|(\partial_2u^1,\,\partial_1u^2)\|_{L^2}^2 = 0.
\end{equation}
Here and throughout, we denote $\omega \eqdefa \partial_1 u^2 - \partial_2 u^1$, so that
\begin{equation}\label{S2eq1}
\partial_1\omega = \Delta u^2 \quad \text{and} \quad \partial_2\omega = -\Delta u^1.
\end{equation}
Then, by the Biot--Savart law, $u = \nabla^\perp \Delta^{-1}\omega$, from which we obtain
\begin{align*}
\|\nabla u\|_{L^2} \lesssim \|\omega\|_{L^2} \lesssim \|(\partial_2u^1,\,\partial_1u^2)\|_{L^2}.
\end{align*}
Combining this with \eqref{L2-ains-u-1} gives
\begin{equation*}
\frac{d}{dt}\|\sqrt{\rho}u(t)\|_{L^2}^2 + c_0\|\nabla u\|_{L^2}^2 \leq 0.
\end{equation*}
By integrating the above equation over $[0,t]$ and using \eqref{bdd-density-visc-1}, we obtain \eqref{L-2-u-0}.
\end{proof}

\begin{prop}[$H^1$ estimate of $u$]\label{lem-H-1}
{\sl Under the assumptions of Proposition \ref{prop-L-2-u-1},
 for any $T \in [0, T^{\ast}[$, we have
\begin{equation}\label{momentum-ains-18}
\|t^{\frac{1}{2}}\omega\|_{L^{\infty}_T(L^2)}^2 + \|t^{\frac{1}{2}}(u_t,\,\Delta{u},\,\nabla\Pi)\|_{L^{2}_T(L^2)}^2
\leq C \|u_0\|_{L^2}^2 \exp\left(C \|u_0\|_{L^2}^4\right).
\end{equation}
}
\end{prop}

\begin{proof}
From the momentum equations in \eqref{eqns-ains-1}, we write
\begin{equation*}
\partial_t u + u \cdot \nabla u - (\partial_2^2u^1,\, \partial^2_1u^2)^T + \nabla \Pi = (1 - \rho)(\partial_t u+ u \cdot \nabla  u).
\end{equation*}
Applying the operator $\nabla^{\perp} \cdot$ to these equations gives
\begin{equation*}\label{momentum-ains-1}
\partial_t \omega + (u \cdot \nabla) \omega - \bigl(\partial_1^3 u^2 - \partial_2^3 u^1\bigr)
= \nabla^{\perp} \cdot \bigl[(1 - \rho)(\partial_t + u \cdot \nabla) u \bigr].
\end{equation*}
Yet by Biot-Sarvart law, $u =\nabla^{\perp} \Delta^{-1}\omega$,  we have
\begin{equation}\label{momentum-ains-2aaa}
\partial_1^3 u^2- \partial_2^3 u^1 = (\partial_1^4 + \partial_2^4)\Delta^{-1}\omega.
\end{equation}
Thus,
\begin{equation}\label{momentum-ains-2}
\partial_t \omega + (u \cdot \nabla) \omega - (\partial_1^4 + \partial_2^4)\Delta^{-1}\omega
= \nabla^{\perp} \cdot \bigl[(1 - \rho)(\partial_t + u \cdot \nabla) u \bigr].
\end{equation}

Taking the $L^2$ inner product of \eqref{momentum-ains-2} with $\omega$ and integrating by parts yields
\begin{equation}\label{momentum-ains-3}
\begin{split}
\frac{1}{2}\frac{d}{dt}\|\omega(t)\|_{L^2}^2
& + \|(\partial_1^4 + \partial_2^4)^{\frac{1}{2}}\Delta^{-\frac{1}{2}}\omega\|_{L^2}^2 \\
& \leq \|\rho - 1\|_{L^\infty} \bigl(\|u_t\|_{L^2} + \|u \cdot \nabla u\|_{L^2}\bigr) \|\nabla^{\perp}\omega\|_{L^2}.
\end{split}
\end{equation}
Noting from \eqref{S2eq1}  that
\begin{equation}\label{momentum-ains-5}
\|(\partial_1^4 + \partial_2^4)^{\frac{1}{2}}\Delta^{-\frac{1}{2}}\omega\|_{L^2}^2 \geq \f12\|\Delta u\|_{L^2}^2,
\end{equation}
substituting \eqref{momentum-ains-5} into \eqref{momentum-ains-3} and using the 2D interpolation inequality $\|a\|_{L^4}\lesssim \|a\|_{L^2}^{\frac12}\|\nabla a\|_{L^2}^{\frac12},$
 we obtain
\begin{equation*}\label{momentum-ains-6}
\begin{split}
\frac{d}{dt}\|\omega(t)\|_{L^2}^2 + \|\Delta u\|_{L^2}^2
&\leq C \|\rho_0 - 1\|_{L^\infty} \bigl(\|u_t\|_{L^2} + \|u\|_{L^4}\|\nabla u\|_{L^4}\bigr) \|\Delta u\|_{L^2} \\
&\leq C \bigl(\|\sqrt{\rho}\,u_t\|_{L^2} + \|u\|_{L^2}^{\frac{1}{2}}\|\nabla u\|_{L^2} \|\Delta u\|_{L^2}^{\frac{1}{2}}\bigr) \|\Delta u\|_{L^2}.
\end{split}
\end{equation*}
Applying Young's inequality yields
\begin{equation}\label{momentum-ains-7}
\frac{d}{dt}\|\omega(t)\|_{L^2}^2 + \frac{1}{2}\|\Delta u\|_{L^2}^2
\leq C_1 \bigl(\|\sqrt{\rho}\,u_t\|_{L^2}^2 + \|u\|_{L^2}^2 \|\nabla u\|_{L^2}^2 \|\omega\|_{L^2}^2 \bigr).
\end{equation}

On the other hand, taking the $L^2$ inner product of the momentum equation in \eqref{eqns-ains-1} with $u_t$ gives
\begin{equation*}\label{momentum-ains-7'}
\begin{split}
\|\sqrt{\rho}\,u_t\|_{L^2}^2 + \frac{1}{2}\frac{d}{dt}\|(\partial_2u^1,\partial_1u^2)(t)\|_{L^2}^2
\leq C \|\sqrt{\rho}u_t\|_{L^2} \|u\|_{L^4}\|\nabla u\|_{L^4},
\end{split}
\end{equation*}
from which we infer
\begin{equation}\label{momentum-ains-8}
\frac{d}{dt}\|(\partial_2u^1,\partial_1u^2)(t)\|_{L^2}^2 + \|\sqrt{\rho}\,u_t\|_{L^2}^2
\leq C \|u\|_{L^2}\|\omega\|_{L^2}^2\|\Delta u\|_{L^2}.
\end{equation}

Adding equation \eqref{momentum-ains-7} multiplied by $\frac{1}{2C_1}$ to equation \eqref{momentum-ains-8} and then applying Young's inequality yields
\begin{equation*}
\begin{split}
 \frac{d}{dt}\Bigl(\|(\partial_2u^1,\partial_1u^2)(t)\|_{L^2}^2+\frac{1}{2C_1}\|\omega(t)\|_{L^2}^2\Bigr)&
+ \f12\|\sqrt{\rho}\,u_t\|_{L^2}^2 + \frac1{8C_1}\|\Delta u\|_{L^2}^2 \\
&\leq  C\|u\|_{L^2}^2\|\nabla u\|_{L^2}^2\|\omega\|_{L^2}^2.
\end{split}
\end{equation*}
Multiplying this inequality by $t$, we obtain
\begin{equation*}\label{momentum-ains-11}
\begin{split}
&\frac{d}{dt} \Bigl(\|t^{\frac{1}{2}}(\partial_2u^1,\partial_1u^2)(t)\|_{L^2}^2+\frac{1}{2C_1}\|t^{\frac{1}{2}}\omega(t)\|_{L^2}^2\Bigr)
+  \bigl(\f12\|t^{\frac{1}{2}}\sqrt{\rho}\,u_t\|_{L^2}^2 + \frac1{8C_1}\|t^{\frac{1}{2}}\Delta u\|_{L^2}^2 \bigr)\\
&\leq \|(\partial_2u^1,\partial_1u^2)\|_{L^2}^2 +\frac{1}{2C_1} \|\omega\|_{L^2}^2
+ C\|u\|_{L^2}^2\|\nabla u\|_{L^2}^2\|t^{\frac12}\omega\|_{L^2}^2.	
\end{split}
\end{equation*}
Applying Gronwall's inequality leads to
\begin{equation*}\label{momentum-ains-12}
\begin{split}
\|t^{\frac12}\omega\|_{L^{\infty}_T(L^2)}^2
+ \|t^{\frac12}(u_t,\,\Delta u)\|_{L^{2}_T(L^2)}^2
\leq C\|\nabla u\|_{L^{2}_T(L^2)}^2
\exp\Bigl(C\|u\|_{L^{\infty}_T(L^2)}^2\|\nabla u\|_{L^{2}_T(L^2)}^2\Bigr),
\end{split}
\end{equation*}
which together with \eqref{L-2-u-0} implies
\begin{equation}\label{momentum-ains-13}
\begin{split}
\|t^{\frac12}\omega\|_{L^{\infty}_T(L^2)}^2
+ \|t^{\frac12}(u_t,\,\Delta u)\|_{L^{2}_T(L^2)}^2
\leq C\|u_0\|_{L^2}^2 \exp\bigl(C\|u_0\|_{L^2}^4\bigr).
\end{split}
\end{equation}

Notice that $\na\Pi$ is perpendicular to $\begin{pmatrix}\partial_2^2u^1 \\ \partial_1^2u^2\end{pmatrix}$ in $L^2(\R^2)$ and 
\begin{equation*}\label{momentum-ains-14}
\nabla\Pi-\begin{pmatrix}\partial_2^2u^1 \\ \partial_1^2u^2\end{pmatrix} = -\rho (u_t +u\cdot\nabla{u}),
\end{equation*}
which implies
\begin{equation*}\label{momentum-ains-15}
\|\nabla\Pi\|_{L^2} \lesssim \|\rho\|_{L^{\infty}}
\bigl(\|u_t\|_{L^2} + \|u\|_{L^4}\|\nabla u\|_{L^4}\bigr). 
\end{equation*}
Using the 2D interpolation inequality $\|a\|_{L^4}\lesssim \|a\|_{L^2}^{\frac12}\|\nabla a\|_{L^2}^{\frac12},$
we obtain
\begin{equation*}\label{momentum-ains-16}
\begin{split}
\|\nabla\Pi\|_{L^2}
&\lesssim \|u_t\|_{L^2}
+ \|u\|_{L^2}^{\frac12}\|\nabla u\|_{L^2}\|\Delta u\|_{L^2}^{\frac12} \\
&\lesssim \|(u_t,\,\Delta u)\|_{L^2} + \|u\|_{L^2}\|\nabla u\|_{L^2}^2,
\end{split}
\end{equation*}
from which and \eqref{momentum-ains-13} we infer
\begin{equation}\label{momentum-ains-17}
\begin{split}
\|t^{\frac12}\nabla\Pi\|_{L^{2}_T(L^2)}
&\leq C\Bigl(\|t^{\frac12}(u_t,\,\Delta u)\|_{L^{2}_T(L^2)}
+ \|u\|_{L^{\infty}_T(L^2)} \|\nabla u\|_{L^{2}_T(L^2)}\|t^{\frac12}\omega\|_{L^{\infty}_T(L^2)}\Bigr) \\
&\leq C \|u_0\|_{L^2}^2 \exp\bigl(C\|u_0\|_{L^2}^4\bigr).	
\end{split}
\end{equation}

Combining \eqref{momentum-ains-13} with \eqref{momentum-ains-17} yields \eqref{momentum-ains-18}.
This completes the proof of Proposition \ref{lem-H-1}.
\end{proof}
Let us now turn to the $\dot H^2$ estimate of $u$.

\begin{prop}[$\dot H^2$ estimate of $u$]\label{lem-H-2}
{\sl Under the assumptions of Proposition \ref{prop-L-2-u-1}, for any $T \in [0, T^{\ast}[$, we have
\begin{equation}\label{est-noj-nabla2ub-sum-3}
\begin{split}
\|t(D_tu, \Delta u,\nabla\Pi)\|_{L^{\infty}_{T}(L^2)}^2
&+\|t\,\nabla D_tu\|_{L^2_{T}(L^2)}^2+\|t^{\frac{1}{2}}\,\nabla u\|_{L^2_T(L^{\infty})}^2 \leq C \|u_0\|_{L^2}^2,
\end{split}
\end{equation}
where $D_t\eqdefa \p_t+u\cdot\na$ and the constant $C$ depends only on $\|u_0\|_{L^{2}}$, $M_1$ and $M_2$.}
\end{prop}

\begin{proof}
We divide the proof of \eqref{est-noj-nabla2ub-sum-3} into the following steps:

\smallskip
\noindent\textbf{$\bullet$ The estimate of $\|D_tu\|_{L^2}$}
\smallskip

Applying the material derivative $D_t \eqdefa \partial_t + u \cdot \nabla$ to the momentum equations in \eqref{eqns-ains-1} and using $D_t\rho = 0$, we obtain
\begin{equation}\label{est-Dtu-1}
\begin{split}
&\rho\partial_t D_t u + \rho(u \cdot \nabla) D_t u
 -\begin{pmatrix}\partial_2^2 D_t u^1 \\ \partial_1^2 D_t u^2\end{pmatrix} \\
&= -\begin{pmatrix}2\partial_2[(\partial_2 u \cdot \nabla) u^1] - \partial_2^2 u \cdot \nabla u^1 \\
2\partial_1[(\partial_1 u \cdot \nabla) u^2] - \partial_1^2 u \cdot \nabla u^2\end{pmatrix}
 -\bigl(\nabla\Pi_t + (u \cdot \nabla)\nabla\Pi\bigr).
\end{split}
\end{equation}
Taking the $L^2$ inner product of \eqref{est-Dtu-1} with $D_t u$ and integrating by parts yields
\begin{equation}\label{est4}
\begin{split}
\frac{1}{2}&\frac{d}{dt}\|\sqrt{\rho}D_t u(t)\|_{L^2}^2
+\|(\partial_2 D_t u^1,\partial_1 D_t u^2)\|_{L^2}^2 \\
=&-2\int_{\R^2}\partial_2[(\partial_2 u \cdot \nabla) u^1] \ |\ D_t u^1\,dx
 -2\int_{\R^2}\partial_1[(\partial_1 u \cdot \nabla) u^2]\ |\  D_t u^2\,dx \\
&+\int_{\R^2}\partial_2^2 u \cdot \nabla u^1 \ |\  D_t u^1\,dx
 +\int_{\R^2}\partial_1^2 u \cdot \nabla u^2 \, |\  D_t u^2\,dx \\
&-\int_{\R^2}\bigl(\nabla\Pi_t + (u \cdot \nabla)\nabla\Pi\bigr)\ |\  D_t u\,dx
\eqdefa  I_1 + I_2 + I_3.
\end{split}
\end{equation}

For $I_1$, we observe that
\begin{equation}\label{est4aaa}
\begin{split}
|I_1|
&\leq \|\nabla u\|_{L^4}^2 \|(\partial_2 D_t u^1,\partial_1 D_t u^2)\|_{L^2} \\
&\leq \frac{1}{4}\|(\partial_2 D_t u^1,\partial_1 D_t u^2)\|_{L^2}^2
+ C\|\nabla u\|_{L^4}^4.
\end{split}
\end{equation}

To estimate $I_2$, we first integrate by parts and then use duality between $BMO$ and the Hardy space $\mathscr{H}^1$:
\begin{align*}
|I_2| &= \Bigl|\int_{\R^2} u^1 \partial_2^2 u \cdot \nabla D_t u^1\,dx
      + \int_{\R^2} u^2 \partial_1^2 u \cdot \nabla D_t u^2\,dx\Bigr| \\
      &\leq \|u\|_{BMO} \bigl(\|\partial_2^2 u \cdot \nabla D_t u^1\|_{\mathscr{H}^1}
      + \|\partial_1^2 u \cdot \nabla D_t u^2\|_{\mathscr{H}^1}\bigr).
\end{align*} 
Since $\dv\,\partial_1^2 u = \dv\,\partial_2^2 u = 0$ and $\curl\nabla D_t u^i = 0$ for $i=1,2$, the div-curl lemma (see \cite{CLM1993}) gives
\begin{equation*}\label{S2eq3}
\begin{split}
\|\partial_2^2 u \cdot \nabla D_t u^1\|_{\mathscr{H}^1}
+ \|\partial_1^2 u \cdot \nabla D_t u^2\|_{\mathscr{H}^1}
\lesssim (\|\partial_1^2 u\|_{L^2} + \|\partial_2^2 u\|_{L^2})\|\nabla D_t u\|_{L^2}.
\end{split}
\end{equation*}
Together with the interpolation inequality $\|f\|_{BMO(\R^2)} \lesssim \|\nabla f\|_{L^2(\R^2)}$, this implies
\begin{equation}\label{S2eq2}
|I_2| \leq C\|\nabla u\|_{L^2}\|\Delta u\|_{L^2}\|\nabla D_t u\|_{L^2}.
\end{equation}

For $I_3$, using integration by parts and $\dv D_t u = \dv\dv(u \otimes u)$, we write
\begin{equation*}
\begin{split}
I_3 &= \int_{\R^2}\Pi_t \, \dv D_t u\,dx - \int_{\R^2} u^j \partial_j \partial_i \Pi \, D_t u^i\,dx \\
    &= \int_{\R^2}\Pi_t \, \dv D_t u\,dx
       + \int_{\R^2}\bigl(u \cdot \nabla\Pi \, \dv D_t u
       + (\partial_i u \cdot \nabla)\Pi \, D_t u^i\bigr)\,dx \\
    &= \int_{\R^2}D_t\Pi\, \dv\dv(u \otimes u)\,dx
       - \int_{\R^2}
       \Pi (\partial_i u \cdot \nabla) D_t u^i\,dx,
\end{split}
\end{equation*}
which follows 
\begin{equation*}
\begin{split}
I_3 = \frac{d}{dt}\int_{\R^2}\Pi \, \dv\dv(u \otimes u)\,dx
       - \int_{\R^2}\bigl(\Pi \, D_t\dv\dv(u \otimes u)
       +\Pi \, \partial_i u^j \partial_j D_t u^i\bigr)\,dx.
\end{split}
\end{equation*}
Note that $[D_t, \partial_i] = -\partial_i u^k \partial_k$, so
\begin{equation*}
\begin{split}
&-\int_{\R^2}\Pi \, D_t\dv\dv(u \otimes u) \,dx= -\int_{\R^2}\Pi \, D_t(\partial_i u^j \partial_j u^i)\,dx \\
&= -2\int_{\R^2}\Pi\,\partial_j u^i \, D_t\partial_iu^j \,dx = -2\int_{\R^2}\Pi \,\partial_j u^i\,(\partial_iD_t u^j+ [D_t,\partial_i] u^j)  \,dx\\
&= -2\int_{\R^2}\Pi \, \partial_j u^i ( \partial_i D_t u^j - \partial_i u^k \partial_k u^j)\,dx.
\end{split}
\end{equation*}
Consequently,
\begin{equation}\label{est4ddd}
\begin{split}
I_3 &= \frac{d}{dt}\int_{\R^2}\Pi (\partial_i u \cdot \nabla u^i)\,dx  - 3\int_{\R^2}\Pi \, \partial_j u^i  \partial_i D_t u^j \,dx
       +2 \int_{\R^2}\Pi \, \partial_i u^j \partial_j u^k \partial_k u^i \,dx \\
    &= \frac{d}{dt}\int_{\R^2}\Pi (\partial_i u \cdot \nabla u^i)\,dx
       - 3\int_{\R^2}\Pi \, \partial_j u^i \, \partial_i D_t u^j \,dx,
\end{split}
\end{equation}
where we used $\dv\, u = 0$ and $\sum_{1\le i,j,k\le2}\partial_i u^j \partial_j u^k \partial_k u^i = 0$ in the last equation.

By duality between $BMO$ and the Hardy space $\mathscr{H}^1$,
\begin{equation*}\label{est4eee}
\Bigl|\int_{\R^2}\Pi (\partial_i u \cdot \nabla) D_t u^i\,dx\Bigr|
\leq \|\Pi\|_{BMO} \|(\partial_i u \cdot \nabla) D_t u^i\|_{\mathscr{H}^1},
\end{equation*}
Since $\dv\, u = 0$ and $\curl\nabla D_t u = 0$, the div-curl lemma (see \cite{CLM1993}) yields
\begin{equation*}\label{est4fff}
\|(\partial_i u \cdot \nabla) D_t u^i\|_{\mathscr{H}^1}
\lesssim \|\nabla u\|_{L^2} \|\nabla D_t u\|_{L^2},
\end{equation*}
which together with the embedding $\|f\|_{BMO(\R^2)} \lesssim \|\nabla f\|_{L^2(\R^2)}$ implies
\begin{equation*}\label{est4ggg}
\Bigl|\int_{\R^2}\Pi (\partial_i u \cdot \nabla) D_t u^i\,dx\Bigr|
\lesssim \|\nabla\Pi\|_{L^2} \|\nabla u\|_{L^2} \|\nabla D_t u\|_{L^2}.
\end{equation*}
By inserting the above estimate into \eqref{est4ddd}, we find
\begin{equation}\label{est4ggg}
\Bigl| I_3- \frac{d}{dt}\int_{\R^2}\Pi (\partial_i u \cdot \nabla u^i)\,dx \Bigr|
\lesssim \|\nabla\Pi\|_{L^2} \|\nabla u\|_{L^2} \|\nabla D_t u\|_{L^2}.
\end{equation}

Substituting the estimates \eqref{est4aaa}, \eqref{S2eq2} and \eqref{est4ggg} into \eqref{est4} gives
\begin{equation}\label{est5}
\begin{split}
\frac{d}{dt}&\Bigl(\|\sqrt{\rho}D_t u\|_{L^2}^2
- 2\int_{\R^2}\Pi (\partial_i u \cdot \nabla u^i)\,dx\Bigr)
+ \|(\partial_2 D_t u^1,\partial_1 D_t u^2)\|_{L^2}^2 \\
&\lesssim \|\nabla u\|_{L^4}^4
+ \|\nabla u\|_{L^2} \|(\nabla\Pi,\,\Delta u)\|_{L^2} \|\nabla D_t u\|_{L^2}.
\end{split}
\end{equation}
Observing that $\|\curl D_t u\|_{L^2}^2 = \int_{\R^2}(\partial_1 D_t u^2 - \partial_2 D_t u^1)^2\,dx
\leq 2\|(\partial_2 D_t u^1,\partial_1 D_t u^2)\|_{L^2}^2$
and $\dv D_t u = \dv\dv(u \otimes u) = \sum_{1\le j,k\le2}\partial_j u^k \partial_k u^j$, we deduce
\begin{equation}\label{est5ddd}
\begin{split}
\|\nabla D_t u\|_{L^2}
&\lesssim \|\curl D_t u\|_{L^2} + \|\dv D_t u\|_{L^2}  \lesssim \|(\partial_2 D_t u^1,\partial_1 D_t u^2)\|_{L^2} + \|\nabla u\|_{L^4}^2.
\end{split}
\end{equation}
Inserting \eqref{est5ddd} into \eqref{est5} and using $\|\nabla u\|_{L^4} \lesssim \|\nabla u\|_{L^2}^{\frac12}\|\Delta u\|_{L^2}^{\frac12}$, we apply Young's inequality to obtain
\begin{equation}\label{est5eee}
\begin{split}
\frac{d}{dt}&\Bigl(\|\sqrt{\rho}D_t u\|_{L^2}^2
- 2\int_{\R^2}\Pi (\partial_i u \cdot \nabla u^i)\,dx\Bigr)
+  \mathfrak{e}_1\|\nabla D_t u\|_{L^2}^2 \leq C\|\nabla u\|_{L^2}^2 \|(\Delta u, \nabla\Pi)\|_{L^2}^2,
\end{split}
\end{equation}
hand in what follows, all $\mathfrak{e}_j$ (with $j\in\mathbb{N}$) denote positive constants that are independent of the initial data.

\smallskip
	\noindent\textbf{$\bullet$ The estimate of $\|(\Delta u, \nabla\Pi)\|_{L^2}$}
\smallskip

From the momentum equations in \eqref{eqns-ains-1}, we have
\begin{equation}\label{est5ggg}
\|\nabla\Pi\|_{L^2} \leq C\bigl(\|\sqrt{\rho}D_t u\|_{L^2} + \|\Delta u\|_{L^2}\bigr).
\end{equation}
Taking the $L^2$ inner product of the momentum equations in \eqref{eqns-ains-1} with $-\nabla^{\perp}\omega$ and integrating by parts yields
\begin{equation*}\label{est5hhh}
\begin{split}
-(\rho D_t u, \nabla^{\perp}\omega)_{L^2}
- \int_{\R^2}\nabla\Pi \cdot \nabla^{\perp}\omega \,dx
- \Bigl(\begin{pmatrix}-\partial_2^2 u^1 \\ -\partial_1^2 u^2\end{pmatrix},\,
\begin{pmatrix}\partial_2\omega \\ -\partial_1\omega\end{pmatrix}\Bigr)_{L^2} = 0.
\end{split}
\end{equation*}
Since $\dv\, u = 0$, we have
\begin{equation*}
\begin{split}
\Bigl(\begin{pmatrix}-\partial_2^2 u^1 \\ -\partial_1^2 u^2\end{pmatrix},\,
\begin{pmatrix}\partial_2\omega \\ -\partial_1\omega\end{pmatrix}\Bigr)_{L^2}
&= \|\partial_1^2 u^2\|_{L^2}^2 + \|\partial_1^2 u^1\|_{L^2}^2  + \|\partial_2^2 u^1\|_{L^2}^2 + \|\partial_2^2 u^2\|_{L^2}^2,
\end{split}
\end{equation*}
and since $\int_{\R^2}\nabla\Pi \cdot \nabla^{\perp}\omega \,dx = 0$, we infer
\begin{equation*}\label{est5iii}
\begin{split}
\|\Delta u\|_{L^2}^2
&\lesssim \|\sqrt{\rho}D_t u\|_{L^2} \|\nabla\omega\|_{L^2} \\
&\lesssim \|\sqrt{\rho}D_t u\|_{L^2} \|\Delta u\|_{L^2} \leq \frac{1}{2}\|\Delta u\|_{L^2}^2 + C\|\sqrt{\rho}D_t u\|_{L^2}^2,
\end{split}
\end{equation*}
which together with \eqref{est5ggg} implies
\begin{equation}\label{est6}
\|(\Delta u, \nabla\Pi)\|_{L^2} \lesssim \|\sqrt{\rho}D_t u\|_{L^2}.
\end{equation}

Substituting \eqref{est6} into \eqref{est5eee} gives
\begin{equation}\label{est6aaa}
\begin{split}
\frac{d}{dt}&\Bigl(\|\sqrt{\rho}D_t u\|_{L^2}^2
- 2\int_{\R^2}\Pi (\partial_i u \cdot \nabla u^i)\,dx\Bigr)
+ \mathfrak{e}_1\|\nabla D_t u\|_{L^2}^2 \leq C \|\nabla u\|_{L^2}^2 \|\sqrt{\rho}D_t u\|_{L^2}^2.
\end{split}
\end{equation}
Yet by
applying the div-curl lemma \cite{CLM1993} and the embedding $\dot{H}^1(\R^2) \hookrightarrow BMO(\R^2)$ again, we obtain
\begin{equation*}
\begin{split}
\Bigl|\int_{\R^2}\Pi (\partial_i u \cdot \nabla u^i)\,dx\Bigr|
&\leq \|\Pi\|_{BMO} \|\partial_i u \cdot \nabla u^i\|_{\mathscr{H}^1} \lesssim \|\nabla\Pi\|_{L^2} \|\nabla u\|_{L^2}^2,
\end{split}
\end{equation*}
which together with \eqref{est6} yields
\begin{equation}\label{est6ccc}
\begin{split}
2\Bigl|\int_{\R^2}\Pi (\partial_i u \cdot \nabla u^i)\,dx\Bigr|
&\leq C \|\sqrt{\rho}D_t u\|_{L^2} \|\nabla u\|_{L^2}^2 \leq \frac{1}{2} \|\sqrt{\rho}D_t u\|_{L^2}^2 + C_2 \|\omega\|_{L^2}^4.
\end{split}
\end{equation}
Moreover, from \eqref{momentum-ains-7} we have
\begin{equation}\label{est6ddd}
\frac{d}{dt} \|\omega\|_{L^2}^4
= 2\|\omega\|_{L^2}^2 \frac{d}{dt} \|\omega\|_{L^2}^2
\lesssim \|\omega\|_{L^2}^2 \bigl(\|u_t\|_{L^2}^2
+ \|u\|_{L^2}^2 \|\nabla u\|_{L^2}^2 \|\omega\|_{L^2}^2\bigr).
\end{equation}

Define
\begin{equation}\label{S2eq3}
E_1(t) \eqdefa \|\sqrt{\rho}D_t u\|_{L^2}^2
- 2\int_{\R^2}\Pi (\partial_i u \cdot \nabla u^i)\,dx
+ 2C_{2} \|\omega\|_{L^2}^4.
\end{equation}
Then it follows from \eqref{est6ccc} that
\begin{equation}\label{est6ggg}
\begin{split}
\frac{1}{2}\|\sqrt{\rho}D_t u\|_{L^2}^2 + C_{2} \|\omega\|_{L^2}^4
&\leq E_1(t) \leq \frac{3}{2}\|\sqrt{\rho}D_t u\|_{L^2}^2 + 3C_{2} \|\omega\|_{L^2}^4.
\end{split}
\end{equation}
Summing \eqref{est6aaa} with $2C_2 \times$ \eqref{est6ddd} yields
\begin{equation*}\label{est6fff}
\begin{split}
\frac{d}{dt}E_1(t) +\mathfrak{e}_1\|\nabla D_t u\|_{L^2}^2
&\leq C\Bigl(\|\nabla u\|_{L^2}^2 \|\sqrt{\rho}D_t u\|_{L^2}^2
+ \|\omega\|_{L^2}^2 \|u_t\|_{L^2}^2  + \|u\|_{L^2}^2 \|\nabla u\|_{L^2}^2 \|\omega\|_{L^2}^4\Bigr).
\end{split}
\end{equation*}
Multiplying this inequality by $t^2$ gives
\begin{equation*}\label{est6hhh}
\begin{split}
\frac{d}{dt}\bigl( t^2 E_1(t) \bigr) + \mathfrak{e}_1 \|t\,\nabla D_t u\|_{L^2}^2
\leq & 2t E_1(t) + C\Bigl(\|\nabla u\|_{L^2}^2 \|t\sqrt{\rho}D_t u\|_{L^2}^2 \\
&+ \|t^{\frac12}\omega\|_{L^2}^2 \|t^{\frac12} u_t\|_{L^2}^2 + \|u\|_{L^2}^2 \|\nabla u\|_{L^2}^2 \|t^{\frac12}\omega\|_{L^2}^4\Bigr).
\end{split}
\end{equation*}
Using \eqref{est6ggg}, we infer
\begin{equation*}\label{est6iii}
\begin{split}
\frac{d}{dt}\bigl(t^2 E_1(t)\bigr) &+  \mathfrak{e}_1\|t\,\nabla D_t u\|_{L^2}^2 \leq 3\|t^{\frac12}\sqrt{\rho}D_t u\|_{L^2}^2 + 6 \|t^{\frac12}\omega\|_{L^2}^2 \|\omega\|_{L^2}^2 \\
&\quad + C\Bigl(\|\nabla u\|_{L^2}^2 t^2 E_1(t)
+ \|t^{\frac12}\omega\|_{L^2}^2 \|t^{\frac12} u_t\|_{L^2}^2 + \|u\|_{L^2}^2 \|\nabla u\|_{L^2}^2 \|t^{\frac12}\omega\|_{L^2}^4\Bigr).
\end{split}
\end{equation*}
Applying Gronwall's inequality yields
\begin{equation*}\label{est6jjj-prelim}
\begin{split}
\sup_{t\in[0,T]}&\bigl(t^2 E_1(t)\bigr) +  \mathfrak{e}_1\|t\,\nabla D_t u\|_{L^2_T(L^2)}^2
\leq C \exp\!\bigl(C\|\nabla u\|_{L^2_T(L^2)}^2\bigr)\times \\
&\Bigl(
\|t^{\frac12}\sqrt{\rho}D_t u\|_{L^2_T(L^2)}^2 + \|t^{\frac12}\omega\|_{L^{\infty}_T(L^2)}^2 \|\omega\|_{L^2_T(L^2)}^2 \\
&\quad+ \|t^{\frac12}\omega\|_{L^{\infty}_T(L^2)}^2 \|t^{\frac12} u_t\|_{L^2_T(L^2)}^2 + \|u\|_{L^{\infty}_T(L^2)}^2 \|\nabla u\|_{L^2_T(L^2)}^2 \|t^{\frac12}\omega\|_{L^{\infty}_T(L^2)}^4\Bigr).
\end{split}
\end{equation*}
Together with \eqref{momentum-ains-18}, this ensures
\begin{equation}\label{est6jjj}
\sup_{t\in[0,T]}\bigl(t^2 E_1(t)\bigr)
+ \|t\,\nabla D_t u\|_{L^2_T(L^2)}^2
\leq C\|u_0\|_{L^2}^2 \exp\!\bigl(C(1+\|u_0\|_{L^2}^2)\|u_0\|_{L^2}^2\bigr).
\end{equation}
Finally, from \eqref{est6} and \eqref{est6ggg} we deduce
\begin{equation}\label{est6kkk}
\|t(\Delta u, \nabla\Pi)\|_{L^{\infty}_T(L^2)}^2
\leq C\|u_0\|_{L^2}^2 \exp\bigl(C\|u_0\|_{L^2}^4\bigr).
\end{equation}

\smallskip
	\noindent\textbf{$\bullet$ The estimate of $\|t^{\frac{1}{2}}\,\nabla{u}\|_{L^2_T(L^\infty)}$}
	\smallskip

We first get, by
applying the operator $\nabla^{\perp}\cdot$ to the momentum equations in \eqref{eqns-ains-1},  that
\begin{equation*}\label{eqns-momentum-ains-2aabb}
\begin{split}
&\nabla^{\perp}\cdot(\rho\,D_tu)-\bigl(\partial_1^3 u^2-\partial_2^3 u^1\bigr)= 0.
	\end{split}
\end{equation*}
It follows from \eqref{momentum-ains-2aaa} that
\begin{equation*}\label{momentum-ains-2bbb}
\begin{split}
&(\partial_1^4+\partial_2^4)\Delta^{-1}\omega=\nabla^{\perp}\cdot(\rho\,D_tu),
	\end{split}
\end{equation*}
from which, we infer
\begin{equation*}\label{momentum-ains-2ccc}
\begin{split}
 \|\nabla^2u\|_{L^4} &\lesssim \|(\partial_1^4+\partial_2^4)^{\frac{3}{4}}(-\Delta)^{-1}\omega\|_{L^4}\lesssim \|(\partial_1^4+\partial_2^4)^{-\frac{1}{4}}\nabla^{\perp}\cdot(\rho\,D_tu)\|_{L^4}  \\
&\lesssim \|\rho\,D_tu\|_{L^4} \lesssim\|D_tu\|_{L^4} \lesssim\|D_tu\|_{L^2}^{\frac{1}{2}} \|\nabla\,D_tu\|_{L^2}^{\frac{1}{2}}.
	\end{split}
\end{equation*}
As a result, we achieve
\begin{equation*}\label{momentum-ains-2ddd}
\begin{split}
&\|\nabla{u}\|_{L^{\infty}} \lesssim \|\nabla{u}\|_{L^4}^{\frac{1}{2}} \|\nabla^2u\|_{L^4}^{\frac{1}{2}}\lesssim \|\nabla{u}\|_{L^2}^{\frac{1}{4}} \|\nabla^2{u}\|_{L^2}^{\frac{1}{4}}  \|D_tu\|_{L^2}^{\frac{1}{4}} \|\nabla\,D_tu\|_{L^2}^{\frac{1}{4}},
	\end{split}
\end{equation*}
we thus deduce from  \eqref{momentum-ains-18}, \eqref{est6jjj} and \eqref{est6kkk} that \begin{equation}\label{momentum-ains-2eee}
\begin{split}
\|t^{\frac{1}{2}}\,\nabla{u}\|_{L^2_T(L^\infty)}  &\lesssim \|\nabla{u}\|_{L^2_T(L^2)}^{\frac{1}{4}} \|t^{\frac{1}{2}}\nabla^2{u}\|_{L^2_T(L^2)}^{\frac{1}{4}} \|t^{\frac{1}{2}}D_tu\|_{L^2_T(L^2)}^{\frac{1}{4}} \|t\nabla\,D_tu\|_{L^2_T(L^2)}^{\frac{1}{4}}\\
&\leq C\|u_0\|_{ L^2}^2\exp\bigl(C\|u_0\|_{ L^2}^4\bigr).
	\end{split}
\end{equation}

By summarizing the  estimates  \eqref{est6jjj}, \eqref{est6kkk}  and \eqref{momentum-ains-2eee}, we conclude the proof of \eqref{est-noj-nabla2ub-sum-3}.
This ends the proof of Proposition \ref{lem-H-2}.
\end{proof}

\subsection{The proof of Theorem \ref{thm-GUS-NS} }

Let's turn to  the proof of  Theorem \ref{thm-GUS-NS}.

\begin{proof}[Proof of Theorem \ref{thm-GUS-NS}]
With the {\it a priori} estimates established in Propositions \ref{prop-L-2-u-1}--\ref{lem-H-2}, it is standard to apply the compactness argument from \cite{LP1996} to construct a global weak solution $(\rho, u, \nabla\Pi)$ to system \eqref{eqns-ains-1} satisfying $\rho \in C_{\rm w}([0,\infty[; L^{\infty})$, $u \in C([0, +\infty[; L^2) \cap L^2(\mathbb{R}^+; \dot{H}^{1})$, and the estimate \eqref{est-variable-2}. It remains to prove the uniqueness of such a global solution.

Let $(\rho, u, \nabla\Pi)$ and $(\bar{\rho}, \bar{u}, \nabla\bar{\Pi})$ be two such global solutions to system \eqref{eqns-ains-1} as constructed in Theorem \ref{thm-GUS-NS}. Define
\begin{equation*}
(\delta\rho, \delta u, \nabla\delta\Pi) \eqdefa (\rho - \bar{\rho}, u - \bar{u}, \nabla\Pi - \nabla\bar{\Pi}).
\end{equation*}
Then, from \eqref{eqns-ains-1}, the system for $(\delta\rho, \delta u, \nabla\delta\Pi)$ reads
\begin{equation}\label{delta-eqns-nonlip-1}
\begin{cases}
\partial_t\delta\rho + \bar{u} \cdot \nabla\delta\rho = -\delta u \cdot \nabla\rho, \qquad (t,x)\in\R^+\times\R^2,\\[2mm]
\rho(\partial_t\delta u + u \cdot \nabla\delta u) - (\partial_2^2\delta u^1, \, \partial_1^2\delta u^2)^{T} + \nabla\delta\Pi
= -\delta\rho \, \bar{D}_t\bar{u} - \rho \delta u \cdot \nabla\bar{u}, \\[2mm]
\na\cdot \delta u = 0, \\[2mm]
(\delta\rho, \delta u)|_{t=0} = (0, 0),
\end{cases}
\end{equation}
where $\bar{D}_t \eqdefa \partial_t + \bar{u} \cdot \nabla$.

\smallskip
\noindent\textbf{$\bullet$ The energy estimate of $\delta u$} \smallskip

Taking the $L^2$ inner product of the $\delta u$ equation in \eqref{delta-eqns-nonlip-1} with $\delta u$ yields
\begin{equation*}\label{delta-nonlip-2}
\begin{split}
\frac{1}{2}&\frac{d}{dt}\|\rho^{\frac12}\delta u(t)\|_{L^2}^2 + \|(\partial_2\delta u^1, \partial_1\delta u^2)\|_{L^2}^2 = \sum_{j=1}^2 II_j(t), \quad \text{with} \\
&II_1(t) \eqdefa -\int_{\mathbb{R}^2} \delta\rho \, \bar{D}_t\bar{u} \cdot \delta u(t) \,dx, \quad
II_2(t) \eqdefa -\int_{\mathbb{R}^2} \rho (\delta u \cdot \nabla)\bar{u} \cdot \delta u(t) \,dx.
\end{split}
\end{equation*}
Integrating this inequality over $[0,T]$ gives
\begin{equation}\label{delta-nonlip-500}
\begin{split}
{\rm{E}}(T) \eqdefa \|\delta u\|_{L^{\infty}_T(L^2)}^2 + \|\nabla\delta u\|_{L^2_T(L^2)}^2 \lesssim \|II_1\|_{L^1_T} + \|II_2\|_{L^1_T}.
\end{split}
\end{equation}

First, observe that
\begin{equation*}
\begin{split}
|I_2| &\lesssim \|\delta u\|_{L^4}^2 \|\nabla\bar{u}\|_{L^2} \lesssim \|\delta u\|_{L^2} \|\nabla\delta u\|_{L^2} \|\nabla\bar{u}\|_{L^2},
\end{split}
\end{equation*}
from which we infer
\begin{equation}\label{est-I-2-4-102}
\begin{split}
\|I_2\|_{L^1_T}
&\lesssim \|\nabla\bar{u}\|_{L^{2}_T(L^2)} \|\delta u\|_{L^{\infty}_T(L^2)} \|\nabla\delta u\|_{L^{2}_T(L^2)} \\
&\leq L(T) \, {\rm{E}}(T).
\end{split}
\end{equation}
Here and hereafter, $L(T)$ denotes a positive function (possibly varying from line to line) such that $L(T) \to 0$ as $T \to 0^+$.

\smallskip
\noindent\textbf{$\bullet$ The Estimate of $\|II_1\|_{L^1_T}$} \smallskip

Following the approach in \cite{HSWZ2024}, we handle the estimate of $II_1$ via a duality argument.
Let $\psi(s, t; x)$ solve the transport equation
\begin{equation}\label{est-I1-1}
\begin{cases}
\partial_s\psi(s, t; x) + \bar{u}(s, x) \cdot \nabla\psi(s, t; x) = 0 \quad \text{for } s \in [0,t], \\[2mm]
\psi(s, t; x)|_{s=t} = [\bar{D}_t\bar{u} \cdot \delta u](t, x).
\end{cases}
\end{equation}
By Lemma \ref{lem-est-transport-1} in Appendix \ref{sect2}, for any $t \geq 0$,
\begin{equation}\label{est-I1-2}
 \|\nabla\psi(s, t)\|_{L^{\frac43}} \leq C \|\nabla(\bar{D}_t\bar{u} \cdot \delta u)(t)\|_{L^{\frac43}} \, e^{C|\ln(t/s)|^{\frac12}}.
\end{equation}

Using the first equation in \eqref{delta-eqns-nonlip-1} and \eqref{est-I1-1}, we compute
\begin{equation*}
\begin{split}
- \frac{d}{ds} \int_{\mathbb{R}^2} \delta\rho(s) \, \psi(s, t) \,dx
&= - \int_{\mathbb{R}^2} \bigl(\psi(s, t) \, \partial_s\delta\rho(s) + \delta\rho(s) \, \partial_s\psi(s, t)\bigr) \,dx \\
&= \int_{\R^2} \psi(s, t) \, \delta u(s) \cdot \nabla\rho(s) \,dx \\
&= - \int_{\mathbb{R}^2} \rho(s) \, \delta u(s) \cdot \nabla\psi(s, t) \,dx.
\end{split}
\end{equation*}
Integrating over $[0,t]$ and using \eqref{est-I1-2} yields
\begin{equation*}
\begin{split}
|II_1(t)| &\leq \int_0^t \|\nabla\psi(s, t)\|_{L^{\frac43}} \, \|\delta u(s)\|_{L^4} \,ds \\
&\leq C \|\nabla(\bar{D}_t\bar{u} \cdot \delta u)(t)\|_{L^{\frac43}} \int_0^t e^{C|\ln(t/s)|^{\frac12}} \, \|\delta u(s)\|_{L^4} \,ds \\
&\leq C \|t\nabla(\bar{D}_t\bar{u} \cdot \delta u)(t)\|_{L^{\frac43}} \, F(t)
\end{split}
\end{equation*}
with $F(t) \eqdefa \frac{1}{t} \int_0^t e^{C|\ln(t/s)|^{\frac12}} \, \|\delta u(s)\|_{L^4} \,ds$.
Consequently,
\begin{equation}\label{est-I1-4}
\|II_1\|_{L^1_T} \leq C \|t\nabla(\bar{D}_t\bar{u} \cdot \delta u)\|_{L^{\frac43}_T(L^{\frac43})} \, \|F\|_{L^4_T}.
\end{equation}

From Lemma \ref{lem-est-transport-2} in the Appendix \ref{sect2}, we have
\begin{equation}\label{est-I1-5aaa}
\begin{split}
\|F\|_{L^4_T} &\leq C \bigl\|\|\delta u(t)\|_{L^4}\bigr\|_{L^4_T} \leq C \|\delta u\|_{L^{\infty}_T(L^2)}^{\frac12} \|\nabla\delta u\|_{L^2_T(L^2)}^{\frac12} \leq C {\rm{E}}^{\frac12}(T).
\end{split}
\end{equation}
Meanwhile,
\begin{equation*}
\begin{split}
\|t\nabla(\bar{D}_t\bar{u} \cdot \delta u)\|_{L^{\frac43}_T(L^{\frac43})}
&\leq \|t(\nabla \bar{D}_t\bar{u}) \delta u\|_{L^{\frac43}_T(L^{\frac43})} + \|t \bar{D}_t\bar{u} \, (\nabla \delta u)\|_{L^{\frac43}_T(L^{\frac43})} \\
&\leq \|t \nabla \bar{D}_t\bar{u}\|_{L^2_T(L^2)} \|\delta u\|_{L^4_T(L^4)} + \|t \bar{D}_t\bar{u}\|_{L^4_T(L^4)} \|\nabla \delta u\|_{L^2_T(L^2)}.
\end{split}
\end{equation*}
Together with \eqref{est-noj-nabla2ub-sum-3}, this implies
\begin{equation}\label{est-I1-5bbb}
\begin{split}
\|t\nabla(\bar{D}_t\bar{u} \cdot \delta u)\|_{L^{\frac43}_T(L^{\frac43})}
&\leq C\Bigl(\|t \nabla \bar{D}_t\bar{u}\|_{L^2_T(L^2)} \|\delta u\|_{L^{\infty}_T(L^2)}^{\frac12} \|\nabla\delta u\|_{L^2_T(L^2)}^{\frac12} \\
&\quad + \|t \bar{D}_t\bar{u}\|_{L^{\infty}_T(L^2)}^{\frac12} \|t \nabla\bar{D}_t\bar{u}\|_{L^2_T(L^2)}^{\frac12} \|\nabla \delta u\|_{L^2_T(L^2)}\Bigr) \leq L(T) {\rm{E}}^{\frac12}(T).
\end{split}
\end{equation}
Substituting \eqref{est-I1-5aaa} and \eqref{est-I1-5bbb} into \eqref{est-I1-4} gives
\begin{equation}\label{est-I1-6}
\|II_1\|_{L^1_T} \leq L(T) {\rm{E}}(T).
\end{equation}

Inserting the estimates \eqref{est-I-2-4-102} and \eqref{est-I1-6} into \eqref{delta-nonlip-500}, we obtain
\begin{equation*}
{\rm{E}}(T) \leq L(T) {\rm{E}}(T) \quad \text{for all } T > 0.
\end{equation*}
Choosing $T$ sufficiently small so that $L(T) \leq \frac12$, we have
\begin{equation*}
{\rm{E}}(T) \leq \frac12 {\rm{E}}(T),
\end{equation*}
which implies ${\rm{E}}(T) = 0$. Therefore, $\delta u(t) = \nabla \delta\Pi(t) \equiv 0$ and consequently $\delta\rho(t) = 0$ for all $t \in [0, T]$. A bootstrap argument then yields the uniqueness part of Theorem \ref{thm-GUS-NS}. This completes the proof of Theorem \ref{thm-GUS-NS}.
\end{proof}

\medskip
\renewcommand{\theequation}{\thesection.\arabic{equation}}
\setcounter{equation}{0}
\section{Global well-posedness of (HINS)}\label{sect-4-periodic}

This section is devoted to the global well-posedness of the two-dimensional inhomogeneous incompressible anisotropic Navier-Stokes equations \eqref{eqns-HINS-1} on $\Om=\mathbb{T}\times \mathbb{R}$.
We first observe from \eqref{average-u-1-a} and \eqref{average-u-1-b} that

\begin{lem}\label{lem-prop-decomp-1}
{\sl Let $\bar{f}$ and $f_{\neq}$ be defined as in \eqref{average-u-1-a} and \eqref{average-u-1-b}. Then the following properties hold:
\begin{equation}\label{average-u-2-a}
\begin{split}
&\bar{\bar{f}} = \bar{f}, \quad (f_{\neq})_{\neq} = f_{\neq}, \quad (\bar{f})_{\neq} = \overline{(f_{\neq})} = 0, \\
&\overline{\partial_1 f} = 0, \quad (\partial_1 f)_{\neq} = \partial_1 f_{\neq}, \quad \overline{\bar{f} \, g_{\neq}} = 0, \quad (\bar{f} \, g_{\neq})_{\neq} = \bar{f} \, g_{\neq}.
\end{split}
\end{equation}
If a vector field $F$ satisfies $\nabla \cdot F = 0$, then
\begin{equation}\label{average-u-2-b}
\nabla \cdot \overline{F} = 0, \quad \nabla \cdot F_{\neq} = 0.
\end{equation}
For any $f \in L^2(\mathbb{T}\times \mathbb{R})$,
\begin{equation}\label{average-u-2-c}
\begin{split}
&(\bar{f}, f_{\neq})_{L^2} = 0, \quad \|f\|_{L^2}^2 = \|\bar{f}\|_{L^2}^2 + \|f_{\neq}\|_{L^2}^2 \andf
\|f_{\neq}\|_{L^2} \lesssim \|\partial_1 f_{\neq}\|_{L^2},
\end{split}
\end{equation}
where $(\cdot, \cdot)_{L^2}$ denotes the $L^2(\Om)$ inner product.

For any two functions $f$ and $g$ on $\mathbb{T}\times \mathbb{R}$, the following decomposition holds:
\begin{equation}\label{decomp-bar-neq-1}
\begin{split}
&\overline{f g} = \bar{f} \,\bar{g} + \overline{f_{\neq} \, g_{\neq}} \andf (f g)_{\neq} = \bar{f} \, g_{\neq} + f_{\neq} \,\bar{g} + (f_{\neq} \, g_{\neq})_{\neq}.
\end{split}
\end{equation}}
\end{lem}

\begin{rmk}\label{rmk-1}
Due to the incompressibility condition $\dv u = 0$, there exists a stream function $\phi$ such that
\begin{equation*}\label{stream-func-1}
u = \nabla^{\perp} \phi =\begin{pmatrix}-\partial_2 \phi \\ \partial_1\phi\end{pmatrix},  \end{equation*}
which together with \eqref{average-u-2-a} implies $\bar{u}^2 \equiv 0$, and hence $u^2 = u_{\neq}^2$.
\end{rmk}

Given initial data $(\r_0,u_0),$ with $\r_0$ satisfying   \eqref{t.1},  $\nabla\rho_0 \in L^{\infty} \cap \dot{H}^1(\Omega)$, and  $u_0 \in H^2(\Omega),$ we deduce by a standard argument that the system \eqref{eqns-HINS-1}  has a unique local strong solution $(\r,u)$ on $[0,T^\ast[.$ Below we shall prove that $T^\ast=\infty$ provided that there holds \eqref{assumption-initial-103}.

\subsection{Estimate of $\|(u_{\neq}, \partial_1 u_{\neq})\|_{L^{\infty}_T(L^2)}$}


\begin{prop}[Estimate of $\|(u_{\neq}, \partial_1 u_{\neq})\|_{L^2}$]\label{prop-HINS-H1-neq-1}
 {\sl Let $\rho_0$ satisfy \eqref{t.1} and let $u_0 \in H^2(\Om)$ be a solenoidal vector field.
Let $(\rho, u)$ be a sufficiently smooth solution of \eqref{eqns-HINS-1} on $[0, T^{\ast}[$. Then for any $T \in [0, T^{\ast}[$, \eqref{bdd-density-visc-1} holds on $\Om,$ and there exists  positive constants $\delta_0$ and $c_1$ such that if
\begin{equation}\label{condition-smallness-sol-1}
\sup_{t \in [0, T^{\ast}[} \bigl(\|\rho_{\neq}(t)\|_{L^{\infty}} + \|(u, \omega)(t)\|_{L^2}\bigr) \leq \delta_0,
\end{equation}
then
\begin{equation}\label{est-Hins-expo-1}
\begin{split}
&\sup_{t\in[0,T]} \bigl(e^{7c_1 t}\|(u_{\neq}, \partial_1 u_{\neq})(t)\|_{L^2}^2\bigr)\\
&\quad+ \int_0^{T} e^{6c_1 t} \|(\partial_1 u_{\neq}, \partial_1^2 u_{\neq}, \partial_t u_{\neq}, \nabla \Pi_{\neq})(t)\|_{L^2}^2 \,dt \lesssim \|(u_{0,\neq}, \partial_1 u_{0,\neq})\|_{L^2}^2.
\end{split}
\end{equation}}
\end{prop}

\begin{proof}
Taking the average of the $u^1$ equation in \eqref{eqns-HINS-1}  on $\mathbb{T}$ and using $\bar{u}^2 = 0$, we obtain
\begin{equation}\label{eqns-HINS-aver-8}
\begin{split}
\overline{\rho} \,\partial_t\bar{u}^1 &+ \overline{\rho} \,\partial_2 (\overline{u_{\neq}^2 u_{\neq}^1})
+ \overline{\rho_{\neq} \,\partial_t u_{\neq}^1} + \overline{\rho_{\neq} \,\partial_1(u^1 u^1)_{\neq}}
+ \overline{\rho_{\neq} \,\partial_2(u^2 u^1)_{\neq}} = 0.
\end{split}
\end{equation}
Applying the operator $\mathcal{P}_{\neq}$ to the momentum equations in \eqref{eqns-HINS-1} yields
\begin{equation}\label{eqns-HINS-aver-9}
\begin{cases}
\bar{\rho} \,\partial_t u_{\neq}^1 - \partial_1^2 u_{\neq}^1 + \partial_1 \Pi_{\neq}
+ \bar{\rho} \,\bar{u}^1 \partial_1 u_{\neq}^1 + \bar{\rho} \, u_{\neq}^2 \partial_2 \bar{u}^1 = F_1, \\[2mm]
\bar{\rho} \,\partial_t u_{\neq}^2 - \partial_1^2 u_{\neq}^2 + \partial_2 \Pi_{\neq}
+ \bar{\rho} \,\bar{u}^1 \partial_1 u_{\neq}^2 = F_2, \\[2mm]
\nabla \cdot u_{\neq} = 0,
\end{cases}
\end{equation}
with
\begin{equation}\label{eqns-HINS-aver-10}
\begin{split}
F_1 \eqdefa &-\Bigl( \rho_{\neq} \,\partial_t \bar{u}^1 + (\rho_{\neq} \,\partial_t u_{\neq}^1)_{\neq}
+ \bar{\rho} \,\partial_1[(u_{\neq}^1 u_{\neq}^1)_{\neq}] + \bar{\rho} \,\partial_2 [(u_{\neq}^2 u_{\neq}^1]_{\neq})
+ \rho_{\neq} \,\partial_2 [\overline{u_{\neq}^2 u_{\neq}^1}] \\
&\qquad+ 2\bigl\{\rho_{\neq} \bigl[ \bar{u}^1 \partial_1 u_{\neq}^1 + (u_{\neq}^1 \partial_1 u_{\neq}^1)_{\neq}\bigr]\bigr\}_{\neq} + \bigl\{\rho_{\neq} \,\partial_2 \bigl[u_{\neq}^2 \bar{u}^1 + (u_{\neq}^2 u_{\neq}^1)_{\neq}\bigr]\bigr\}_{\neq} \Bigr), \\[2mm]
F_2 \eqdefa  &-\Bigl( \rho_{\neq} \,\partial_t u_{\neq}^2 + \bar{\rho} \,\partial_1(u_{\neq}^2 u_{\neq}^1)
+ \bar{\rho} \,\partial_2 (u_{\neq}^2 u_{\neq}^2) + \rho_{\neq} \,\partial_1 (u_{\neq}^1 u_{\neq}^2) \\
&\qquad+ \rho_{\neq} \,\partial_2 (u_{\neq}^2 u_{\neq}^2)
{{+ \rho_{\neq} \partial_1(\bar{u}^1 u_{\neq}^2)}} \Bigr).
\end{split}
\end{equation}

\noindent\textbf{$\bullet$ The energy estimate of $u_{\neq}$} \smallskip

Taking the $L^2$ inner product of \eqref{eqns-HINS-aver-9} with $u_{\neq}$, integrating by parts and using $\dv u_{\neq} = 0$, we obtain
\begin{equation*}\label{est-Hins-L2-1}
\begin{split}
\frac{1}{2}\frac{d}{dt}\|\bar{\rho}^{\frac12} u_{\neq}(t)\|_{L^2}^2 + \|\partial_1 u_{\neq}\|_{L^2}^2
= &\frac{1}{2}\int_{\Om} |u_{\neq}|^2 \partial_t\bar{\rho} \,dx
   - \int_{\Om} \bar{\rho} \,\bar{u}^1 u_{\neq}^1 \partial_1 u_{\neq}^1 \,dx \\
& - \int_{\Om} \bar{\rho} \,u_{\neq}^1 u_{\neq}^2 \partial_2 \bar{u}^1 \,dx
   - \int_{\Om} \bar{\rho} \,\bar{u}^1 u_{\neq}^2 \partial_1 u_{\neq}^2 \,dx \\
& + \int_{\Om} (F_1 u_{\neq}^1 + F_2 u_{\neq}^2) \,dx,
\end{split}
\end{equation*}
from which  and the facts:
\begin{align*}
\int_{\Om} \bar{\rho} \,\bar{u}^1 \partial_1 u_{\neq}^1 \,u_{\neq}^1 \,dx
&= -\frac{1}{2}\int_{\Om} \partial_1 (\bar{\rho} \,\bar{u}^1) |u_{\neq}^1|^2 \,dx = 0, \andf
\int_{\Om} \bar{\rho} \,\bar{u}^1 u_{\neq}^2 \partial_1 u_{\neq}^2 \,dx = 0,
\end{align*}
we deduce that
\begin{equation}\label{est-Hins-L2-2}
\begin{split}
&\frac{1}{2}\frac{d}{dt}\|\bar{\rho}^{\frac12} u_{\neq}(t)\|_{L^2}^2 + \|\partial_1 u_{\neq}\|_{L^2}^2 = \sum_{j=1}^3 III_j \with\\
&III_1 \eqdefa \frac{1}{2}\int_{\Om} |u_{\neq}|^2 \partial_t\bar{\rho} \,dx,
\quad
III_2 \eqdefa -\int_{\Om} \bar{\rho} \,u_{\neq}^2 u_{\neq}^1 \partial_2 \bar{u}^1 \,dx, \andf\\
&III_3 \eqdefa \int_{\Om} (F_1 u_{\neq}^1 + F_2 u_{\neq}^2) \,dx.
\end{split}
\end{equation}

Taking the average of the density equation in \eqref{eqns-HINS-1} and using $\bar{u}^2 = 0$ yields
\begin{equation}\label{eqns-HINS-aver-rho-1}
\partial_t\bar{\rho} + \partial_2(\overline{u_{\neq}^2 \rho_{\neq}}) = 0.
\end{equation}
In view of \eqref{eqns-HINS-aver-rho-1}, we get, by using integration by parts, that
\begin{equation*}\label{est-I-L2-1}
\begin{split}
III_1
&=\int_{\Omega}  \overline{u_{\neq}^2 \rho_{\neq}} \bigl(u_{\neq}^1\partial_2 u_{\neq}^1+u_{\neq}^2\partial_2 u_{\neq}^2\bigr) \,dx=\int_{\Omega}({u}_{\neq}^1\partial_2{u}_{\neq}^1-{u}_{\neq}^2\partial_1{u}_{\neq}^1)\overline{u_{\neq}^2\rho_{\neq}} \,dx,
\end{split}
\end{equation*}
from which, we infer
\begin{equation*}\label{est-I-L2-2}
\begin{split}
 |III_1|&\lesssim \bigl(\|{u}_{\neq}^1\|_{L^2}\|\partial_2{u}_{\neq}^1\|_{L^2}+\|{u}_{\neq}^2\|_{L^2}\|\partial_1{u}_{\neq}^1\|_{L^2}\bigr)
 \|u_{\neq}^2\|_{L^{\infty}} \|\rho_{\neq}\|_{L^{\infty}}.
\end{split}
\end{equation*}
Note that
\begin{equation}\label{est-I-L2-3}
\begin{split}
\|u_{\neq}^2\|_{L^{\infty}} &\lesssim \|\partial_1 u_{\neq}^2\|_{L^2}^{\frac12} \|\partial_2\partial_1 u_{\neq}^2\|_{L^2}^{\frac12}
      \lesssim \|\partial_1 u_{\neq}^2\|_{L^2}^{\frac12} \|\partial_1^2 u_{\neq}^1\|_{L^2}^{\frac12}, \\
\|u_{\neq}^1\|_{L^{\infty}} &\lesssim \|\partial_1 u_{\neq}^1\|_{L^2}^{\frac12} \|\partial_2\partial_1 u_{\neq}^1\|_{L^2}^{\frac12}, \\
\|\bar{u}^1\|_{L^{\infty}} &\lesssim \|\bar{u}^1\|_{L^2}^{\frac12} \|\partial_2\bar{u}^1\|_{L^2}^{\frac12}.
\end{split}
\end{equation}
Using \eqref{est-I-L2-3} and the fact $\|{u}_{\neq}\|_{L^2}\lesssim \|\partial_1{u}_{\neq}\|_{L^2}$, we obtain
\begin{equation*}\label{est-I-L2-4}
\begin{split}
 |III_1|&\lesssim  \|(\partial_2{u}_{\neq}^1,\,\partial_1{u}_{\neq}^1)\|_{L^2}\|{u}_{\neq}\|_{L^2}
 \|\partial_1{u}^2_{\neq}\|_{L^2}^{\frac{1}{2}} \|\partial_1^2{u}^1_{\neq}\|_{L^2}^{\frac{1}{2}}\\
 &\lesssim  \|(\partial_2{u}_{\neq}^1,\,\partial_1{u}_{\neq}^1)\|_{L^2}
\bigl (\|\partial_1{u}_{\neq}\|_{L^2}^2+\|\partial_1^2{u}^1_{\neq}\|_{L^2}^2\bigr).
\end{split}
\end{equation*}

For $III_2$, we have
\begin{equation*}\label{est-Hins-L2-4}
\begin{split}
 |III_2| &\lesssim\|\pa_2\bar{u}^1\|_{L^2}\|u_{\neq}^1\|_{L^2_y(L^\infty_x)}\|u_{\neq}^2\|_{L^2_x(L^\infty_y)}\\
   &\lesssim \|\pa_2\bar{u}^1\|_{L^2}\|u_{\neq}^1\|_{L^2}^{\frac12}\|\pa_1u_{\neq}^1\|_{L^2}^{\frac12}
   \|u_{\neq}^2\|_{L^2}^{\frac12}\|\pa_2u_{\neq}^2\|_{L^2}^{\frac12}\\
& \lesssim
  \|\pa_2\bar{u}^1\|_{L^2}\|\pa_1u_{\neq}^1\|_{L^2}^{\frac12}
    \|\pa_1u_{\neq}^1\|_{L^2}^{\frac12}\|\partial_1u_{\neq}^2\|_{L^2}^{\frac12}\|\pa_1u_{\neq}^1\|_{L^2}^{\frac12}  \lesssim \|\pa_2\bar{u}^1\|_{L^2}\|\pa_1u_{\neq} \|_{L^2}^2.
\end{split}
\end{equation*}
Thus we obtain
\begin{equation}\label{est-Hins-L2-6}
\begin{split}
\frac{d}{dt}&\|\bar{\rho}^{\frac12} u_{\neq}(t)\|_{L^2}^2 + 2\|\partial_1 u_{\neq}\|_{L^2}^2 \lesssim \|\pa_2\bar{u}^1\|_{L^2}\|\pa_1u_{\neq} \|_{L^2}^2\\
&+ \|(\partial_2 u_{\neq}^1,\partial_1 u_{\neq}^1)\|_{L^2}
        \bigl(\|\partial_1 u_{\neq}\|_{L^2}^2 + \|\partial_1^2 u_{\neq}^1\|_{L^2}^2\bigr)  + \|(F_1, F_2)\|_{L^2} \|u_{\neq}\|_{L^2}.
\end{split}
\end{equation}

\noindent\textbf{$\bullet$ The energy estimate of $\p_1 u_{\neq}$} \smallskip

By  taking the $L^2$ inner product of the $u_{\neq}$ equations in \eqref{eqns-HINS-aver-9} with $\partial_t u_{\neq}$ and integrating by parts, we find
\begin{equation}\label{est-Hins-H1-u-neq-1}
\begin{split}
\|\bar{\rho}^{\frac12} \partial_t u_{\neq}\|_{L^2}^2
+ \frac{1}{2}\frac{d}{dt}\|\partial_1 u_{\neq}(t)\|_{L^2}^2 =
&\int_{\Om} \bigl(F_1 - \bar{\rho} \,\bar{u}^1 \partial_1 u_{\neq}^1
   - \bar{\rho} \,u_{\neq}^2 \partial_2 \bar{u}^1\bigr) \partial_t u_{\neq}^1 \,dx \\
& + \int_{\Om} \bigl(F_2 - \bar{\rho} \,\bar{u}^1 \partial_1 u_{\neq}^2\bigr) \partial_t u_{\neq}^2 \,dx.
\end{split}
\end{equation}

Note that
\begin{equation*}\label{est-Hins-H1-u-neq-2}
\begin{split}
\Bigl|\int_{\Om} \bar{\rho} \,\bar{u}^1 \partial_1 u_{\neq}^1 \partial_t u_{\neq}^1 \,dx\Bigr|
&\lesssim \|\bar{u}^1\|_{L^2_x(L^{\infty}_y)} \|\partial_1 u_{\neq}^1\|_{L^2_y(L^{\infty}_x)} \|\partial_t u_{\neq}^1\|_{L^2} \\
&\lesssim \|\bar{u}^1\|_{L^2}^{\frac12} \|\partial_2\bar{u}^1\|_{L^2}^{\frac12}
          \|\partial_1^2 u_{\neq}^1\|_{L^2} \|\partial_t u_{\neq}^1\|_{L^2},
\end{split}
\end{equation*}
and
\begin{equation*}\label{est-Hins-H1-u-neq-3}
\begin{split}
\Bigl|\int_{\Om} \bar{\rho} \,\bar{u}^1 \partial_1 u_{\neq}^2 \partial_t u_{\neq}^2 \,dx\Bigr|
&\lesssim \|\bar{u}^1\|_{L^2_x(L_y^{\infty})} \|\partial_1 u_{\neq}^2\|_{L^{2}_y(L_x^{\infty})} \|\partial_t u_{\neq}^2\|_{L^2} \\
&\lesssim \|\bar{u}^1\|_{L^{2}}^{\frac12} \|\partial_2\bar{u}^1\|_{L^{2}}^{\frac12}
          \|\partial_1^2 u_{\neq}^2\|_{L^{2}} \|\partial_t u_{\neq}^2\|_{L^2}.
\end{split}
\end{equation*}
For $\int_{\Om} \bar{\rho} \,u_{\neq}^2 \partial_2 \bar{u}^1 \partial_t u_{\neq}^1 \,dx$, we have
\begin{equation*}
\begin{split}
\Bigl|\int_{\Om} \bar{\rho} \,u_{\neq}^2 \partial_2 \bar{u}^1 \partial_t u_{\neq}^1 \,dx\Bigr|
&\lesssim \|\partial_2\bar{u}^1\|_{L^2} \|u_{\neq}^2\|_{L^{\infty}} \|\partial_t u_{\neq}^1\|_{L^2},
\end{split}
\end{equation*}
which together with \eqref{est-I-L2-3} implies
\begin{equation*}\label{est-Hins-H1-u-neq-4}
\begin{split}
\Bigl|\int_{\Om} \bar{\rho} \,u_{\neq}^2 \partial_2 \bar{u}^1 \partial_t u_{\neq}^1 \,dx\Bigr|
&\lesssim \|\partial_2\bar{u}^1\|_{L^2}
          \|\partial_1 u_{\neq}^2\|_{L^{2}}^{\frac12}
          \|\partial_1^2 u_{\neq}^1\|_{L^{2}}^{\frac12} \|\partial_t u_{\neq}^1\|_{L^2}.
\end{split}
\end{equation*}

Substituting the above estimates into \eqref{est-Hins-H1-u-neq-1} yields
\begin{equation}\label{est-Hins-H1-u-neq-6}
\begin{split}
\frac{d}{dt}&\|\partial_1 u_{\neq}(t)\|_{L^2}^2 + 2\|\bar{\rho}^{\frac12} \partial_t u_{\neq}\|_{L^2}^2 \\
\lesssim & \Bigl(|\bar{u}^1\|_{L^2}^{\frac12} \|\partial_2\bar{u}^1\|_{L^2}^{\frac12}
        \|\partial_1^2 u_{\neq}^1\|_{L^2} + \|\partial_2\bar{u}^1\|_{L^2}
        \|\partial_1 u_{\neq}^2\|_{L^{2}}^{\frac12}
        \|\partial_1^2 u_{\neq}^1\|_{L^{2}}^{\frac12} \\
       & + \|\bar{u}^1\|_{L^{2}}^{\frac12} \|\partial_2\bar{u}^1\|_{L^{2}}^{\frac12}
        \|\partial_1^2 u_{\neq}^2\|_{L^{2}} + \|(F_1, F_2)\|_{L^2} \Bigr)\|\partial_t u_{\neq}\|_{L^2}.
\end{split}
\end{equation}

Combining \eqref{est-Hins-H1-u-neq-6} with \eqref{est-Hins-L2-6} gives
\begin{equation}\label{est-Hins-H1-u-neq-7}
\begin{split}
&\frac{d}{dt}\bigl(\|\bar{\rho}^{\frac12} u_{\neq}(t)\|_{L^2}^2 + \|\partial_1 u_{\neq}(t)\|_{L^2}^2\bigr)
+ 2\bigl(\|\bar{\rho}^{\frac12} \partial_t u_{\neq}\|_{L^2}^2 + \|\partial_1 u_{\neq}\|_{L^2}^2\bigr) \\
&\lesssim  \|\partial_2\bar{u}^1\|_{L^2} \|\partial_1 u_{\neq}\|_{L^2}^2
+ \|\bar{u}^1\|_{L^2}^{\frac12} \|\partial_2\bar{u}^1\|_{L^2}^{\frac12}
        \bigl(\|\partial_1^2 u_{\neq}^1\|_{L^2}^2 + \|\partial_t u_{\neq}^1\|_{L^2}^2\bigr) \\
&\quad + \|\partial_2\bar{u}^1\|_{L^2}
        \bigl(\|\partial_1 u_{\neq}^2\|_{L^{2}}^2 + \|\partial_1^2 u_{\neq}^1\|_{L^{2}}^2 + \|\partial_t u_{\neq}^1\|_{L^2}^2\bigr) \\
&\quad + \|\bar{u}^1\|_{L^{2}}^{\frac12} \|\partial_2\bar{u}^1\|_{L^{2}}^{\frac12}
        \bigl(\|\partial_1^2 u_{\neq}^2\|_{L^{2}}^2 + \|\partial_t u_{\neq}^2\|_{L^2}^2\bigr) \\
&\quad + \|(\partial_2 u_{\neq}^1,\partial_1 u_{\neq}^1)\|_{L^2}
        \bigl(\|\partial_1 u_{\neq}\|_{L^2}^2 + \|\partial_1^2 u_{\neq}^1\|_{L^2}^2\bigr)  + \|(F_1, F_2)\|_{L^2} \bigl(\|u_{\neq}\|_{L^2} + \|\partial_t u_{\neq}\|_{L^2}\bigr).
\end{split}
\end{equation}
Since $\dv u_{\neq} = 0$, from the momentum equations in \eqref{eqns-HINS-aver-9} we obtain
\begin{equation*}\label{est-Hins-H1-u-neq-8}
\begin{split}
&\|\partial_1^2 u_{\neq}\|_{L^2}^2 + \|\nabla \Pi_{\neq}\|_{L^2}^2
\lesssim \|\bar{\rho}\|_{L^{\infty}} \|\bar{\rho}^{\frac12} \partial_t u_{\neq}\|_{L^2}^2 \\
&\quad + \|\bar{\rho}\|_{L^{\infty}}^2 \|\bar{u}^1\|_{L^2_x(L^{\infty}_y)}^2
 \|\partial_1 u_{\neq}\|_{L^2_y(L^{\infty}_x)}^2 + \|\bar{\rho}\|_{L^{\infty}}^2 \|u_{\neq}^2\|_{L^{\infty}}^2 \|\partial_2\bar{u}^1\|_{L^2}^2 + \|(F_1, F_2)\|_{L^2}^2.
\end{split}
\end{equation*}
Using \eqref{est-I-L2-3}, this implies
\begin{equation}\label{est-Hins-H1-u-neq-9}
\begin{split}
\|\partial_1^2 u_{\neq}\|_{L^2}^2 + \|\nabla \Pi_{\neq}\|_{L^2}^2
&\lesssim \|\bar{\rho}^{\frac12} \partial_t u_{\neq}\|_{L^2}^2 + \|\bar{u}^1\|_{L^2} \|\partial_2\bar{u}^1\|_{L^2} \|\partial_1^2 u_{\neq}\|_{L^2}^2 \\
&\quad + \|\partial_2\bar{u}^1\|_{L^2}^2 \|\partial_1 u_{\neq}^2\|_{L^2} \|\partial_1^2 u_{\neq}^1\|_{L^2} + \|(F_1, F_2)\|_{L^2}^2.
\end{split}
\end{equation}

Define
\begin{equation}\label{def-H1-u-neq-1}
\begin{split}
\mathcal{E}_{1, \neq}(t) &\eqdefa \|(\bar{\rho}^{\frac12} u_{\neq}, \partial_1 u_{\neq})(t)\|_{L^2}^2 
\andf\\
\mathcal{D}_{1, \neq}(t) &\eqdefa \|(\bar{\rho}^{\frac12} u_{\neq}, u_{\neq}, \partial_1 u_{\neq}, \partial_t u_{\neq}, \partial_1^2 u_{\neq}, \nabla \Pi_{\neq})(t)\|_{L^2}^2.
\end{split}
\end{equation}

From \eqref{est-Hins-H1-u-neq-7}, \eqref{est-Hins-H1-u-neq-9}, and the fact $\|\bar{\rho}^{\frac12} u_{\neq}\|_{L^2}^2 \lesssim \|u_{\neq}\|_{L^2}^2 \lesssim \|\partial_1 u_{\neq}\|_{L^2}^2$, we deduce that there exists a uniform positive constant $c_1$ so that
\begin{equation}\label{est-Hins-H1-u-neq-10}
\begin{split}
\frac{d}{dt} &\mathcal{E}_{1, \neq}(t) + 8c_1 \mathcal{D}_{1, \neq}
\lesssim \bigl(\|(\bar{u}^1, \partial_2\bar{u}^1, \partial_2 u_{\neq}^1, \partial_1 u_{\neq}^1)\|_{L^{2}}
+\|\bar{u}^1\|_{L^2}^2+ \|\partial_2\bar{u}^1\|_{L^2}^2 \bigr)\\
&\quad\times
        \|(\partial_1 u_{\neq}, \partial_1^2 u_{\neq}, \partial_t u_{\neq})\|_{L^2}^2  + \|(F_1, F_2)\|_{L^2} \|(\partial_1 u_{\neq}, \partial_t u_{\neq})\|_{L^2}+ \|(F_1, F_2)\|_{L^2}^2.
\end{split}
\end{equation}

\noindent{\bf $\bullet$ The estimate of $\|(F_1, F_2)\|_{L^2}$}\smallskip

We first deduce from \eqref{eqns-HINS-aver-8} that
\begin{equation*}\label{eqns-HINS-aver-tbaru-1}
  \begin{split}
&\|\partial_t\bar{u}^1\|_{L^2}\lesssim \|\partial_2 (\overline{{u}_{\neq}^2u^1_{\neq}})\|_{L^2}+\| {\rho}_{\neq}\|_{L^{\infty}}\|\partial_t {u}_{\neq}^1\|_{L^2} +\|\partial_1({u}^1u^1)_{\neq}\|_{L^2}+\|\partial_2({u}^2u^1)_{\neq}\|_{L^2}.
  \end{split}
\end{equation*}
Note that
\begin{equation*}\label{eqns-HINS-aver-8eee}
  \begin{split}
&\|\partial_1({u}^1u^1)_{\neq}\|_{L^2}\lesssim \|{u}^1\|_{L^2_x(L^{\infty}_y)}\|\partial_1u_{\neq}^1\|_{L^2_y(L^{\infty}_x)}\lesssim \|{u}^1\|_{L^2}^{\frac{1}{2}}\|\partial_2{u}^1\|_{L^2}^{\frac{1}{2}}
\|\partial_1^2u_{\neq}^1\|_{L^2}
  \end{split}
\end{equation*}
and
\begin{equation*}\label{eqns-HINS-aver-8fff}
  \begin{split}
\|\partial_2 (\overline{{u}_{\neq}^2u^1_{\neq}})\|_{L^2}+\|\partial_2({u}^2u^1)_{\neq}\|_{L^2}&\lesssim\|u_{\neq}^2\|_{L^{\infty}}\|\partial_2{u}^1\|_{L^2}+ \|(\bar{u}^1,\,u_{\neq}^1)\|_{L^2_x(L^{\infty}_y)}
\|\partial_1u_{\neq}^1\|_{L^2_y(L^{\infty}_x)} \\
&\lesssim \|\partial_2{u}^1\|_{L^2}\|u_{\neq}^2\|_{L^{\infty}}+ \|{u}^1\|_{L^2}^{\frac{1}{2}}\|\partial_2{u}^1\|_{L^2}^{\frac{1}{2}}\|\partial_1^2u_{\neq}\|_{L^2},
  \end{split}
\end{equation*}
then, due to \eqref{est-I-L2-3}, we obtain
\begin{equation}\label{eqns-HINS-aver-8ggg}
  \begin{split}
\|\partial_t\bar{u}^1\|_{L^2} \lesssim \|{\rho}_{\neq}\|_{L^{\infty}}\|\partial_t {u}_{\neq}^1\|_{L^2} &+\|\partial_2{u}^1\|_{L^2}\|\partial_1{u}^2_{\neq}\|_{L^2}^{\frac{1}{2}} \|\partial_1^2{u}^1_{\neq}\|_{L^2}^{\frac{1}{2}}\\
&+ \|{u}^1\|_{L^2}^{\frac{1}{2}}\|\partial_2{u}^1\|_{L^2}^{\frac{1}{2}}\|\partial_1^2u_{\neq}\|_{L^2}.
  \end{split}
\end{equation}
Thanks to \eqref{eqns-HINS-aver-10}, one has
\begin{equation*}\label{eqns-HINS-aver-11}
  \begin{split}
 \|F_1\|_{L^2}\lesssim&  \|{\rho}_{\neq }\|_{L^{\infty}}(\|\partial_t \bar{u}^1\|_{L^2}+  \|\partial_t{u}_{\neq}^1\|_{L^2})\\
&  + \bigl(\|\bar{u}^1\|_{L^2_x(L^{\infty}_y)} + \|{u}_{\neq}^1\|_{L^2_x(L^{\infty}_y)}\bigr)\|\partial_1u_{\neq}^1\|_{L^2_y(L^{\infty}_x)} + \|\partial_2 (\bar{u}^1,\, {u}_{\neq}^1)\|_{L^2}\|{u}^2_{\neq}\|_{L^{\infty}},
  \end{split}
\end{equation*}
and
\begin{equation*}\label{eqns-HINS-aver-12}
  \begin{split}
 \|F_2\|_{L^2}\lesssim &\|{\rho}_{\neq}\|_{L^{\infty}}\|\partial_t  {u}_{\neq}^2\|_{L^2} \\
&+ \|{u}^2_{\neq}\|_{L^2_x(L^{\infty}_y)}\|\partial_1{u}^1_{\neq}\|_{L^2_y(L^{\infty}_x)}+  \bigl(\|\bar{u}^1\|_{L^2_x(L^{\infty}_y)} +\|{u}_{\neq}^1\|_{L^2_x(L^{\infty}_y)} \bigr)\|\partial_1{u}^2_{\neq}\|_{L^2_y(L^{\infty}_x)},
  \end{split}
\end{equation*}
which along with \eqref{eqns-HINS-aver-8ggg} and \eqref{est-I-L2-3} implies
\begin{equation}\label{eqns-HINS-aver-13}
  \begin{split}
 \|(F_1, F_2)\|_{L^2}\lesssim&  (\|{\rho}_{\neq }\|_{L^{\infty}} + \|(u,\,\omega)\|_{L^{2}})\|(\partial_1{u}_{\neq},\,\partial_1^2u_{\neq},\,\partial_t{u}_{\neq})\|_{L^2}.
  \end{split}
\end{equation}

\noindent\textbf{$\bullet$ The exponential decay of $\|(u_{\neq}, \partial_1 u_{\neq})\|_{L^2}$}
\smallskip

Substituting \eqref{eqns-HINS-aver-13} into \eqref{est-Hins-H1-u-neq-10} yields
\begin{equation}\label{est-Hins-H1-u-neq-15}
\begin{split}
\frac{d}{dt} \mathcal{E}_{1, \neq}(t) + 8c_1\mathcal{D}_{1, \neq}
\leq C_3 \bigl(\|\rho_{\neq}\|_{L^{\infty}} + \|(u, \omega)\|_{L^{2}}\bigr)
\bigl(1 + (\|\rho_{\neq}\|_{L^{\infty}} + \|(u, \omega)\|_{L^{2}})\bigr) \mathcal{D}_{1, \neq}.
\end{split}
\end{equation}
If $\d_0$ in \eqref{condition-smallness-sol-1} is so small that $\d_0\leq \frac{1}{2}\min\bigl\{c_1C_3^{-1},\,1\bigr\},$
then we have
\begin{equation}\label{est-Hins-H1-u-neq-17}
\frac{d}{dt} \mathcal{E}_{1, \neq} (t)+ 7c_1\mathcal{D}_{1, \neq} \leq 0,
\end{equation}
which implies for $t<T^\ast$
\begin{equation}\label{est-Hins-H1-u-neq-18}
\sup_{\tau\in[0, t]} \mathcal{E}_{1, \neq}(\tau) \leq
\|(\bar{\rho}_0^{\frac12} u_{0,\neq}, \partial_1 u_{0,\neq})\|_{L^2}^2 \, e^{-7c_1 t}.
\end{equation}
Multiplying \eqref{est-Hins-H1-u-neq-17} by $e^{6c_1 t}$ gives
\begin{equation*}\label{est-Hins-H1-u-neq-17aaa}
\frac{d}{dt} \bigl(e^{6c_1 t}\mathcal{E}_{1, \neq}\bigr) + 7c_1 e^{6c_1 t}\mathcal{D}_{1, \neq} \leq 6c_1 e^{6c_1 t}\mathcal{E}_{1, \neq}.
\end{equation*}
From \eqref{est-Hins-H1-u-neq-18} we thus obtain
\begin{equation}\label{est-Hins-H1-u-neq-19}
\int_0^{T} e^{6c_1 t} \mathcal{D}_{1, \neq} \,dt \lesssim
\|(\bar{\rho}_0^{\frac12} u_{0,\neq}, \partial_1 u_{0,\neq})\|_{L^2}^2 \quad \forall\,T<T^\ast.
\end{equation}
This completes the proof of Proposition \ref{prop-HINS-H1-neq-1}.
\end{proof}

\subsection{Estimate of $\|(\omega_{\neq}, \partial_1\omega_{\neq})\|_{L^{\infty}_T(L^2)}$}

\begin{prop}[Estimate of $\|(\omega_{\neq}, \partial_1\omega_{\neq})\|_{L^2}$]
\label{prop-HINS-H2-neq-2}
{\sl Under the assumptions of Proposition \ref{prop-HINS-H1-neq-1}, if in addition
\begin{equation}\label{conditions-sol-3}
 \|\nabla\rho\|_{L^\infty_{t}(L^{\infty}\cap \dot{H}^1)} \leq 2\|\nabla\rho_0\|_{L^{\infty}\cap \dot{H}^1} \, e^{\frac{c_1}2 t} \andf  \|(u, \omega, \nabla\omega)\|_{L^\infty_{t}(L^2)} <\infty,
 \end{equation}
 for any $t<T^\ast,$ 
then there exists a positive constant $c_2$ such that
\begin{equation}\label{est-Hins-expo-2}
\begin{split}
&\sup_{t\in[0,T]} \bigl(e^{3c_{2} t} \|(\partial_1 u_{\neq}, \partial_1 \omega_{\neq})(t)\|_{L^2}^2 \bigr)
 + \int_0^{T} e^{2c_{2} t} \|(\partial_1 \omega_{\neq}, \partial_t\omega_{\neq}, \partial_1^2\omega_{\neq})(t)\|_{L^2}^2 \,dt \\
&\lesssim \|(\partial_1 u_{0\neq}, \partial_1 \omega_{0\neq})\|_{L^2}^2  
+ \bigl(1 + \|\nabla\rho_0\|_{L^{\infty}}^2\bigr) \|u_{0}\|_{H^1}^2 \bigl(1 + \|(u, \omega, \nabla\omega)\|_{L^{\infty}_T(L^{2})}^2\bigr).
\end{split}
\end{equation}}
\end{prop}

\begin{proof}
Applying the operator $\nabla^{\perp}\cdot$ to the momentum equations in \eqref{eqns-HINS-1} gives
\begin{equation}\label{eqns-HINS-omega-2}
\rho(\partial_t \omega + u \cdot \nabla \omega) - \partial_1^2 \omega
= -(\partial_t u + u \cdot \nabla u) \cdot \nabla^{\perp}\rho.
\end{equation}
Applying $\mathcal{P}_{\neq}$ to \eqref{eqns-HINS-omega-2} yields
\begin{equation}\label{eqns-HINS-oscillation-omega-1}
\bar{\rho}\bigl(\partial_t \omega_{\neq} + \bar{u}^1 \partial_1\omega_{\neq}\bigr) - \partial_1^2 \omega_{\neq} = \sum_{j=1}^4 G_j,
\end{equation}
with
\begin{equation}\label{eqns-HINS-oscillation-omega-2}
\begin{split}
G_1 &\eqdefa -\rho_{\neq} \partial_t \overline{\omega}, \quad G_2 \eqdefa -\bar{\rho}\Bigl((u_{\neq}^1 \partial_1\omega_{\neq})_{\neq}
       + u_{\neq}^2 \partial_2 \overline{\omega} + (u_{\neq}^2 \partial_2 \omega_{\neq})_{\neq}\Bigr), \\
G_3 &\eqdefa -\Bigl[\rho_{\neq} \bigl(\partial_t \omega_{\neq}
       + \partial_1 (u^1 \omega)_{\neq} + \partial_2 (u^2 \omega)_{\neq}\bigr)\Bigr]_{\neq}, \,\quad G_4 \eqdefa -\Bigl[\bigl(\partial_t u + u \cdot \nabla u\bigr)\cdot \nabla^{\perp}\rho\Bigr]_{\neq}.
\end{split}
\end{equation}
These expressions follow from the decompositions:
\begin{align*}
\rho\partial_t\omega - \overline{\rho\partial_t\omega}
&= \overline{\rho} \partial_t \overline{\omega} + \overline{\rho} \partial_t \omega_{\neq}
   + \rho_{\neq} \partial_t \overline{\omega} + \rho_{\neq} \partial_t \omega_{\neq}
   - \overline{\rho} \partial_t \overline{\omega} - \overline{\rho_{\neq} \partial_t \omega_{\neq}} \\
&= \overline{\rho} \partial_t \omega_{\neq} + \rho_{\neq} \partial_t \overline{\omega}
   + \bigl(\rho_{\neq} \partial_t \omega_{\neq}\bigr)_{\neq},
\end{align*}
and (using $\bar{u}^2 = 0$)
\begin{align*}
\rho u \cdot \nabla \omega - \overline{\rho u \cdot \nabla \omega}
&= \overline{\rho} \bar{u} \cdot \nabla \omega_{\neq}
   + \overline{\rho} u_{\neq} \cdot \nabla \omega
   + \rho_{\neq} u \cdot \nabla\omega - \overline{\rho} \overline{u \cdot \nabla \omega}
   - \overline{\rho_{\neq} \bigl( u \cdot \nabla \omega\bigr)_{\neq}} \\
&= \overline{\rho} \bar{u}^1 \partial_1 \omega_{\neq}
   + \overline{\rho} \bigl( u_{\neq} \cdot \nabla\omega_{\neq}\bigr)_{\neq} + \bigl[\rho_{\neq} \bigl( u \cdot \nabla \omega\bigr)_{\neq}\bigr]_{\neq}
   + \overline{\rho} u_{\neq}^2 \partial_2 \overline{\omega}.
\end{align*}

\noindent\textbf{$\bullet$ The energy estimate of $\om_{\neq} $} \smallskip

Taking the $L^2$ inner product of \eqref{eqns-HINS-oscillation-omega-1} with $\omega_{\neq}$ and integrating by parts gives
\begin{equation*}\label{est-H1-oscillation-omega-1}
\frac{1}{2}\frac{d}{dt}\|\bar{\rho}^{\frac12} \omega_{\neq}(t)\|_{L^2}^2 + \|\partial_1\omega_{\neq}\|_{L^2}^2
= \frac{1}{2}\int_{\Om} \omega_{\neq}^2 \partial_t\bar{\rho} \,dx
  +\sum_{j=1}^4 \int_{\Om} \omega_{\neq}  G_j \,dx.
\end{equation*}
Together with the $\bar{\rho}$ equation \eqref{eqns-HINS-aver-rho-1}, this implies
\begin{equation}\label{est-H1-oscillation-omega-2}
\begin{split}
\frac{d}{dt}\|\bar{\rho}^{\frac12} \omega_{\neq}(t)\|_{L^2}^2 + 2\|\partial_1\omega_{\neq}\|_{L^2}^2
&\lesssim \|\omega_{\neq}\|_{L^2_x(L^{\infty}_y)} \|\omega_{\neq}\|_{L^2_y(L^{\infty}_x)}
         \|\partial_2(\overline{u_{\neq}^2 \rho_{\neq}})\|_{L^2} \\
&\quad + \sum_{j=1}^4\|\omega_{\neq}\|_{L^2}  \|G_j\|_{L^2}.
\end{split}
\end{equation}
Note that
\begin{equation*}\label{est-H1-oscillation-omega-3}
\|\partial_2(\overline{u_{\neq}^2 \rho_{\neq}})\|_{L^2}
\lesssim \|\partial_1 u_{\neq}^1\|_{L^2} \|\rho_{\neq}\|_{L^{\infty}}
        + \|u_{\neq}^2\|_{L^{\infty}} \|\partial_2\rho_{\neq}\|_{L^2}.
\end{equation*}
Using \eqref{est-I-L2-3}, we obtain
\begin{equation}\label{est-H1-oscillation-omega-11}
\begin{split}
\|\partial_2(\overline{u_{\neq}^2 \rho_{\neq}})\|_{L^2}
&\lesssim \|\partial_1 u_{\neq}^1\|_{L^2} \|\rho_{\neq}\|_{L^{\infty}} + \|\partial_1 u_{\neq}\|_{L^2}^{\frac12} \|\partial_1^2 u_{\neq}\|_{L^2}^{\frac12}
         \|\partial_2\rho_{\neq}\|_{L^2}.
\end{split}
\end{equation}
Hence, from \eqref{est-H1-oscillation-omega-2}, \eqref{est-H1-oscillation-omega-11}, the fact $\|\omega_{\neq}\|_{L^2} \lesssim \|\partial_1\omega_{\neq}\|_{L^2}$, and Young's inequality, we find
\begin{equation}\label{est-H1-oscillation-omega-15}
\begin{split}
\frac{d}{dt}&\|\bar{\rho}^{\frac12} \omega_{\neq}\|_{L^2}^2 + \frac{3}{2}\|\partial_1\omega_{\neq}\|_{L^2}^2 \lesssim \sum_{j=1}^4 \|G_j\|_{L^2}^2 \\
&+ \|\omega_{\neq}\|_{L^2} \|\partial_2\omega_{\neq}\|_{L^2} \Bigl(
       \|\partial_1 u_{\neq}\|_{L^2}^2 \|\rho_{\neq}\|_{L^{\infty}}^2 + \|\partial_1 u_{\neq}\|_{L^2} \|\partial_1^2 u_{\neq}\|_{L^2}
          \|\partial_2\rho_{\neq}\|_{L^2}^2\Bigr).
\end{split}
\end{equation}

Taking the $L^2$ inner product of \eqref{eqns-HINS-oscillation-omega-1} with $\partial_t\omega_{\neq}$ gives
\begin{equation*}
\begin{split}
\frac{1}{2}&\frac{d}{dt}\|\partial_1\omega_{\neq}(t)\|_{L^2}^2 + \|\bar{\rho}^{\frac12} \partial_t\omega_{\neq}\|_{L^2}^2 = -\int_{\Om} \bar{\rho} \bar{u}^1 \partial_1\omega_{\neq} \,\partial_t\omega_{\neq} \,dx
   +\sum_{j=1}^4 \int_{\Om} \partial_t\omega_{\neq}  G_j \,dx,
\end{split}
\end{equation*}
which implies
\begin{equation*}
\begin{split}
\frac{d}{dt}&\|\partial_1\omega_{\neq}(t)\|_{L^2}^2 + 2\|\bar{\rho}^{\frac12} \partial_t\omega_{\neq}\|_{L^2}^2 \\
&\lesssim \|\bar{u}^1\|_{L^2_x(L^{\infty}_y)} \|\partial_1\omega_{\neq}\|_{L^2_y(L^{\infty}_x)}
         \|\bar{\rho}^{\frac12} \partial_t\omega_{\neq}\|_{L^2}  + \sum_{j=1}^4\|\bar{\rho}^{\frac12} \partial_t\omega_{\neq}\|_{L^2}  \|G_j\|_{L^2}.
\end{split}
\end{equation*}
Thus, thanks to Young's inequality, one has
\begin{equation}\label{est-L2-oscillation-t-omega-3}
\begin{split}
\frac{d}{dt}&\|\partial_1\omega_{\neq}(t)\|_{L^2}^2 + \|\bar{\rho}^{\frac12} \partial_t\omega_{\neq}\|_{L^2}^2 \lesssim \|\bar{u}^1\|_{L^2} \|\partial_2\bar{u}^1\|_{L^2} \|\partial_1^2\omega_{\neq}\|_{L^2}^2
         + \sum_{j=1}^4 \|G_j\|_{L^2}^2.
\end{split}
\end{equation}
From \eqref{eqns-HINS-oscillation-omega-1}, we have
\begin{equation}\label{est-L2-oscillation-t-omega-4}
\begin{split}
\|\partial_1^2\omega_{\neq}\|_{L^2}
&\lesssim \|\partial_t\omega_{\neq}\|_{L^2} + \|\bar{u}^1 \partial_1\omega_{\neq}\|_{L^2}
         + \sum_{j=1}^4 \|G_j\|_{L^2} \\
&\lesssim \|\partial_t\omega_{\neq}\|_{L^2}
         + \|\bar{u}^1\|_{L^2}^{\frac12} \|\partial_2\bar{u}^1\|_{L^2}^{\frac12}
           \|\partial_1^2\omega_{\neq}\|_{L^2}
         + \sum_{j=1}^4 \|G_j\|_{L^2}.
\end{split}
\end{equation}
Combining \eqref{est-L2-oscillation-t-omega-3} with \eqref{est-L2-oscillation-t-omega-4} and \eqref{est-H1-oscillation-omega-15}, we get from 
the fact $\|\bar{\rho}^{\frac12} \omega_{\neq}\|_{L^2} \lesssim \|\omega_{\neq}\|_{L^2} \lesssim \|\partial_1\omega_{\neq}\|_{L^2}$ that
\begin{equation}\label{est-L2-oscillation-t-omega-6}
\begin{split}
\frac{d}{dt}&\|(\bar{\rho}^{\frac12} \omega_{\neq}, \partial_1\omega_{\neq})(t)\|_{L^2}^2
           + \mathfrak{e}_2\|(\bar{\rho}^{\frac12} \omega_{\neq},\partial_1\omega_{\neq}, \partial_t\omega_{\neq}, \partial_1^2\omega_{\neq})\|_{L^2}^2 \\
&\lesssim \|\omega_{\neq}\|_{L^2} \|\partial_2\omega_{\neq}\|_{L^2} \Bigl(
          \|\partial_1 u_{\neq}\|_{L^2}^2 \|\rho_{\neq}\|_{L^{\infty}}^2+ \|\partial_1 u_{\neq}\|_{L^2} \|\partial_1^2 u_{\neq}\|_{L^2}
            \|\partial_2\rho_{\neq}\|_{L^2}^2\Bigr) \\
&\quad + \|(u, \omega)\|_{L^2} \bigl(1 + \|(u, \omega)\|_{L^2}\bigr)
          \|(\partial_t\omega_{\neq}, \partial_1^2\omega_{\neq})\|_{L^2}^2  + \bigl(1 + \|(u, \omega)\|_{L^2}\bigr) \sum_{j=1}^4 \|G_j\|_{L^2}^2.
\end{split}
\end{equation} 

\noindent\textbf{$\bullet$  The estimate of $\|G_j\|_{L^2}$.}
\vskip 0.2cm

In view of \eqref{eqns-HINS-oscillation-omega-2}, we have $\|G_1\|_{L^2} \lesssim  \|\rho_{\neq}\|_{L^{\infty}} \|\partial_t\overline{\omega}\|_{L^2}$,
\begin{equation*} 
\begin{split}
\|G_2\|_{L^2}
\lesssim& \|u_{\neq}^1 \partial_1\omega_{\neq}\|_{L^2} + \|u_{\neq}^2 \partial_2\omega\|_{L^2} 
\lesssim  \|u_{\neq}^1\|_{L^{\infty}} \|\partial_1\omega_{\neq}\|_{L^2}
        + \|u_{\neq}^2\|_{L^{\infty}} \|\partial_2\omega\|_{L^2},
\end{split}
\end{equation*}
\begin{equation*}\label{est-H1-oscillation-omega-6}
\begin{split}
\|G_3\|_{L^2}
&\lesssim \|\rho_{\neq}\|_{L^{\infty}} \bigl(\|\partial_t \omega_{\neq}\|_{L^2}
        + \|\bar{u}^1 \partial_1\omega_{\neq} + (u_{\neq}^1 \partial_1\omega_{\neq})_{\neq} + u_{\neq}^2 \partial_2 \overline{\omega} + (u_{\neq}^2 \partial_2 \omega_{\neq})_{\neq}\|_{L^2}\bigr) \\
&\lesssim \|\rho_{\neq}\|_{L^{\infty}} \|\partial_t \omega_{\neq}\|_{L^2}
        + \|\rho_{\neq}\|_{L^{\infty}} \|(\bar{u}^1, u_{\neq}^1)\|_{L^{\infty}} \|\partial_1\omega_{\neq}\|_{L^2}  + \|u_{\neq}^2\|_{L^{\infty}} \|\partial_2\omega\|_{L^2},
\end{split}
\end{equation*}
and $\|G_4\|_{L^2} \lesssim \|\nabla\rho\|_{L^{\infty}} \|(\partial_t u + u \cdot \nabla u)\|_{L^2}$.

For $\partial_t \overline{\omega}$ in $G_1$, averaging \eqref{eqns-HINS-omega-2} with respect to $x_1 \in \mathbb{T}$ gives
\begin{equation*}\label{eqns-HINS-aver-omega-1}
\begin{split}
\bar{\rho}\bigl(\partial_t \overline{\omega} + \partial_2\overline{(u_{\neq}^2 \omega_{\neq})}\bigr)
&= -\overline{\rho_{\neq}(\partial_t \omega_{\neq} + \partial_2(u^2 \omega)_{\neq})} \\
&\quad - \overline{(\partial_t u + u \cdot \nabla u) \cdot \nabla^{\perp}\rho},
\end{split}
\end{equation*}
which implies
\begin{equation}\label{eqns-HINS-aver-omega-2}
\begin{split}
\|\partial_t \overline{\omega}\|_{L^2}
&\lesssim \|(\omega_{\neq} \partial_2 u_{\neq}^2 + u_{\neq}^2 \partial_2 \omega_{\neq})\|_{L^2}
        + \|\partial_2(u^2 \omega)\|_{L^2} \\
&\quad + \|\rho_{\neq}\|_{L^{\infty}} \|\partial_t \omega_{\neq}\|_{L^2}
 + \|\nabla\rho\|_{L^{\infty}} \|(\partial_t u + u \cdot \nabla u)\|_{L^2}.
\end{split}
\end{equation}
Note that
\begin{equation*}\label{eqns-HINS-aver-omega-3}
\begin{split}
\|(\omega_{\neq} \partial_2 u_{\neq}^2 + u_{\neq}^2 \partial_2 \omega_{\neq})\|_{L^2}
   + \|\partial_2(u^2 \omega)\|_{L^2}
&\lesssim \|u_{\neq}^2 \partial_2\omega\|_{L^2} + \|\omega \partial_1 u_{\neq}^1\|_{L^2} \\
&\lesssim \|u_{\neq}^2\|_{L^{\infty}} \|\partial_2\omega\|_{L^2}
        + \|\omega\|_{L^{2}_x(L^{\infty}_y)} \|\partial_1 u_{\neq}^1\|_{L^{2}_y(L^{\infty}_x)},
\end{split}
\end{equation*}
and
\begin{equation*}\label{H1-hins-u-3}
\begin{split}
\|(\partial_t u + u \cdot \nabla u)\|_{L^2}
&\lesssim \|\partial_t \bar{u}^1\|_{L^2} + \|\partial_t u_{\neq}^1\|_{L^2} + \|u^1\|_{L^{\infty}} \|\partial_1 u\|_{L^2} + \|u^2 \partial_2 u\|_{L^2} \\
&\lesssim \|\partial_t \bar{u}^1\|_{L^2} + \|\partial_t u_{\neq}^1\|_{L^2}  + \bigl(\|\bar{u}^1\|_{L^{\infty}} + \|u_{\neq}^1\|_{L^{\infty}}\bigr) \|\partial_1 u_{\neq}\|_{L^2} \\
&\quad + \|u_{\neq}^2\|_{L^{\infty}} \|\partial_2 u\|_{L^2}.
\end{split}
\end{equation*}
Using \eqref{est-I-L2-3} and \eqref{eqns-HINS-aver-8ggg}, we obtain
\begin{equation*}\label{eqns-HINS-aver-omega-4aaa}
\begin{split}
&\|\omega_{\neq} \partial_2 u_{\neq}^2 + u_{\neq}^2 \partial_2 \omega_{\neq}\|_{L^2}
   + \|\partial_2(u^2 \omega)\|_{L^2} \lesssim (\|\omega\|_{L^2} + \|\partial_2\omega\|_{L^2})
        \|(\partial_1 u_{\neq}, \partial_1^2 u_{\neq})\|_{L^2},
\end{split}
\end{equation*}
and
\begin{equation}\label{H1-hins-u-4}
\begin{split}
\|(\partial_t u + u \cdot \nabla u)\|_{L^2}
&\lesssim \|\partial_t u_{\neq}^1\|_{L^2} + \|(u, \omega)\|_{L^2} \|(\partial_1 u_{\neq}, \partial_1^2 u_{\neq})\|_{L^2} \\
&\quad + \|(u, \omega)\|_{L^2}^{\frac12}
        \|\partial_1\omega_{\neq}\|_{L^2}^{\frac12} \|\partial_1 u_{\neq}\|_{L^2}.
\end{split}
\end{equation}

Substituting these estimates into \eqref{eqns-HINS-aver-omega-2} yields
\begin{equation}\label{eqns-HINS-aver-omega-5}
\begin{split}
\|\partial_t \overline{\omega}\|_{L^2}
&\lesssim \|\rho_{\neq}\|_{L^{\infty}} \|\partial_t \omega_{\neq}\|_{L^2}
 + \|\nabla\rho\|_{L^{\infty}} \|(u, \omega)\|_{L^2}^{\frac12}
        \|\partial_1\omega_{\neq}\|_{L^2}^{\frac12} \|\partial_1 u_{\neq}\|_{L^2} \\
&\quad + \bigl(\|\nabla\rho\|_{L^{\infty}}(1 + \|(u, \omega)\|_{L^2})
        + \|(\omega, \partial_2\omega)\|_{L^2}\bigr)  \|(\partial_1 u_{\neq}, \partial_1^2 u_{\neq}, \partial_t u_{\neq})\|_{L^2}.
\end{split}
\end{equation}

Using \eqref{eqns-HINS-aver-omega-5}, \eqref{est-I-L2-3}, \eqref{eqns-HINS-aver-8ggg}, and \eqref{H1-hins-u-4}, we obtain
\begin{equation*}\label{est-H1-oscillation-omega-10}
\begin{split}
\|G_1\|_{L^2} + \|G_4\|_{L^2}
&\lesssim \|\rho_{\neq}\|_{L^{\infty}} \|\partial_t \omega_{\neq}\|_{L^2}
+ \|\nabla\rho\|_{L^{\infty}} \|(u, \omega)\|_{L^2}^{\frac12}
        \|\partial_1\omega_{\neq}\|_{L^2}^{\frac12} \|\partial_1 u_{\neq}\|_{L^2} \\
&\quad + \bigl(\|\nabla\rho\|_{L^{\infty}}(1 + \|(u, \omega)\|_{L^2})
        + \|(\omega, \partial_2\omega)\|_{L^2}\bigr) \\
&\qquad\quad \times \|(\partial_1 u_{\neq}, \partial_1^2 u_{\neq}, \partial_t u_{\neq})\|_{L^2},
\end{split}
\end{equation*}
and
\begin{equation*}\label{est-H1-oscillation-omega-8}
\begin{split}
\|G_2\|_{L^2} + \|G_3\|_{L^2}
&\lesssim \|\partial_2\omega\|_{L^2} \|(\partial_1 u_{\neq}, \partial_1^2 u_{\neq})\|_{L^2} \\
&\quad + (\|\rho_{\neq}\|_{L^{\infty}} + \|(u, \omega)\|_{L^2})
        \|(\partial_1\omega_{\neq}, \partial_1^2\omega_{\neq}, \partial_t \omega_{\neq})\|_{L^2}.
\end{split}
\end{equation*}
Thus,
\begin{equation}\label{H1-hins-est-H-1}
\begin{split}
\sum_{j=1}^4 \|G_j\|_{L^2}
&\lesssim (\|\rho_{\neq}\|_{L^{\infty}} + \|(u, \omega)\|_{L^2})
        \|(\partial_1\omega_{\neq}, \partial_1^2\omega_{\neq}, \partial_t \omega_{\neq})\|_{L^2} \\
&\quad + \bigl(\|\nabla\rho\|_{L^{\infty}}(1 + \|(u, \omega)\|_{L^2})
        + \|(\omega, \partial_2\omega)\|_{L^2}\bigr)  \|(\partial_1 u_{\neq}, \partial_1^2 u_{\neq}, \partial_t u_{\neq})\|_{L^2} \\
&\quad + \|\nabla\rho\|_{L^{\infty}} \|(u, \omega)\|_{L^2}^{\frac12}
        \|\partial_1\omega_{\neq}\|_{L^2}^{\frac12} \|\partial_1 u_{\neq}\|_{L^2}.
\end{split}
\end{equation}

Hence, inserting \eqref{H1-hins-est-H-1} into \eqref{est-L2-oscillation-t-omega-6} yields
\begin{equation}\label{est-L2-oscillation-t-omega-7}
\begin{split}
\frac{d}{dt}&\|(\bar{\rho}^{\frac12} \omega_{\neq}, \partial_1\omega_{\neq})(t)\|_{L^2}^2
           + \mathfrak{e}_2\|(\bar{\rho}^{\frac12} \omega_{\neq}, \partial_1\omega_{\neq},
                    \partial_t\omega_{\neq}, \partial_1^2\omega_{\neq})\|_{L^2}^2 \\
&\lesssim (1 + \|(u, \omega)\|_{L^2})
        (\|\rho_{\neq}\|_{L^{\infty}}^2 + \|(u, \omega)\|_{L^2}^2)
        \|(\partial_1\omega_{\neq}, \partial_t\omega_{\neq}, \partial_1^2\omega_{\neq})\|_{L^2}^2 \\
&\quad + (1 + \|(u, \omega)\|_{L^2})
        \|(u, \omega, \nabla\omega)\|_{L^2}^2 \mathcal{D}_{1, \neq} \\
&\quad + \bigl(1 + \|(u, \omega)\|_{L^2}^3 + \|(u, \omega, \nabla\omega)\|_{L^2}^2  + \|(u, \omega)\|_{L^2} \|(u, \omega, \nabla\omega)\|_{L^2}^2\bigr)
        \|\nabla\rho\|_{L^{\infty}}^2 \mathcal{D}_{1, \neq},
\end{split}
\end{equation}
where $\mathcal{D}_{1, \neq}$ is given by \eqref{def-H1-u-neq-1}.

Define
\begin{equation}\label{est-L2-oscillation-t-omega-8}
\begin{split}
\mathcal{E}_{2, \neq} (t)&\eqdefa \|(\bar{\rho}^{\frac12} (u_{\neq}, \omega_{\neq}),
                                 \partial_1 u_{\neq}, \partial_1 \omega_{\neq})(t)\|_{L^2}^2, \\
\mathcal{D}_{2, \neq}(t) &\eqdefa \|(\bar{\rho}^{\frac12} (u_{\neq}, \omega_{\neq}),
                                 \partial_1 u_{\neq}, \partial_1 \omega_{\neq},
                                 \partial_t u_{\neq}, \partial_1^2 u_{\neq},
                                 \nabla \Pi_{\neq}, \partial_t\omega_{\neq},
                                 \partial_1^2\omega_{\neq})(t)\|_{L^2}^2.
\end{split}
\end{equation}
Combining \eqref{est-L2-oscillation-t-omega-7} with \eqref{est-Hins-H1-u-neq-15} gives
\begin{equation*}
\begin{split}
&\frac{d}{dt} \mathcal{E}_{2, \neq}(t)+ 2\mathfrak{e}_3\mathcal{D}_{2, \neq} \\
&\leq C_4 \bigl(1 + \|\rho_{\neq}\|_{L^{\infty}} + \|(u, \omega)\|_{L^2}\bigr)
        \bigl(\|\rho_{\neq}\|_{L^{\infty}} + \|(u, \omega)\|_{L^2}\bigr)
        \|(\partial_1\omega_{\neq}, \partial_t\omega_{\neq}, \partial_1^2\omega_{\neq})\|_{L^2}^2 \\
&\quad + C_4  \bigl(1 + \|(u, \omega)\|_{L^2}\bigr)
        \|(u, \omega, \nabla\omega)\|_{L^2}^2 \mathcal{D}_{1, \neq} \\
&\quad + C_4  \bigl(1 + \|(u, \omega)\|_{L^2}^3
        + \|(u, \omega, \nabla\omega)\|_{L^2}^2 + \|(u, \omega)\|_{L^2} \|(u, \omega, \nabla\omega)\|_{L^2}^2\bigr)
        \|\nabla\rho\|_{L^{\infty}}^2 \mathcal{D}_{1, \neq}. 
\end{split}
\end{equation*}

\noindent\textbf{$\bullet$ The exponential decay of $\|(\omega_{\neq}, \partial_1 \omega_{\neq})\|_{L^2}$}
\smallskip

If  \eqref{condition-smallness-sol-1}, \eqref{conditions-sol-3} and 
 \begin{equation}\label{assumption-delta0-11}
  C_4  (1 + \delta_0)\delta_0 \leq \mathfrak{e}_3
 \end{equation}
 hold,  then
\begin{equation}\label{est-L2-oscillation-t-omega-9}
\begin{split}
\frac{d}{dt} \mathcal{E}_{2, \neq}(t) + \mathfrak{e}_3\mathcal{D}_{2, \neq}
&\leq C_{5} \|(u, \omega, \nabla\omega)\|_{L^2}^2 \mathcal{D}_{1, \neq} \\
&\quad + C_{5} (1 + \|(u, \omega, \nabla\omega)\|_{L^2})^2
        \|\nabla\rho_0\|_{L^{\infty}}^2 \, e^{c_1 t} \mathcal{D}_{1, \neq}.
\end{split}
\end{equation}
Since $\mathcal{E}_{2, \neq} (t)\lesssim \mathcal{D}_{2, \neq}(t)$, the inequality \eqref{est-L2-oscillation-t-omega-9} implies that, there exists a positive constant $c_2 \leq c_1$ such that
\begin{equation}\label{est-L2-oscillation-t-omega-11}
\frac{d}{dt} \mathcal{E}_{2, \neq}(t) + 3c_{2}\mathcal{E}_{2, \neq} \leq \frak{f}(t)e^{c_1 t} \mathcal{D}_{1, \neq},
\end{equation}
with
\begin{equation*}\label{est-L2-oscillation-t-omega-10-cont-2}
\begin{split}
\frak{f}(t)
&\eqdefa C_{5} \bigl(\|(u, \omega, \nabla\omega)\|_{L^{\infty}_t(L^{2})} + (1 + \|(u, \omega, \nabla\omega)\|_{L^{\infty}_t(L^{2})}) \|\nabla\rho_0\|_{L^{\infty}}\bigr)^2,
\end{split}
\end{equation*}
which implies
\begin{equation*}\label{est-L2-oscillation-t-omega-12}
\frac{d}{dt}\bigl(e^{3c_2 t}\mathcal{E}_{2, \neq}(t)\bigr)
\leq  \frak{f}(t)e^{(c_1 + 3c_{2}) t} \mathcal{D}_{1, \neq}.
\end{equation*}
Together with \eqref{est-Hins-H1-u-neq-19} and \eqref{conditions-sol-3}, this yields
\begin{equation}\label{est-L2-oscillation-t-omega-14}
\begin{split}
\mathcal{E}_{2, \neq}(t)
\leq &e^{-3c_{2} t} \Bigl(\mathcal{E}_{2,\neq}(0)
        + \frak{f}(t) \int_{0}^{t} e^{4c_{1} \tau} \mathcal{D}_{1, \neq}(\tau) \,d\tau\Bigr)  \leq  e^{-3c_{2} t} \widetilde{\mathcal{E}}_{2,\neq, t},
\end{split}
\end{equation}
where
\begin{equation}\label{est-L2-oscillation-t-omega-15}
\begin{split}
\widetilde{\mathcal{E}}_{2,\neq, t}
&\eqdefa \mathcal{E}_{2,\neq, 0}
      + \bigl(1 + \|\nabla\rho_0\|_{L^{\infty}}^2\bigr) \|u_{0}\|_{H^1}^2\bigl(1+
        \|(u, \omega, \nabla\omega)\|_{L^{\infty}_t(L^{2})}^2\bigr).
\end{split}
\end{equation}
From \eqref{est-L2-oscillation-t-omega-11} and \eqref{est-L2-oscillation-t-omega-14}, we get, by a similar derivation of \eqref{est-Hins-H1-u-neq-19} that  for all $T <T^\ast$,
\begin{equation*}\label{est-L2-oscillation-t-omega-16}
\int_0^{T} e^{2c_{2} t} \mathcal{D}_{2, \neq}(t) \,dt \lesssim \widetilde{\mathcal{E}}_{2,\neq, T},
\end{equation*}
which together with \eqref{est-L2-oscillation-t-omega-14} ensures \eqref{est-Hins-expo-2}.
We thus complete the proof of  Proposition \ref{prop-HINS-H2-neq-2}.
\end{proof}

\subsection{Proof of Theorem \ref{thm-GWP-HINS}}

\begin{proof}[Proof of Theorem \ref{thm-GWP-HINS}]
To prove the global well-posedness of system \eqref{eqns-HINS-1}, it suffices to establish the necessary {\it a priori } estimates for solutions of \eqref{eqns-HINS-1}. Toward this, we assume that \eqref{assumption-delta0-11} holds and define
\beq \label{S4eq1}
T^\star_1\eqdefa \sup\bigl\{T<T^\ast, \ \ \mbox{so that there hold} \ \  \eqref{condition-smallness-sol-1} \andf \eqref{conditions-sol-3}\ \bigr\}.
\eeq

\noindent\textbf{$\bullet$ The estimate of $\|u\|_{L^2}$.}\smallskip

Taking the $L^2$ inner product of the $u$ equations in \eqref{eqns-HINS-1} with $u$ and integrating by parts yields
\begin{equation}\label{L2-hins-u-1aaa}
\frac{1}{2}\frac{d}{dt}\|\rho^{\frac12} u(t)\|_{L^2}^2 + \|\partial_1 u\|_{L^2}^2 = 0,
\end{equation}
which implies
\begin{equation*}\label{L2-hins-u-1}
\|\rho^{\frac12} u\|_{L^\infty_t(L^2)}^2
+ 2\int_0^t \|\partial_1 u(\tau)\|_{L^2}^2 \,d\tau \leq  \|\rho_0^{\frac12} u_0\|_{L^2}^2.
\end{equation*}

\noindent\textbf{$\bullet$ The estimate of $\|\partial_1 u\|_{L^2}$.} \smallskip

Taking the $L^2$ inner product of the $u$ equations in \eqref{eqns-HINS-1} with $\partial_t u$ and integrating by parts gives
\begin{equation*}\label{L2-hins-p1u-1}
\frac{1}{2}\frac{d}{dt}\|\partial_1 u(t)\|_{L^2}^2 + \|\rho^{\frac12} \partial_t u\|_{L^2}^2
= -\int_{\Om} \r u \cdot \nabla u \cdot \partial_t u \,dx,
\end{equation*}
so that
\begin{equation*} 
\f12\frac{d}{dt}\|\partial_1 u(t)\|_{L^2}^2 + \|\rho^{\frac12} \partial_t u\|_{L^2}^2
\lesssim \|u \cdot \nabla u\|_{L^2} \|\partial_t u\|_{L^2}.
\end{equation*}
Using the $u$ equations in \eqref{eqns-HINS-1} again, we obtain that, for some positive constant $\mathfrak{e}_4$,
\begin{equation*} 
\frac{d}{dt}\|\partial_1 u(t)\|_{L^2}^2 + \mathfrak{e}_4\|(\partial_t u, \partial_1 u, \nabla \Pi)\|_{L^2}^2
\lesssim \|u \cdot \nabla u\|_{L^2}^2.
\end{equation*}
Note that
\begin{equation*} 
\begin{split}
\|u \cdot \nabla u\|_{L^2}
&\lesssim \|u^1\|_{L^{2}_x(L^{\infty}_y)} \|\partial_1 u_{\neq}\|_{L^{2}_y(L^{\infty}_x)}
        + \|u_{\neq}^2\|_{L^{\infty}} \|\partial_2 u\|_{L^2} \\
&\lesssim \|u\|_{L^{2}}^{\frac12} \|\omega\|_{L^{2}}^{\frac12} \|\partial_1^2 u_{\neq}\|_{L^{2}}
        + \|\partial_2 u\|_{L^2} \|(\partial_1 u_{\neq}, \partial_1^2 u_{\neq})\|_{L^{2}}.
\end{split}
\end{equation*}
Thus, for $\mathcal{D}_{1, \neq}$ given by \eqref{def-H1-u-neq-1}, we obtain
\begin{equation}\label{L2-hins-p1u-5}
\frac{d}{dt}\|\partial_1 u(t)\|_{L^2}^2 +  \mathfrak{e}_4 \|(\partial_t u, \partial_1 u, \nabla \Pi)\|_{L^2}^2
\lesssim \|(u, \omega)\|_{L^2}^2 \mathcal{D}_{1, \neq}.
\end{equation}

\noindent\textbf{$\bullet$ The  estimate of $\|\omega\|_{L^2}$.}\smallskip

Taking the $L^2$ inner product of the $\omega$ equation \eqref{eqns-HINS-omega-2} with $\omega$ and integrating by parts gives
\begin{equation*}\label{H1-hins-u-1}
\frac{1}{2}\frac{d}{dt}\|\rho^{\frac12} \omega(t)\|_{L^2}^2 + \|\partial_1\omega\|_{L^2}^2
= -\int_{\Om} \omega \, \bigl(\partial_t u + u \cdot \nabla u\bigr) \cdot \nabla^{\perp}\rho \,dx,
\end{equation*}
so that
\begin{equation*}\label{H1-hins-u-2}
\frac{d}{dt}\|\rho^{\frac12} \omega(t)\|_{L^2}^2 + 2\|\partial_1\omega\|_{L^2}^2
\lesssim \|\omega\|_{L^2} \|(\partial_t u + u \cdot \nabla u)\|_{L^2} \|\nabla\rho\|_{L^{\infty}}.
\end{equation*}
Using \eqref{H1-hins-u-4}, we find
\begin{equation*}\label{H1-hins-u-5}
\begin{split}
\frac{d}{dt}\|\rho^{\frac12} \omega(t)\|_{L^2}^2 + 2\|\partial_1\omega\|_{L^2}^2
&\lesssim \|\nabla\rho\|_{L^{\infty}} \|\omega\|_{L^2} \|\partial_t u_{\neq}^1\|_{L^2} \\
&\quad + \|\nabla\rho\|_{L^{\infty}} \|(u, \omega)\|_{L^2}^2 \|(\partial_1 u_{\neq}, \partial_1^2 u_{\neq})\|_{L^2} \\
&\quad + \|\nabla\rho\|_{L^{\infty}} \|(u, \omega)\|_{L^2}^{\frac32}
        \|\partial_1\omega\|_{L^2}^{\frac12} \|\partial_1 u_{\neq}\|_{L^2}.
\end{split}
\end{equation*}
Applying Young's inequality yields
\begin{equation}\label{H1-hins-u-6}
\begin{split}
\frac{d}{dt}\|\rho^{\frac12} \omega(t)\|_{L^2}^2 + \frac{3}{2}\|\partial_1\omega\|_{L^2}^2
&\lesssim \|\nabla\rho\|_{L^{\infty}} \|\omega\|_{L^2} \|\partial_t u_{\neq}^1\|_{L^2} \\
&\quad + \|\nabla\rho\|_{L^{\infty}} \|(u, \omega)\|_{L^2}^2 \|(\partial_1 u_{\neq}, \partial_1^2 u_{\neq})\|_{L^2} \\
&\quad + \|\nabla\rho\|_{L^{\infty}}^{\frac43} \|(u, \omega)\|_{L^2}^2 \|\partial_1 u_{\neq}\|_{L^2}^{\frac43}.
\end{split}
\end{equation}

\noindent\textbf{$\bullet$ The  Estimate of $\|(u, \omega)\|_{L^2}$.}\smallskip

Combining \eqref{H1-hins-u-6} with \eqref{L2-hins-u-1aaa} and \eqref{L2-hins-p1u-5} gives
\begin{equation}\label{H1-hins-u-7}
\begin{split}
\frac{d}{dt}&\|(\rho^{\frac12} u, \rho^{\frac12} \omega, \partial_1 u)(t)\|_{L^2}^2
           +  \mathfrak{e}_5\|(\partial_1 u, \partial_1\omega, \partial_1^2 u, \partial_t u)\|_{L^2}^2 \\
&\lesssim \|\nabla\rho\|_{L^{\infty}} \|(u, \omega)\|_{L^2}
          \bigl(1 + \|(u, \omega)\|_{L^2}\bigr) \mathcal{D}_{1, \neq}^{\frac12}  + \|(u, \omega)\|_{L^2}^2\bigl( \mathcal{D}_{1, \neq} + \|\nabla\rho\|_{L^{\infty}}^{\frac43} \mathcal{E}_{1, \neq}^{\frac23}\bigr).
\end{split}
\end{equation}
for $\mathcal{E}_{1, \neq}$ and $\mathcal{D}_{1, \neq}^{\frac12} $ given by \eqref{def-H1-u-neq-1}.

Applying Gronwall's inequality yields
\begin{equation}\label{H1-hins-u-8a}
\begin{split}
&\|(\rho^{\frac12} u, \rho^{\frac12} \omega, \partial_1 u)\|_{L^\infty_t(L^2)}^2
           + \mathfrak{e}_5\|(\partial_1 u, \partial_1\omega, \partial_1^2 u, \partial_t u)\|_{L^2_t(L^2)}^2\\
&\leq\Bigl( \|(\rho_0^{\frac12} (u_0, \omega_0), \p_1u_0)\|_{L^2}^2 + C \int_0^t\|\nabla\rho\|_{L^{\infty}} \mathcal{D}_{1, \neq}^{\frac12} \,d\tau\Bigr) \\
&\qquad\times\exp\Bigl(C \int_0^t\bigl( \|\nabla\rho\|_{L^{\infty}} \mathcal{D}_{1, \neq}^{\frac12}
     +\mathcal{D}_{1, \neq} 
             + \|\nabla\rho\|_{L^{\infty}}^{\frac43} \mathcal{E}_{1, \neq}^{\frac23}\bigr) \,d\tau\Bigr).
\end{split}
\end{equation}
Using \eqref{est-Hins-H1-u-neq-18}, \eqref{est-Hins-H1-u-neq-19}, and assumption \eqref{conditions-sol-3}, we get
\begin{equation*} 
\begin{split}
&\|(\rho^{\frac12} u, \rho^{\frac12} \omega, \partial_1 u)\|_{L^\infty_t(L^2)}^2
           +  \mathfrak{e}_5\|(\partial_1 u, \partial_1\omega, \partial_1^2 u, \partial_t u)\|_{L^2_t(L^2)}^2\\
         &  \leq C \|u_0\|_{H^1} (\|u_0\|_{H^1} + \|\nabla\rho_0\|_{L^{\infty}} )\exp\Bigl(C(1 + \|\nabla\rho_0\|_{L^{\infty}})^2 \|u_0\|_{H^1}^{2} \Bigr).
         \end{split}
\end{equation*}
Set
\begin{equation}\label{small-init-u0-1}
\epsilon_0 \eqdefa \|u_{0}\|_{H^1} (1 + \|\nabla\rho_0\|_{L^{\infty}\cap\dot{H}^1})
\end{equation}
and assume
\begin{equation}\label{small-init-u0-2}
\epsilon_0 \leq 1,
\end{equation}
so one has
\begin{equation}\label{H1-hins-u-9}
\begin{split}
&\|(u, \omega)\|_{L^\infty_t(L^2)}^2
           +  \|(\partial_1 u, \partial_1\omega, \partial_1^2 u, \partial_t u)\|_{L^2_t(L^2)}^2 \leq C_6 \epsilon_0.
         \end{split}
\end{equation}

\noindent\textbf{$\bullet$ The estimate of $\|\nabla \omega\|_{L^2}$.}\smallskip

Taking the $L^2$ inner product of the $\omega$ equation \eqref{eqns-HINS-omega-2} with $\partial_t\omega$ and integrating by parts gives
\begin{equation*}\label{H2-hins-u-1}
\begin{split}
\frac{1}{2}\frac{d}{dt}\|\partial_1\omega(t)\|_{L^2}^2 + \|\rho^{\frac12} \partial_t\omega\|_{L^2}^2
&= -\int_{\Om} \partial_t\omega \, \bigl(\partial_t u + u \cdot \nabla u\bigr) \cdot \nabla^{\perp}\rho \,dx  - \int_{\Om} \partial_t\omega \, \rho u \cdot \nabla \omega \,dx.
\end{split}
\end{equation*}
Together with \eqref{eqns-HINS-omega-2}, this yields
\begin{equation*}\label{H2-hins-u-2}
\begin{split}
\frac{d}{dt}\|\partial_1\omega(t)\|_{L^2}^2 &+ \mathfrak{e}_6\|(\partial_t\omega, \partial_1^2\omega)\|_{L^2}^2
\lesssim \|\partial_t\omega\|_{L^2} \|(\partial_t u + u \cdot \nabla u)\|_{L^2} \|\nabla^{\perp}\rho\|_{L^{\infty}} \\
&\quad + \|\partial_t\omega\|_{L^2} \|u \cdot \nabla \omega\|_{L^2} + \|u \cdot \nabla \omega\|_{L^2}^2 + \|(\partial_t u + u \cdot \nabla u)\|_{L^2}^2 \|\nabla^{\perp}\rho\|_{L^{\infty}}^2.
\end{split}
\end{equation*}
Applying Young's inequality leads to
\begin{equation}\label{H2-hins-u-3}
\begin{split}
\frac{d}{dt}&\|\partial_1\omega(t)\|_{L^2}^2 + \f{\mathfrak{e}_6}2\|(\partial_t\omega, \partial_1^2\omega)\|_{L^2}^2 \\
&\lesssim \|(\partial_t u + u \cdot \nabla u)\|_{L^2}^2 \|\nabla\rho\|_{L^{\infty}}^2 + \|u \cdot \nabla \omega\|_{L^2}^2.
\end{split}
\end{equation}

Applying $\partial_2$ to the $\omega$ equation \eqref{eqns-HINS-omega-2} yields
\begin{equation}\label{eqns-HINS-partial2-omega-1}
\begin{split}
&\rho\bigl(\partial_t \partial_2\omega + u \cdot \nabla \partial_2\omega\bigr) - \partial_1^2\partial_2\omega= -\partial_2\bigl(\partial_t u + u \cdot \nabla u\bigr) \cdot \nabla^{\perp}\rho \\
&\quad - \bigl(\partial_t u + u \cdot \nabla u\bigr) \cdot \nabla^{\perp}\partial_2\rho  - \rho \partial_2 u \cdot \nabla \omega - \partial_2\rho \bigl(\partial_t\omega + u \cdot \nabla \omega\bigr).
\end{split}
\end{equation}
Taking the $L^2$ inner product of \eqref{eqns-HINS-partial2-omega-1} with $\partial_2\omega$ and integrating by parts gives
\begin{equation}\label{est-partial2-omega-2}
\frac{1}{2}\frac{d}{dt}\|\rho^{\frac12} \partial_2\omega(t)\|_{L^2}^2 + \|\partial_1\partial_2\omega\|_{L^2}^2 = \sum_{j=1}^4 IV_j,
\end{equation}
with
\begin{align*}
IV_1 &\eqdefa -\int_{\Om} \partial_2\omega \, \partial_2\bigl(\partial_t u + u \cdot \nabla u\bigr) \cdot \nabla^{\perp}\rho \,dx, \quad IV_2 \eqdefa -\int_{\Om} \partial_2\omega \, \rho \partial_2 u \cdot \nabla \omega \,dx, \\
IV_3 &\eqdefa -\int_{\Om} \partial_2\omega \, \bigl(\partial_t u + u \cdot \nabla u\bigr) \cdot \nabla^{\perp}\partial_2\rho \,dx, \quad IV_4 \eqdefa -\int_{\Om} \partial_2\omega \, \partial_2\rho \bigl(\partial_t\omega + u \cdot \nabla \omega\bigr) \,dx.
\end{align*}
Note that
\begin{equation*}\label{est-partial2-omega-3}
\begin{split}
|IV_1| + |IV_4| &\lesssim \|\partial_2\omega\|_{L^2} \bigl(\|\partial_t \omega\|_{L^2} + \|u \cdot \nabla\partial_2 u\|_{L^2} \\
&\quad + \|(\partial_2 u \cdot \nabla) u\|_{L^2} + \|u \cdot \nabla \omega\|_{L^2}\bigr) \|\nabla\rho\|_{L^{\infty}},
\end{split}
\end{equation*}
\begin{equation*}\label{est-partial2-omega-4}
\begin{split}
|IV_3| &\lesssim \|\partial_2\omega\|_{L^2_y(L^{\infty}_x)} \|(\partial_t u + u \cdot \nabla u)\|_{L^2_x (L^{\infty}_y)} \|\nabla\partial_2\rho\|_{L^2} \\
&\lesssim (\|\partial_2\omega\|_{L^2} + \|\partial_1\partial_2\omega\|_{L^2}) \bigl(\|(\partial_t u + u \cdot \nabla u)\|_{L^2} \\
&\qquad + \|\partial_t \omega\|_{L^2} + \|u \cdot \nabla \partial_2 u\|_{L^2} + \|(\partial_2 u \cdot \nabla) u\|_{L^2}\bigr) \|\nabla^2\rho\|_{L^2},
\end{split}
\end{equation*}
and
\begin{equation*}\label{est-partial2-omega-5}
\begin{split}
|IV_2| &\lesssim \|\partial_2\omega\|_{L^2_y(L^{\infty}_x)} \|\partial_2 u \cdot \nabla
 \omega\|_{L^1_x(L^{2}_y)} \\
&\lesssim \bigl(\|\partial_2\omega\|_{L^2} + \|\partial_1\partial_2\omega\|_{L^2}\bigr) \|\partial_2 u \cdot \nabla \omega\|_{L^1_x(L^{2}_y)}.
\end{split}
\end{equation*}
Then from \eqref{H2-hins-u-3}, \eqref{est-partial2-omega-2} and Young's inequality, we obtain
\begin{equation}\label{est-partial2-omega-5aaa}
\begin{split}
&\frac{d}{dt}\|(\partial_1\omega, \rho^{\frac12} \partial_2\omega)(t)\|_{L^2}^2
           + \mathfrak{e}_7 \|(\partial_t\omega, \partial_1\omega, \partial_1\partial_2\omega)\|_{L^2}^2 \\
&\lesssim \|\nabla\rho\|_{L^{\infty}\cap\dot{H}^1} \|\partial_2\omega\|_{L^2} \, \mathfrak{m}(t)
         + \mathfrak{m}^2(t) \|\nabla\rho\|_{L^{\infty}\cap\dot{H}^1}^2 \\
&\quad + \|\partial_2\omega\|_{L^2} \|\partial_2 u \cdot \nabla \omega\|_{L^1_x(L^{2}_y)}
         + \|\partial_2 u \cdot \nabla \omega\|_{L^1_x(L^{2}_y)}^2+ \|u \cdot \nabla \omega\|_{L^2}^2,
\end{split}
\end{equation}
where
\begin{equation}\label{est-partial2-omega-5bbb}
\begin{split}
\mathfrak{m}(t) &\eqdefa \|\partial_t \omega\|_{L^2} + \|u \cdot \nabla \partial_2 u\|_{L^2}
                 + \|\partial_2 u \cdot \nabla u\|_{L^2}  + \|u \cdot \nabla \omega\|_{L^2}
                 + \|\partial_t u + u \cdot \nabla u\|_{L^2}.
\end{split}
\end{equation}

\smallskip
\noindent\textbf{$\bullet$ The estimates of $\mathfrak{m}(t)$.}
\smallskip

For $\|u \cdot \nabla \omega\|_{L^2}$, we have
\begin{equation*} 
\begin{split}
\|u \cdot \nabla \omega\|_{L^2}
&\lesssim \|u^1 \partial_1 \omega\|_{L^2} + \|u^2 \partial_2 \omega\|_{L^2}  \lesssim \|u^1\|_{L^{\infty}} \|\partial_1 \omega_{\neq}\|_{L^2}
        + \|u_{\neq}^2\|_{L^{\infty}} \|\partial_2 \omega\|_{L^2}.
\end{split}
\end{equation*}
Using \eqref{est-I-L2-3}, this implies
\begin{equation}\label{est-partial2-omega-8}
\begin{split}
\|u \cdot \nabla \omega\|_{L^2}
&\lesssim \|(u, \omega)\|_{L^2} \|\partial_1 \omega_{\neq}\|_{L^2}
        + \|\partial_1 u_{\neq}^1\|_{L^2}^{\frac12} \|\partial_1\omega_{\neq}\|_{L^2}^{\frac32} \\
&\qquad + \|\partial_1 u_{\neq}^2\|_{L^2}^{\frac12} \|\partial_1^2 u_{\neq}^1\|_{L^2}^{\frac12} \|\partial_2 \omega\|_{L^2}.
\end{split}
\end{equation}
The quantity $\|(u \cdot \nabla) \partial_2 u\|_{L^2}$ admits the same estimate as \eqref{est-partial2-omega-8}.

For $\|(\partial_2 u \cdot \nabla) u\|_{L^2}$, we have
\begin{equation*}\label{est-partial2-omega-9}
\begin{split}
\|(\partial_2 u \cdot \nabla) u\|_{L^2}
&\lesssim \|\partial_2 u^1 \partial_1 u\|_{L^2} + \|\partial_2 u^2 \partial_2 u\|_{L^2} \\
&\lesssim \|\partial_2 u\|_{L^2_x(L^{\infty}_y)} \|\partial_1 u_{\neq}\|_{L^2_y(L^{\infty}_x)}  \lesssim \|\omega\|_{L^2}^{\frac12} \|\partial_2\omega\|_{L^2}^{\frac12} \|\partial_1^2 u_{\neq}\|_{L^2}.
\end{split}
\end{equation*}
Thus,
\begin{equation}\label{est-partial2-omega-10}
\begin{split}
\|u \cdot \nabla \partial_2 u\|_{L^2} &+ \|(\partial_2 u \cdot \nabla) u\|_{L^2} + \|u \cdot \nabla \omega\|_{L^2} \\
&\lesssim \|(u, \omega)\|_{L^2} \|\partial_1 \omega_{\neq}\|_{L^2} + \|\partial_1 u_{\neq}^1\|_{L^2}^{\frac12} \|\partial_1\omega_{\neq}\|_{L^2}^{\frac32} \\
&\quad + \|\partial_1 u_{\neq}^2\|_{L^2}^{\frac12} \|\partial_1^2 u_{\neq}^1\|_{L^2}^{\frac12} \|\partial_2 \omega\|_{L^2}  + \|\omega\|_{L^2}^{\frac12} \|\partial_2\omega\|_{L^2}^{\frac12} \|\partial_1^2 u_{\neq}\|_{L^2}.
\end{split}
\end{equation}

For $\|\partial_t \omega\|_{L^2}$, using \eqref{eqns-HINS-aver-omega-5} gives
\begin{equation}\label{est-partial2-omega-13}
\begin{split}
\|\partial_t \omega\|_{L^2}
&\lesssim \|\partial_t \omega_{\neq}\|_{L^2} + \|\partial_t \overline{\omega}\|_{L^2} \\
&\lesssim \|\partial_t \omega_{\neq}\|_{L^2}
        + (\|\omega\|_{L^2}^{\frac12} \|\partial_2\omega\|_{L^2}^{\frac12} + \|\partial_2\omega\|_{L^2})
          \|(\partial_1 u_{\neq}, \partial_1^2 u_{\neq})\|_{L^2} \\
&\quad + (1 + \|(u, \omega)\|_{L^2})
        \|(\partial_1 u_{\neq}, \partial_1^2 u_{\neq}, \partial_t u_{\neq}^1)\|_{L^2} \|\nabla\rho\|_{L^{\infty}} \\
&\quad + \|(u, \omega)\|_{L^2}^{\frac12} \|\partial_1\omega_{\neq}\|_{L^2}^{\frac12}
        \|\partial_1 u_{\neq}\|_{L^2} \|\nabla\rho\|_{L^{\infty}}.
\end{split}
\end{equation}
Therefore, from \eqref{H1-hins-u-4}, \eqref{est-partial2-omega-10} and \eqref{est-partial2-omega-13}, we obtain
\begin{equation}\label{est-partial2-omega-14}
\begin{split}
\mathfrak{m}(t)
&\lesssim \|\partial_t \omega_{\neq}\|_{L^2}
        + \|(u, \omega)\|_{L^2} \mathcal{E}_{2, \neq}^{\frac12}
        + \mathcal{E}_{1, \neq}^{\frac14} \mathcal{E}_{2, \neq}^{\frac34}  + \|(u, \omega)\|_{L^2}^{\frac12} \|\nabla\rho\|_{L^{\infty}}
        \mathcal{E}_{1, \neq}^{\frac12} \mathcal{E}_{2, \neq}^{\frac14} \\
&\quad + \|\omega\|_{L^2}^{\frac12} \|\partial_2\omega\|_{L^2}^{\frac12} \mathcal{D}_{1, \neq}^{\frac12}
        + \|\partial_2\omega\|_{L^2} \mathcal{D}_{1, \neq}^{\frac12}  + (1 + \|(u, \omega)\|_{L^2}) \|\nabla\rho\|_{L^{\infty}} \mathcal{D}_{1, \neq}^{\frac12}.
\end{split}
\end{equation}
For $\|\partial_2 u \cdot \nabla \omega\|_{L^1_x(L^{2}_y)}$, we have
\begin{equation}\label{est-partial2-omega-11}
\begin{split}
\|\partial_2 u \cdot \nabla \omega\|_{L^1_x(L^{2}_y)}
&\lesssim \|\partial_2 u^1\|_{L^2_x(L^{\infty}_y)} \|\partial_1\omega\|_{L^2}
        + \|\partial_2 u^2\|_{L^2_x(L^{\infty}_y)} \|\partial_2\omega\|_{L^2} \\
&\lesssim \|\partial_2 u^1\|_{L^2}^{\frac12} \|\partial_2^2 u^1\|_{L^2}^{\frac12} \|\partial_1\omega_{\neq}\|_{L^2}  + \|\partial_1 u_{\neq}^1\|_{L^2}^{\frac12} \|\partial_1\partial_2 u_{\neq}^1\|_{L^2}^{\frac12} \|\partial_2\omega\|_{L^2} \\
&\lesssim \|\omega\|_{L^2}^{\frac12} \|\partial_2\omega\|_{L^2}^{\frac12} \mathcal{E}_{2, \neq}^{\frac12}.
\end{split}
\end{equation}

Substituting \eqref{est-partial2-omega-14} and \eqref{est-partial2-omega-11} into \eqref{est-partial2-omega-5aaa}, and using \eqref{H1-hins-u-7}, we obtain
\begin{equation}\label{est-partial2-omega-15}
\frac{d}{dt} \mathcal{E}_2(t) + \mathfrak{e}_8 \mathcal{D}_2 \leq C_{7} \sum_{j=1}^5 \mathcal{F}_j
\end{equation}
with
\begin{equation}\label{est-partial2-omega-16}
\begin{split}
\mathcal{E}_2(t) &\eqdefa \|(\rho^{\frac12} u, \rho^{\frac12} \omega, \partial_1 u,
                         \partial_1\omega, \rho^{\frac12} \partial_2\omega)(t)\|_{L^2}^2, \andf\\
\mathcal{D}_2(t) &\eqdefa \|(\partial_1 u, \partial_1\omega, \partial_1^2 u,
                         \partial_t u, \partial_t\omega, \partial_1^2\omega,
                         \partial_1\partial_2\omega)(t)\|_{L^2}^2,
\end{split}
\end{equation}
and
\begin{align*}
\mathcal{F}_1 &\eqdefa \|\nabla\rho\|_{L^{\infty}\cap\dot{H}^1} \Bigl(
        \mathcal{E}_2^{\frac12} \mathcal{D}_{2, \neq}^{\frac12}
        + \mathcal{E}_2^{\frac12} (1 + \mathcal{E}_2^{\frac12}) \mathcal{D}_{1, \neq}^{\frac12}  + \mathcal{E}_2^{\frac12} \bigl(\|(u, \omega)\|_{L^2} \mathcal{E}_{2, \neq}^{\frac12}
              + \mathcal{E}_{1, \neq}^{\frac14} \mathcal{E}_{2, \neq}^{\frac34}\bigr) \Bigr),\\
\mathcal{F}_2 &\eqdefa \|\nabla\rho\|_{L^{\infty}\cap\dot{H}^1}^2 \Bigl(
        \mathcal{D}_{2, \neq} + \mathcal{E}_2^{\frac12} \mathcal{D}_{1, \neq}^{\frac12} + \|(u, \omega)\|_{L^2}^{\frac12} \mathcal{E}_2^{\frac12}
                \mathcal{E}_{1, \neq}^{\frac12} \mathcal{E}_{2, \neq}^{\frac14} + \|(u, \omega)\|_{L^2}^2 \mathcal{E}_{2, \neq}  \\
      &\qquad\qquad\qquad\qquad                + \mathcal{E}_{1, \neq}^{\frac12} \mathcal{E}_{2, \neq}^{\frac32}
                + \mathcal{E}_2 \mathcal{D}_{1, \neq} \Bigr),\\
\mathcal{F}_3 &\eqdefa \|\nabla\rho\|_{L^{\infty}\cap\dot{H}^1}^4
        \bigl(\mathcal{D}_{1, \neq} + \|(u, \omega)\|_{L^2} \mathcal{E}_{1, \neq} \mathcal{E}_{2, \neq}^{\frac12}\bigr), \\
         \mathcal{F}_4  &\eqdefa \bigl(\|\nabla\rho\|_{L^{\infty}}^2 \mathcal{E}_{1, \neq}\bigr)^{\frac23}
        \|(u, \omega)\|_{L^2}^2, \\
\mathcal{F}_5 &\eqdefa \|\omega\|_{L^2}^{\frac12} \mathcal{E}_2^{\frac34} \mathcal{E}_{2, \neq}^{\frac12}
        + \|\omega\|_{L^2} \mathcal{E}_2^{\frac12} \mathcal{E}_{2, \neq}
        + \|(u, \omega)\|_{L^2}^2 \mathcal{D}_{1, \neq}.
\end{align*}

Define
\begin{equation}\label{est-partial2-omega-18}
\mathfrak{E}_2(T) \eqdefa \sup_{t \in [0, T]} \mathcal{E}_2(t) +\mathfrak{e}_8 \int_0^T \mathcal{D}_2(t) \,dt.
\end{equation}
Using  \eqref{est-L2-oscillation-t-omega-15} and \eqref{small-init-u0-1}, we have
\begin{equation}\label{est-partial2-omega-T-18aa}
\begin{split}
\widetilde{\mathcal{E}}_{2,\neq, T}
&\lesssim \mathcal{E}_{2,\neq, 0}
    + \epsilon_0^2 + \epsilon_0^2 \mathfrak{E}_2(T)
                + \|u_{0}\|_{H^1} \mathfrak{E}_2^{\frac12}(T).
\end{split}
\end{equation}
Together with \eqref{est-L2-oscillation-t-omega-14} and \eqref{est-L2-oscillation-t-omega-16}, this implies
\begin{equation}\label{est-L2-oscillation-t-omega-19}
\begin{split}
\sup_{t \in [0,T]} &\bigl(e^{3c_{2} t} \mathcal{E}_{2, \neq}(t)\bigr)
   + \int_0^{T} e^{2c_{2} t} \mathcal{D}_{2, \neq}(t) \,dt \\
   &\leq C_{8} \bigl(\mathcal{E}_{2,\neq, 0}+\epsilon_0^2 + \epsilon_0^2 \mathfrak{E}_2(T)
              + \|u_{0}\|_{H^1} \mathfrak{E}_2^{\frac12}(T)\bigr).
\end{split}
\end{equation}

Let us deine
\begin{equation}\label{conditions-sol-5}
\begin{split}
T^\star_2\eqdefa \sup\bigl\{t<T^\star_1, \  \, \|\nabla\rho\|_{L^\infty_t(L^{\infty}\cap\dot{H}^1)} \leq 2\|\nabla\rho_0\|_{L^{\infty}\cap\dot{H}^1} \, e^{\frac{c_2}{4}t}\ \ \bigr\},
\end{split}
\end{equation}
then from \eqref{est-partial2-omega-15}, we obtain
\begin{equation}\label{est-partial2-omega-16}
\begin{split}
&\mathfrak{E}_2(T)\leq C_{9} \bigg(\mathcal{E}_2(0)+\|\nabla\rho_0\|_{L^{\infty}\cap\dot{H}^1} \, \mathfrak{F}_1(T)+\|\nabla\rho_0\|_{L^{\infty}\cap\dot{H}^1}^2 \, \mathfrak{F}_2(T)\\
&+\|\nabla\rho_0\|_{L^{\infty}\cap\dot{H}^1}^4 \mathfrak{F}_3(T)+\|\nabla\rho_0\|_{L^{\infty}\cap\dot{H}^1}^{\frac{4}{3}} \,\epsilon_0^2 \,\mathfrak{F}_4(T) +\epsilon_0^{\frac{1}{2}}\,\mathfrak{F}_5(T)+ \epsilon_0\,\mathfrak{F}_6(T)+\epsilon_0^2\,\mathfrak{F}_7(T)\bigg),
\end{split}
\end{equation}
where
\begin{align*}
\mathfrak{F}_1(T) \eqdefa& \mathfrak{E}_2^{\frac12} (T)\int_{0}^T e^{\frac{c_{2}}4 t} \mathcal{D}_{2, \neq}^{\frac12}(t) \,dt + \mathfrak{E}_2^{\frac12}(T)\bigl (1 + \mathfrak{E}_2^{\frac12}(T)\bigr) \int_{0}^T e^{\frac{c_{2}}4 t} \mathcal{D}_{1, \neq}^{\frac12}(t) \,dt \\
& + \mathfrak{E}_2^{\frac12}(T) \Bigl( \epsilon_0 \int_{0}^T e^{\frac{c_{2}}4 t} \mathcal{E}_{2, \neq}^{\frac12}(t) \,dt
          + \int_{0}^T e^{\frac{c_{2}}4 t} \mathcal{E}_{1, \neq}^{\frac14}(t) \mathcal{E}_{2, \neq}^{\frac34} (t)\,dt \Bigr),\\
\mathfrak{F}_2(T) \eqdefa & \mathfrak{E}_2^{\frac12}(T) \int_{0}^T e^{\frac{c_{2}}2 t} \mathcal{D}_{1, \neq}^{\frac12}(t) \,dt  + \epsilon_0^{\frac12} \mathfrak{E}_2^{\frac12}(T) \int_{0}^T e^{\frac{c_{2}}4 t}
        \mathcal{E}_{1, \neq}^{\frac12}(t) \mathcal{E}_{2, \neq}^{\frac14} (t)\,dt \\
        &+ \epsilon_0^2 \int_{0}^T e^{\frac{c_{2}}2 t} \mathcal{E}_{2, \neq}(t) \,dt + \int_{0}^T e^{\frac{c_{2}}2 t} \mathcal{E}_{1, \neq}^{\frac12}(t) \mathcal{E}_{2, \neq}^{\frac32}(t) \,dt \\
        &+ \mathfrak{E}_2(T) \int_{0}^T e^{\frac{c_{2}}2 t} \mathcal{D}_{1, \neq}(t) \,dt + \int_{0}^T e^{\frac{c_{2}}2 t} \mathcal{D}_{2, \neq}(t) \,dt,\\
\mathfrak{F}_3(T) \eqdefa & \int_{0}^T e^{c_{2} t} \mathcal{D}_{1, \neq} (t)\,dt
                  + \epsilon_0 \int_{0}^T e^{c_{2} t} \mathcal{E}_{1, \neq} (t)\mathcal{E}_{2, \neq}^{\frac12} (t)\,dt, \\
\mathfrak{F}_4(T) \eqdefa&  \int_{0}^T e^{\frac{c_{2}}3 t} \mathcal{E}_{1, \neq}^{\frac23} (t)\,dt,\qquad
\mathfrak{F}_5(T) \eqdefa  \mathfrak{E}_2^{\frac34}(T) \int_{0}^T \mathcal{E}_{2, \neq}^{\frac12}(t) \,dt, \\
\mathfrak{F}_6(T) \eqdefa&  \mathfrak{E}_2^{\frac12}(T) \int_{0}^T \mathcal{E}_{2, \neq}(t) \,dt, \qquad
\mathfrak{F}_7(T) \eqdefa \int_{0}^T \mathcal{D}_{1, \neq}(t) \,dt.
\end{align*}
Hence,
\begin{equation}\label{est-partial2-omega-22}
\begin{split}
\mathfrak{E}_2(T) &\leq  C_{10} \Bigl(\mathcal{E}_2(0)  +\|u_0\|_{H^1} \|\nabla\rho_0\|_{L^{\infty}\cap\dot{H}^1}
        \bigl(1 + \|u_0\|_{H^1} \|\nabla\rho_0\|_{L^{\infty}\cap\dot{H}^1}\bigr) \mathfrak{E}_2(T) \\
&\quad + \bigl(\mathcal{G}_1 + \mathcal{G}_2 + \mathcal{G}_3\bigr) \mathfrak{E}_2^{\frac12}(T) + \mathcal{G}_4 \mathfrak{E}_2^{\frac34}(T) + \mathcal{G}_5 + \mathcal{G}_6 + \mathcal{G}_7\Bigr),
\end{split}
\end{equation}
where
\begin{align*}
\mathcal{G}_1 &\eqdefa \|\nabla\rho_0\|_{L^{\infty}\cap\dot{H}^1} \bigl(
        \widetilde{\mathcal{E}}_{2,\neq, T}^{\frac12} + \|u_0\|_{H^1}  + \epsilon_0 \widetilde{\mathcal{E}}_{2,\neq, T}^{\frac12}
          + \|u_0\|_{H^1}^{\frac12} \widetilde{\mathcal{E}}_{2,\neq, T}^{\frac34} \bigr), \\
\mathcal{G}_2 &\eqdefa \|\nabla\rho_0\|_{L^{\infty}\cap\dot{H}^1}^2 \bigl(
        \|u_0\|_{H^1} + \epsilon_0^{\frac12} \|u_0\|_{H^1} \widetilde{\mathcal{E}}_{2,\neq, T}^{\frac14} \bigr),\\
\mathcal{G}_3 &\eqdefa \epsilon_0 \widetilde{\mathcal{E}}_{2,\neq, T}, \quad
\mathcal{G}_4 \eqdefa\epsilon_0^{\frac12} \widetilde{\mathcal{E}}_{2,\neq, T}^{\frac12}, \quad \mathcal{G}_5  \eqdefa \|\nabla\rho_0\|_{L^{\infty}\cap\dot{H}^1}^2
        \bigl(1 + \epsilon_0^2 + \|u_0\|_{H^1} \widetilde{\mathcal{E}}_{2,\neq, T}^{\frac12}\bigr)
        \widetilde{\mathcal{E}}_{2,\neq, T}, \\
\mathcal{G}_6 &\eqdefa  \|\nabla\rho_0\|_{L^{\infty}\cap\dot{H}^1}^4
        \|u_0\|_{H^1}^2 \bigl(1 + \epsilon_0 \widetilde{\mathcal{E}}_{2,\neq, T}^{\frac12}\bigr), \quad \mathcal{G}_7  \eqdefa \|\nabla\rho_0\|_{L^{\infty}\cap\dot{H}^1}^{\frac43} \epsilon_0^2 \|u_0\|_{H^1}^{\frac43}
        + \epsilon_0^2 \|u_0\|_{H^1}^2.
\end{align*}

If
\begin{equation}\label{est-partial2-omega-24-condition}
\begin{split}
\|u_0\|_{H^1} \|\nabla\rho_0\|_{L^{\infty}\cap\dot{H}^1}
\bigl(1 + \|u_0\|_{H^1} \|\nabla\rho_0\|_{L^{\infty}\cap\dot{H}^1}\bigr) \leq \frac{1}{3C_{10}},
\end{split}
\end{equation}
then by Young's inequality, we obtain
\begin{equation*}\label{est-partial2-omega-26}
\begin{split}
\mathfrak{E}_2(T) &\leq C_{11} \bigl(\mathcal{E}_2(0) + \mathcal{G}_1^2 + \mathcal{G}_2^2
                  + \mathcal{G}_3^2 + \mathcal{G}_4^4 + \mathcal{G}_5 + \mathcal{G}_6 + \mathcal{G}_7\bigr).
\end{split}
\end{equation*}
Using \eqref{est-L2-oscillation-t-omega-19}, we find
\begin{equation*}\label{est-partial2-omega-27}
\begin{split}
\mathfrak{E}_2(T) &\leq C_{12} \mathfrak{I}_0+ C_{12} \Bigl((\|\nabla\rho_0\|_{L^{\infty}\cap\dot{H}^1}^2 + 1) \epsilon_0^2 \mathfrak{E}_2(T) + \epsilon_0^4 (1 + \|\nabla\rho_0\|_{L^{\infty}\cap\dot{H}^1}^2) \mathfrak{E}^2_2(T)\Bigr),
\end{split}
\end{equation*}
where
\begin{equation}\label{est-partial2-omega-28}
\begin{split}
\mathfrak{I}_0 &\eqdefa \|u_0\|_{H^2}^2 + \|\nabla\rho_0\|_{L^{\infty}\cap\dot{H}^1}^2
        \bigl(\mathcal{E}_{2,\neq, 0} + \mathcal{E}_{2,\neq, 0}^2 + \epsilon_0^2\bigr)  + \|u_0\|_{H^1} \mathcal{E}_{2,\neq, 0}^{\frac32}+ \epsilon_0^2 \mathcal{E}_{2,\neq, 0}^2 \\
    &\quad         + \epsilon_0^3 + \|\nabla\rho_0\|_{L^{\infty}\cap\dot{H}^1}^4
        \bigl(\|u_0\|_{H^1}^2 + \epsilon_0^2 \|u_0\|_{H^1} \mathcal{E}_{2,\neq, 0} + \epsilon_0^2\bigr) + \|\nabla\rho_0\|_{L^{\infty}\cap\dot{H}^1}^8
        \|u_0\|_{H^1}^4 \epsilon_0^4.
\end{split}
\end{equation}
Assume that
\begin{equation}\label{cond-smallness-112}
(1 + \|\nabla\rho_0\|_{L^{\infty}\cap\dot{H}^1}^2) \epsilon_0^2 \leq \frac{1}{2C_{12}}.
\end{equation}
Then
\begin{equation*}\label{est-partial2-omega-29}
\mathfrak{E}_2(T) \leq  C_{13} \frak{K}_0 +  \epsilon_0^2 \mathfrak{E}_2(T)^2,
\end{equation*}
and
\begin{equation}\label{est-partial2-omega-30}
\begin{split}
\frak{K}_0 \eqdefa &\|u_0\|_{H^2}^2
        + \|\nabla\rho_0\|_{L^{\infty}\cap\dot{H}^1}^2
          \bigl(\mathcal{E}_{2,\neq, 0} + \mathcal{E}_{2,\neq, 0}^2 + \epsilon_0^2\bigr) \\
        &+ \|u_0\|_{H^1} \mathcal{E}_{2,\neq, 0}^{\frac32}
          + \epsilon_0^2 \mathcal{E}_{2,\neq, 0}^2  + \epsilon_0^3 + \|\nabla\rho_0\|_{L^{\infty}\cap\dot{H}^1}^4 \epsilon_0^2\,(1+\|u_0\|_{H^1} \mathcal{E}_{2,\neq, 0}).
\end{split}
\end{equation}
Assuming
\begin{equation}\label{cond-smallness-113}
\epsilon_0^2 \mathfrak{K}_0 \leq \frac{1}{4C_{13}},
\end{equation}
we obtain
\begin{equation}\label{est-partial2-omega-32}
\mathfrak{E}_2(T) \leq 2C_{13} \mathfrak{K}_0,\quad \forall\ T\leq T_2^\star.
\end{equation}

\noindent\textbf{$\bullet$ Smallness of $\|\rho_{\neq}\|_{L^{\infty}_t(L^{\infty})}$.}\smallskip

Applying $\mathcal{P}_{\neq}$ to the transport equation for $\rho$ in \eqref{eqns-HINS-1} yields
\begin{equation*}\label{eqns-hins-rho-oscillation-1}
\begin{split}
\partial_t \rho_{\neq} + \bar{u}^1 \partial_1\rho_{\neq}
+ (u_{\neq}^1 \partial_1\rho_{\neq})_{\neq}
+ u_{\neq}^2 \partial_2\bar{\rho}
+ (u_{\neq}^2 \partial_2\rho_{\neq})_{\neq} = 0.
\end{split}
\end{equation*}

Define $\bar{X}(t, x)$ by
\begin{equation*}\label{est-rho-neq-1}
\begin{cases}
\frac{d}{dt}\bar{X}(t, x) = \bar{u}(t, \bar{X}(t, x)), \\[1mm]
\bar{X}(t, x)|_{t=0} = x,
\end{cases}
\end{equation*}
then since $\bar{u}^2 = 0$, we have
\begin{equation*}\label{est-rho-neq-2}
\begin{split}
\frac{d}{dt} \bigl(\rho_{\neq}(t, \bar{X}(t, x))\bigr)
&= - (u_{\neq}^1 \partial_1\rho_{\neq})_{\neq} \circ \bar{X}(t, x)  - (u_{\neq}^2 \partial_2\bar{\rho}) \circ \bar{X}(t, x)  - (u_{\neq}^2 \partial_2\rho_{\neq})_{\neq} \circ \bar{X}(t, x).
\end{split}
\end{equation*}
Hence,
\begin{equation*}\label{est-rho-neq-3}
\begin{split}
\rho_{\neq}(t, \bar{X}(t, x))
&= \rho_{0\neq}
   - \int_0^t (u_{\neq}^1 \partial_1\rho_{\neq})_{\neq} \circ \bar{X}(s, x) \,ds  \\
&\quad - \int_0^t (u_{\neq}^2 \partial_2\bar{\rho}) \circ \bar{X}(s, x) \,ds  - \int_0^t (u_{\neq}^2 \partial_2\rho_{\neq})_{\neq} \circ \bar{X}(s, x) \,ds.
\end{split}
\end{equation*}
This implies
\begin{equation}\label{est-rho-neq-5}
\begin{split}
\|\rho_{\neq}(t)\|_{L^{\infty}}
&\leq \|\rho_{0,\neq}\|_{L^{\infty}}
   + C \int_0^t \|(u_{\neq}^1, u_{\neq}^2)(s)\|_{L^{\infty}} \|\na \rho(s)\|_{L^{\infty}} \,ds.
\end{split}
\end{equation}
Using \eqref{est-I-L2-3}, we obtain
\begin{equation*}\label{est-rho-neq-6aaa}
\begin{split}
\|\rho_{\neq}(t)\|_{L^{\infty}}
&\leq \|\rho_{0,\neq}\|_{L^{\infty}}
   + C \int_0^t \|\partial_1 u_{\neq}(s)\|_{L^2}^{\frac12}
             \|\partial_1\omega_{\neq}(s)\|_{L^2}^{\frac12}
             \|\na \rho(s)\|_{L^{\infty}} \,ds,
\end{split}
\end{equation*}
so that we deduce from \eqref{est-Hins-H1-u-neq-18}, \eqref{conditions-sol-3} and \eqref{est-L2-oscillation-t-omega-14} that
\begin{equation}\label{est-rho-neq-6}
\begin{split}
\|\rho_{\neq}(t)\|_{L^{\infty}}
&\leq \|\rho_{0\neq}\|_{L^{\infty}}
   + C  \int_0^t \mathcal{E}_{1, \neq}^{\frac14}
                \mathcal{E}_{2, \neq}^{\frac14}
                \|\nabla\rho\|_{L^{\infty}} \,ds \\
&\leq \|\rho_{0,\neq}\|_{L^{\infty}}
   + C  \|\nabla\rho_0\|_{L^{\infty}\cap \dot{H}^1}
           \|u_0\|_{H^1}^{\frac12}
           \widetilde{\mathcal{E}}_{2,\neq, t}^{\frac14},
\end{split}
\end{equation}
where $\mathcal{E}_{1, \neq}$ and $
                \mathcal{E}_{2, \neq}$ are defined respectively by \eqref{def-H1-u-neq-1} and \eqref {est-L2-oscillation-t-omega-8}.

                \smallskip
\noindent\textbf{$\bullet$ The estimate of $\|\nabla\rho\|_{L^{\infty}_t(L^{\infty})}$.}\smallskip

First, applying $\nabla$ and $\nabla\partial$ to the $\rho$ equation in \eqref{eqns-HINS-1} yields
\begin{equation*}\label{L2-hins-u-6aaa}
\begin{split}
\partial_t(\nabla\rho) + u \cdot \nabla(\nabla\rho) &= -\nabla u \cdot \nabla\rho,
\end{split}
\end{equation*}
and
\begin{equation*}\label{L2-hins-u-6bbb}
\begin{split}
\partial_t(\nabla\partial\rho) + u \cdot \nabla(\nabla\partial\rho)
&= -\partial u \cdot \nabla(\nabla\rho)  - \partial\nabla u \cdot \nabla\rho  - \nabla u \cdot \partial\nabla\rho.
\end{split}
\end{equation*}
The classical theory of transport equations then gives
\begin{equation}\label{L2-hins-u-7}
\begin{split}
\|\nabla\rho\|_{L^{\infty}_t(L^{\infty}\cap \dot{H}^1)}
\leq \|\nabla\rho_0\|_{L^{\infty}\cap \dot{H}^1}
   \, e^{C \int_0^t \|\nabla u(s)\|_{L^{\infty}\cap\dot{H}^1} \,ds}.
\end{split}
\end{equation}

\noindent\textbf{$\bullet$ The estimate of $\|\nabla u\|_{L^{\infty}\cap \dot{H}^1}$.}\smallskip

We now estimate $\|\nabla u\|_{L^{\infty}}$. Note that
\begin{equation}\label{est-Lip-u-1}
\begin{split}
&\|\nabla u\|_{L^{\infty}}
 \lesssim \|\partial_1 u\|_{L^{\infty}} + \|\partial_2 u^1\|_{L^{\infty}} \\
&\lesssim \|\partial_1 u\|_{L^2_y(L^{\infty}_x)}^{\frac12}
          \|\partial_2\partial_1 u\|_{L^2_y(L^{\infty}_x)}^{\frac12}
        + \|\partial_2 u^1\|_{L^2_y(L^{\infty}_x)}^{\frac12}
          \|\partial_2^2 u^1\|_{L^2_y(L^{\infty}_x)}^{\frac12} \\
&\lesssim \|\partial_1^2 u_{\neq}\|_{L^2}^{\frac12}
          \|\partial_1^2 \omega_{\neq}\|_{L^2}^{\frac12} + \bigl(\|\partial_2 u^1\|_{L^2} + \|\partial_1\partial_2 u_{\neq}^1\|_{L^2}\bigr)^{\frac12}
          \bigl(\|\partial_2\omega\|_{L^2} + \|\partial_1\partial_2\omega\|_{L^2}\bigr)^{\frac12},
\end{split}
\end{equation}
which implies
\begin{equation*}\label{est-Lip-u-3}
\begin{split}
\|\nabla u\|_{L^{\infty}}^2
&\lesssim \|\partial_1^2 u_{\neq}\|_{L^2} \|\partial_1^2 \omega_{\neq}\|_{L^2}  + (\|\omega\|_{L^2}
        + \|\partial_1 u_{\neq}^1\|_{L^2}^{\frac12}
          \|\partial_1\partial_2\omega\|_{L^2}^{\frac12})  \bigl(\|\partial_2\omega\|_{L^2} + \|\partial_1\partial_2\omega\|_{L^2}\bigr).
\end{split}
\end{equation*}
Similarly,
\begin{equation*}\label{est-Lip-H2-u-4}
\begin{split}
\|\nabla u\|_{L^{\infty}\cap\dot{H}^1}^2
&\lesssim \|\partial_1^2 u_{\neq}\|_{L^2} \|\partial_1^2 \omega_{\neq}\|_{L^2}  + \|\omega\|_{L^2} \|\partial_2\omega\|_{L^2}
        + \|\omega\|_{L^2} \|\partial_1\partial_2\omega\|_{L^2} \\
&\quad + \|\partial_1 u_{\neq}^1\|_{L^2}^{\frac12}
          \|\partial_1\partial_2\omega\|_{L^2}^{\frac32}+ \|\partial_1 u_{\neq}^1\|_{L^2}^{\frac12}
          \|\partial_1\partial_2\omega\|_{L^2}^{\frac12}
          \|\partial_2\omega\|_{L^2}  .
\end{split}
\end{equation*}
Hence, 
\begin{equation*}\label{est-Lip-u-5}
\begin{split}
\int_0^t \|\nabla u\|_{L^{\infty}\cap\dot{H}^1}^2 \,d\tau
&\lesssim \|\partial_1^2 u_{\neq}\|_{L^{2}_t(L^2)}
          \|\partial_1^2 \omega_{\neq}\|_{L^{2}_t(L^2)}  + \|\omega\|_{L^{\infty}_t(L^2)}
          \|\partial_2\omega\|_{L^{\infty}_t( L^2)} \,t \\
&\quad + \|\omega\|_{L^{\infty}_t(L^2)}
          \|\partial_1\partial_2\omega\|_{L^{2}_t(L^2)} \, t^{\frac12} + \|\partial_1 u_{\neq}^1\|_{L^{2}_t(L^2)}^{\frac12}
          \|\partial_1\partial_2\omega\|_{L^{2}_t(L^2)}^{\frac32}\\
&\quad + \|e^{c_1 \tau} \partial_1 u_{\neq}^1\|_{L^{2}_t(L^2)}^{\frac12}
          \|\partial_1\partial_2\omega\|_{L^{2}_t(L^2)}^{\frac12}
          \|\partial_2\omega\|_{L^{\infty}_t(L^2)}
          \|e^{-c_1 \tau}\|_{L^2([0, t])}.
\end{split}
\end{equation*}
Thus, by virtue of \eqref{est-Hins-expo-1}, \eqref{est-Hins-expo-2} and \eqref{est-partial2-omega-32}, we find
\begin{equation*}\label{est-Lip-u-6}
\begin{split}
\int_0^t &\|\nabla u\|_{L^{\infty}\cap\dot{H}^1}^2 \,d\tau \lesssim \|u_0\|_{H^1}^{\frac12} (1 + \|u_0\|_{H^1}^{\frac12})
          \mathfrak{K}_0^{\frac12} (1 + t)
        + \|u_0\|_{H^1}^{\frac12} \mathfrak{K}_0^{\frac34},
\end{split}
\end{equation*}
and then
\begin{equation}\label{est-Lip-u-7}
\begin{split}
\|\nabla u\|_{L^1([0, t]; L^{\infty}\cap\dot{H}^1)}
&\lesssim \|u_0\|_{H^1}^{\frac14} (1 + \|u_0\|_{H^1}^{\frac14})
          \mathfrak{I}_0^{\frac14} (1 + t)^{\frac12} t^{\frac12}  + \|u_0\|_{H^1}^{\frac14} \mathfrak{K}_0^{\frac38} t^{\frac12} \\
&\lesssim \|u_0\|_{H^1}^{\frac14} (1 + \mathfrak{K}_0^{\frac18})
          \mathfrak{K}_0^{\frac14} (1 + t).
\end{split}
\end{equation}
Using \eqref{L2-hins-u-7}, we obtain
\begin{equation}\label{est-H2-rho-hins-1}
\begin{split}
\|\nabla\rho\|_{L^{\infty}_t(L^{\infty}\cap \dot{H}^1)}
&\leq \|\nabla\rho_0\|_{L^{\infty}\cap \dot{H}^1}
   \, e^{C_{14} \|u_0\|_{H^1}^{\frac14} (1 + \mathfrak{K}_0^{\frac18})
          \mathfrak{K}_0^{\frac14} (1 + t)}.
\end{split}
\end{equation}
From \eqref{est-partial2-omega-T-18aa}, \eqref{est-partial2-omega-32} and \eqref{est-rho-neq-6}, we have
\begin{equation}\label{est-partial2-omega-33}
\begin{split}
\widetilde{\mathcal{E}}_{2,\neq, T}
&\lesssim \mathcal{E}_{2,\neq, 0} + \epsilon_0^2
        + \epsilon_0^2 \mathfrak{K}_0
        + \|u_{0}\|_{H^1} \mathfrak{K}_0^{\frac12},
\end{split}
\end{equation}
and
\begin{equation}\label{rho-occis-bdd-11}
\begin{split}
\|\rho_{\neq}(t)\|_{L^{\infty}}
&\leq \|\rho_{0,\neq}\|_{L^{\infty}}  +  C_{15}\|\nabla\rho_0\|_{L^{\infty}\cap \dot{H}^1} 
            \|u_0\|_{H^1}^{\frac12}
            (\mathcal{E}_{2,\neq, 0} + \epsilon_0^2
             + \epsilon_0^2 \mathfrak{K}_0
             + \|u_{0}\|_{H^1} \mathfrak{K}_0^{\frac12})^{\frac14}.
\end{split}
\end{equation}

\noindent\textbf{$\bullet$ Smallness assumptions on the initial data.}
\smallskip

Under the smallness condition \eqref{assumption-initial-103} for $ \mathfrak{c}$ being  sufficiently small, we find
\begin{equation}\label{initial-smallness-101}
\begin{split}
&\|u_0\|_{H^1}^{\frac14} (1 + \mathfrak{K}_0^{\frac18}) \mathfrak{K}_0^{\frac14}
   \lesssim \mathfrak{c}^{\frac14}, \quad \epsilon_0^2 \mathfrak{K}_0 \lesssim \mathfrak{c}^2, \\
   & \|\rho_{0, \neq}\|_{L^{\infty}}+\epsilon_0 \lesssim \mathfrak{c}, \quad \bigl(1+ \|\nabla\rho_0\|_{L^{\infty}\cap\dot{H}^1}^2\bigr) \epsilon_0^2
   \lesssim \mathfrak{c}^2, \\
   &  \|u_0\|_{H^1} \|\nabla\rho_0\|_{L^{\infty}\cap\dot{H}^1}
   (1 + \|u_0\|_{H^1} \|\nabla\rho_0\|_{L^{\infty}\cap\dot{H}^1})
   \lesssim \mathfrak{c}, \\
&\|\nabla\rho_0\|_{L^{\infty}\cap \dot{H}^1} 
            \|u_0\|_{H^1}^{\frac12}
            (\mathcal{E}_{2,\neq, 0} + \epsilon_0^2
             + \epsilon_0^2 \mathfrak{K}_0
             + \|u_{0}\|_{H^1} \mathfrak{K}_0^{\frac12})^{\frac14}
   \lesssim \mathfrak{c}^{\frac{1}{4}}.
\end{split}
\end{equation}
Then we deduce from \eqref{assumption-initial-103}, \eqref{H1-hins-u-9} and \eqref{rho-occis-bdd-11}-\eqref{initial-smallness-101} that
\begin{equation}\label{u-H1-smallness-122}
\sup_{t\in [0,T_2^\star]}(\|\rho_{\neq}(t)\|_{L^{\infty}}+\|(u, \omega)(t)\|_{L^2})
\leq C_{16}\frak{c}^{\frac{1}{4}} (1 +\frak{c}^{\frac{3}{4}})\leq \f{\delta_0}2,
\end{equation}
if $\frak{c}$ is sufficiently small.

While due to smallness of the positive constant $\frak{c}$, it follows from \eqref{est-H2-rho-hins-1} that
\begin{align*}
\sup_{t\in [0,T_2^\star]}\|\nabla\rho(t)\|_{L^{\infty}\cap \dot{H}^1}
&\leq \|\nabla\rho_0\|_{L^{\infty}\cap \dot{H}^1}
   \, e^{C_{17} {\frak{c}}^{\f14}(1 + t)}\leq \f32\|\nabla\rho_0\|_{L^{\infty}\cap \dot{H}^1}
   \, e^{\frac{c_1}4 t}.
   \end{align*}

Along with \eqref{condition-smallness-sol-1}, \eqref{conditions-sol-3},  
\eqref{small-init-u0-2}, \eqref{est-partial2-omega-24-condition}-\eqref{est-partial2-omega-32}, \eqref{est-partial2-omega-33} and \eqref{u-H1-smallness-122}, we deduce that
 \beno
 T_2^\star=T_1^\star=T^\ast=\infty,
 \eeno
 for $T_1^\star, T_2^\star$ being defined respectively by \eqref{S4eq1} and \eqref{conditions-sol-5}.
 This finishes the proof of the global existence of a solution to  the system \eqref{eqns-HINS-1} satisfying the bounds \eqref{expo-Hins-u-neq-1}--\eqref{bdd-Hins-u-neq-4}. Uniqueness follows directly from the estimates \eqref{est-Lip-u-7},\eqref{est-H2-rho-hins-1} and classical argument for the system  \eqref{eqns-HINS-1}. This completes the proof of Theorem \ref{thm-GWP-HINS}.
\end{proof}

\medskip
\renewcommand{\theequation}{\thesection.\arabic{equation}}
\setcounter{equation}{0}

\appendix
\section{Preliminary}\label{sect2}

For the convenience of readers, we recall the following two technical lemmas  from \cite{HSWZ2024}:

\begin{lem}[Lemma 5.1 in \cite{HSWZ2024}]\label{lem-est-transport-1}
{\sl Fix $t > 0$. Assume that $\bar{u} = \bar{u}(s, x):  ]0, +\infty[ \times \mathbb{R}^d \rightarrow \mathbb{R}^d$ satisfies $s^{1/2} \nabla \bar{u} \in L^2(\mathbb{R}^+; L^{\infty}(\mathbb{R}^d))$. Let $p \in ]1, +\infty[$ and $\varphi(x) \in W^{1, p}(\mathbb{R}^d)$. Then there exists a unique solution $\psi = \psi(s, x):  ]0, t] \times \mathbb{R}^d \rightarrow \mathbb{R}^d$ to the equation
\begin{equation}\label{transport-psi-1}
\begin{cases}
\partial_s \psi + \bar{u} \cdot \nabla \psi = 0, & \, \, (s, x) \in ]0, t] \times \mathbb{R}^d,\\
\psi(t, x) = \varphi(x), & \
\end{cases}
\end{equation}
such that
\begin{equation}\label{est-transport-2}
\|\nabla \psi(s)\|_{L^p(\mathbb{R}^d)} \leq C \|\nabla \varphi\|_{L^p(\mathbb{R}^d)} \, e^{C |\ln(t/s)|^{\frac{1}{2}}} \quad \forall\, s \in ]0, t],
\end{equation}
where $C > 0$ is a constant depending only on $\|s^{1/2} \nabla \bar{u}\|_{L^2(\mathbb{R}^+; L^{\infty}(\mathbb{R}^d))}$.}
\end{lem}

\begin{lem}[Lemma 5.2 in \cite{HSWZ2024}]\label{lem-est-transport-2}
{\sl Let $T > 0$ and $f:  ]0, T] \rightarrow [0, +\infty]$ be such that $f \in L^4(0, T)$. Let $A > 0$ and define
\begin{equation}\label{def-F-1}
F(t) \eqdefa \frac{1}{t} \int_0^t f(s) \, e^{A |\ln(t/s)|^{\frac{1}{2}}} \, ds \quad \mbox{for}\, t \in ]0, T].
\end{equation}
Then $F \in L^4(0, T)$ and
\begin{equation}\label{est-F-2}
\|F\|_{L^4(0, T)} \leq C_A \|f\|_{L^4(0, T)},
\end{equation}
where $C_A > 0$ is a constant depending only on $A$.}
\end{lem}

\noindent {\bf Acknowledgments.}
 G. Gui is supported in part by National Natural Science Foundation of China under Grants 12371211 and 12126359.
  P. Zhang is partially  supported by National Key R$\&$D Program of China under grant 2021YFA1000800 and by National Natural Science Foundation of China under Grants  No. 12421001, No. 12494542 and No. 12288201.
  
  \vskip 0.2cm

\noindent{\bf Data availability statement.}  Data sharing is not applicable to this article as no datasets were generated or analysed
during the current study.

\noindent{\bf Conflict of interest.} On behalf of all authors, the corresponding author states that there is no conflict of
interest. All authors read and approved the final manuscript.

\end{document}